\documentclass[a4paper]{amsart}
\usepackage{amsmath,amsthm,verbatim,amssymb,amsfonts,amscd,mathrsfs,bm,scalerel,stackengine,hyperref,enumitem,listings,lmodern,nicematrix,booktabs,babel}
\usepackage[all]{xy}

\allowdisplaybreaks

\newtheorem{theorem}{Theorem}
\newtheorem{proposition}[theorem]{Proposition}

\newtheorem{corollary}[theorem]{Corollary}

\newtheorem{lemma}[theorem]{Lemma}
\newtheorem{maintheorem}{Theorem}

\theoremstyle{remark}
\newtheorem{remark}[theorem]{Remark}

\theoremstyle{definition}
\newtheorem{definition}[theorem]{Definition}
\newtheorem{example}[theorem]{Example}
\newtheorem*{notation}{Notation}

\stackMath
\newcommand\reallywidehat[1]{%
\savestack{\tmpbox}{\stretchto{%
  \scaleto{%
    \scalerel*[\widthof{\ensuremath{#1}}]{\kern.1pt\mathchar"0362\kern.1pt}%
    {\rule{0ex}{\textheight}}
  }{\textheight}%
}{2.4ex}}%
\stackon[-6.9pt]{#1}{\tmpbox}%
}

\DeclareMathOperator{\ppcm}{lcm}
\DeclareMathOperator{\pgcd}{gcd}

\DeclareMathOperator{\res}{Res}

\DeclareMathOperator{\rk}{rk}

\newcommand{\dd}{\mathrm{d}}
\newcommand{\phie}{\mathrm{e}}
\newcommand{\phii}{\mathrm{i}}
\newcommand{\glaisher}{\mathrm{A}}

\hypersetup{
pdftitle={Values and derivative values at nonpositive integers of generalized multiple Hurwitz zeta functions},
pdfauthor={Simon Rutard},
pdfkeywords={Multiple zeta functions, Special values, Witten zeta functions}
}

\title[{Special values of multiple Hurwitz zeta functions}]{Values and derivative values at nonpositive integers of generalized multiple Hurwitz zeta functions}
\author{Simon Rutard}
\thanks{\textit{Acknowledgements.} This work was supported in part by a JSPS Postdoctoral Fellowship (PE24744) at Nagoya University. This work also came to light thanks to Prof. Driss Essouabri, with whom I discussed many ideas of the paper. The author is grateful to the reviewer for their valuable suggestions.}
\date{\today}
\email{simon.rutard@math.nagoya-u.ac.jp}
\address{Graduate School of Mathematics, Nagoya University, Chikusa-ku, Nagoya 464-8602, Japan}
\keywords{Multiple zeta functions, Special values, Witten zeta functions}
\subjclass[2020]{Primary 11M32, Secondary 11M41}

\begin{document}

\begin{abstract}
We establish the meromorphic continuation of certain multiple zeta functions of generalized Hurwitz type. From this meromorphic continuation, we obtain explicit formulas for their (derivative) values at nonpositive integers along a given direction. As an application, we provide explicit formulas for some values and derivative values of the Witten zeta functions $\zeta_{\mathfrak{g}_2}$ and $\zeta_{\mathfrak{so}(5)}$. Furthermore, by employing a Meinardus-type theorem, we investigate the asymptotic behavior of the number of $n$-dimensional representations of the exceptional Lie algebra $\mathfrak{g}_2$.
\end{abstract}

\maketitle

\section{Introduction}

\subsection{Background}
Throughout this paper, we fix positive integers $r,m \in \mathbb{N}$, and two tuples $\mathbf{x}=(x_1,\ldots,x_r)$ and $\mathbf{c}=(c_{q,p})_{1 \leq q \leq m, 1 \leq p \leq r}$ such that $x_1,\ldots,x_r>0$ and $\Re(c_{1,1}),\ldots,\Re(c_{m,r})>0$. We consider the following Dirichlet series
\begin{equation}
    \zeta(\mathbf{c},\mathbf{x},\mathbf{s}) := \sum_{n_1, \ldots, n_r \geq 0} \prod_{i=1}^{r} (n_i+x_i)^{-s_i} \prod_{j=1}^{m} \left( \sum_{i=1}^r c_{j,i} (n_i+x_i) \right)^{-s_{r+j}}, \label{eq:def_multiple_zeta_function_hurwitz_type}
\end{equation}
which converges absolutely for $\Re(s_1),\ldots,\Re(s_{r+m})>1$. We view this as a multiple zeta function of generalized Hurwitz type. The analytic properties of (multiple) zeta functions have been studied by many, including H. Mellin, P. Cassou-Noguès, P. Sargos, B. Lichtin, M. Kaneko, T. Arakawa, S. Akiyama, Y. Tanigawa, S. Egami, K. Matsumoto, D. Essouabri, and Y. Komori. In \cite{komori2010integral}, Komori studied the analytic continuation and the singularities of multiple zeta functions of generalized Hurwitz-Lerch type, and proved that it admits a meromorphic continuation to the whole complex space, and that nonpositive integers are points of indeterminacy. In particular, the values at nonpositive integers of $\zeta(\mathbf{c},\mathbf{x},\mathbf{s})$ are not well-defined. One can still give meaning to values at nonpositive integers by taking successive limits of the variables at those points (see \cite{akiyama2001multiple}). One can also consider directional values (see (\ref{eq:def_directional_value})), which consist of approaching a tuple $-\mathbf{N}=(-N_1,\ldots,-N_{r+m}) \in \mathbb{Z}_{\leq 0}^{r+m}$ along a direction $\bm{\theta}=(\theta_1,\ldots,\theta_{r+m}) \in \mathbb{C}^{r+m}$. These directional values are well-defined because Komori proved that the single-variable function $z \mapsto \zeta(\mathbf{c},\mathbf{x},-\mathbf{N}+z\bm{\theta})$ is regular at $z=0$, assuming certain conditions in the direction $\bm{\theta}$ (see \cite[Theorem 3.22]{komori2010integral}). Komori also gave a non-explicit formula for directional values at nonpositive integers involving generalized Bernoulli numbers. Further work by Essouabri and Matsumoto established several explicit formulas for directional values of the Euler-Zagier type, and of the polynomial type (see \cite{essouabri20valueseulerzagier}, \cite{essouabri2021values}).

There are several motivations for studying multiple zeta functions as defined in (\ref{eq:def_multiple_zeta_function_hurwitz_type}). Firstly, the function $\zeta(\mathbf{c},\mathbf{x},\mathbf{s})$ as defined in (\ref{eq:def_multiple_zeta_function_hurwitz_type}) is a natural generalization of multiple Hurwitz zeta function of generalized Mordell–Tornheim type introduced by Onodera (see \cite{onodera2021multiple}), and several methods used to study its analytic properties extend naturally to the multiple zeta function $\zeta(\mathbf{c},\mathbf{x},\mathbf{s})$. Secondly, the multiple zeta functions associated with root systems of rank $2$ (see Definition \ref{def:multiple_zeta_root_system} and Example \ref{ex:multiple_zeta_rank_2}) can be written as $\zeta(\mathbf{c},\mathbf{x},\mathbf{s})$ with suitable data $\mathbf{c}$ and $\mathbf{x}$. Note that higher ranks are intractable in our work. Third, if we set $s_1=\cdots=s_r=0$ and $s_{r+1}=\cdots=s_{r+m}=s$, the zeta function $\zeta(\mathbf{c},\mathbf{x},(0,\ldots,0,s,\ldots,s))$ corresponds to a zeta function of Shintani type, which is closely linked to Stark's conjectures.

\begin{definition}
[\cite{komori2023root}]
\label{def:multiple_zeta_root_system}
Let $\Delta$ be a root system equipped with an inner product $\langle \cdot , \cdot \rangle$. We denote by $\Delta_+$ the set of positive roots, and let $\{ \lambda_1,\ldots,\lambda_r \}$ be its fundamental weights, where $r \in \mathbb{N}$ denotes the rank of $\Delta$. One can define a \textit{multiple zeta function associated with the root system} $\Delta$ as the Dirichlet series
$$ \zeta_{\Delta}(\mathbf{s}) := \sum_{n_1,\ldots,n_r \geq 1} \prod_{\alpha \in \Delta_+} \langle \alpha^{\vee}, n_1 \lambda_1 + \cdots + n_r \lambda_r \rangle^{-s_{\alpha}} \quad (\mathbf{s}=(s_{\alpha})_{\alpha \in \Delta_+} \in \mathbb{C}^{|\Delta_+|}, \ \Re(s_{\alpha})>1). $$
\end{definition}

It is known that the function $\zeta_{\Delta}(\mathbf{s})$ has a meromorphic continuation to $\mathbb{C}^{|\Delta_+|}$, and that it is a multiple-variable generalization of the Witten zeta functions defined by
$$ \zeta_{\mathfrak{g}}(s) := \sum_{\varphi} \frac{1}{\dim(\varphi)^s} $$
where $\varphi$ runs over all isomorphism classes of finite-dimensional irreducible representations of a semi-simple Lie algebra $\mathfrak{g}$. Indeed, if we denote by $\Delta(\mathfrak{g})$ the root system of $\mathfrak{g}$, and if we set $s_{\alpha}=s$ for all $\alpha \in \Delta_+(\mathfrak{g})$, by Weyl's dimension formula, we have $\zeta_{\mathfrak{g}}(s)=K(\mathfrak{g})^s \zeta_{\Delta(\mathfrak{g})}(\mathbf{s})$ where $K(\mathfrak{g}) \in \mathbb{N}$ is explicit. Note that its abscissa of convergence is $\sigma_0(\mathfrak{g})=\frac{\rk(\mathfrak{g})}{|\Delta_+(\mathfrak{g})|}$ (see \cite{larsen2008representation}), thus $s=\sigma_0(\mathfrak{g})$ is a singularity for $\zeta_{\mathfrak{g}}(s)$ by Landau's theorem.

\begin{example}
[\cite{komori2023root}]
\label{ex:multiple_zeta_rank_2}
For the root systems of rank $2$, namely $A_2,B_2$ and $G_2$, we have
\begin{align*}
    \zeta_{A_2}(s_1,s_2,s_3) =& \sum_{n_1,n_2 \geq 1} \frac{1}{n_1^{s_1} n_2^{s_2} (n_1+n_2)^{s_3}}, \\
    \zeta_{B_2}(s_1,s_2,s_3,s_4) =& \sum_{n_1,n_2 \geq 1} \frac{1}{n_1^{s_1} n_2^{s_2} (n_1+n_2)^{s_3}(n_1+2n_2)^{s_4}}, \\
    \zeta_{G_2}(s_1,\ldots,s_6) =& \sum_{n_1,n_2 \geq 1} \frac{1}{n_1^{s_1} n_2^{s_2} (n_1+n_2)^{s_3}(n_1+2n_2)^{s_4}(n_1+3n_2)^{s_5}(2n_1+3n_2)^{s_6}}.
\end{align*}
They correspond to the Witten zeta functions
$$ \zeta_{\mathfrak{sl}(3)}(s) = 2^s \zeta_{A_2}(s,s,s), \ \zeta_{\mathfrak{so}(5)}(s) = 6^s \zeta_{B_2}(s,s,s,s), \ \zeta_{\mathfrak{g}_2}(s) = 120^{s} \zeta_{G_2}(s,s,s,s,s,s). $$
\end{example}

Recently, the analytic properties of Witten zeta functions have been studied for their connection with asymptotic formulas for the number of isomorphism classes of $n$-dimensional representations of a semi-simple Lie algebra $\mathfrak{g}$, denoted $r_{\mathfrak{g}}(n)$. Such a sequence can be seen as a generalization of partition numbers, and in \cite{romik2017representations}, Romik proved an asymptotic formula for $r_{\mathfrak{sl}(3)}(n)$ by expanding the scope of the Meinardus theorem (see \cite{andrews1976partitions}) to a more general setting. To make the asymptotic formula explicit, he studied the analytic continuation of $\zeta_{\mathfrak{sl}(3)}$, and obtained information about its singularities, residues, and the values $\zeta_{\mathfrak{sl}(3)}(0)$ and $\zeta_{\mathfrak{sl}(3)}'(0)$. He later refined his asymptotic formula using results from Borwein and Dilcher (see \cite{borwein2018tornheim}), who provided a simpler formula than Romik's original expression for $\zeta_{\mathfrak{sl}(3)}'(0)$. Later on, Bridges, Brindle, Bringmann, and Franke proved a generalized Meinardus-type theorem to study the asymptotic behavior of the number of partitions generated by some infinite product (see \cite[Theorem 4.4]{bridges2024asymptotic}). As an application, they showed an asymptotic formula for $r_{\mathfrak{so}(5)}(n)$ (see \cite[Theorem 1.3]{bridges2024asymptotic}), although one constant was not computed in their paper due to the lack of explicit formula for $\zeta_{\mathfrak{so}(5)}'(0)$ in the literature.

\subsection{Notation}
Let $s \in \mathbb{C}$, we write $s = \sigma + \phii \tau$, with $\sigma, \tau \in \mathbb{R}$. For any $x \in \mathbb{R}$, we set $H_x := \{ s \in \mathbb{C}; \sigma > x \}$, and let $\overline{H}_x$ denote its closure in $\mathbb{C}$. We denote by $D_{a}(r)$ the open disk centered at $a \in \mathbb{C}$ with radius $r \geq 0$, and let $\overline{D}_a(r)$ denote its closure in $\mathbb{C}$. We write $[\![k,n]\!]:=\{k,\ldots,n\}$ for $k,n \in \mathbb{N}$. Throughout the paper, we assume that a product over the empty set is $1$, and a sum over the empty set is $0$.

For any finite sets $\mathcal{B} \subseteq \mathcal{A}$ and any complex tuple $\mathbf{z}=(z_a)_{a \in \mathcal{A}}$, we write $|\mathbf{z}|_{|\mathcal{B}} := \sum_{b \in \mathcal{B}} z_b$. If $\mathcal{B}=\mathcal{A}$, we simply write $|\mathbf{z}| := |\mathbf{z}|_{|\mathcal{A}}$. Note that if $\mathcal{A}=\emptyset$, then $\mathbf{z}$ corresponds to the empty tuple $\emptyset$, thus by convention $|\mathbf{z}|=0$.

For any finite sets $\mathcal{A}, \mathcal{B}$ and any integer tuple $\mathbf{u}=(u_{a,b})_{(a,b) \in \mathcal{A} \times \mathcal{B}}$, we set $\mathbf{u}_{\bullet b} := (u_{a,b})_{a \in \mathcal{A}}$ for all $b \in \mathcal{B}$, and $\mathbf{u}_{a \bullet} := (u_{a,b})_{b \in \mathcal{B}}$ for all $a \in \mathcal{A}$.

To simplify the notation of certain finite sums in the paper, we introduce a special notation. Let $m \in \mathbb{Z}_{\geq 1}$, $k \in \mathbb{Z}_{\geq 0}$, $j \neq \ell \in [\![ 1,m ]\!]$ be integers, and let $\mathbf{n}=(n_1,\ldots,n_m) \in \mathbb{Z}^{m}$ be a tuple. The sum $\displaystyle{\sum^{\widehat{j,l}}_{\substack{u_{1,1}+\cdots+u_{1,k}=n_1 \\ \cdots \\ u_{m,1}+\cdots+u_{m,k}=n_m}}}$ denotes the finite sum over all nonnegative integer tuples $\mathbf{u}=(u_{q,p})_{\substack{q \in [\![ 1,m ]\!] \setminus \{ j,\ell \} \\ p \in [\![ 1,k ]\!]}}$ such that $u_{q,1}+\cdots+u_{q,k}=n_q \ (1 \leq q \neq j,\ell \leq m)$. \\
Similarly, the sum $\displaystyle{\sum^{\widehat{j}}_{\substack{u_{1,1}+\cdots+u_{1,k}=n_1 \\ \cdots \\ u_{m,1}+\cdots+u_{m,k}=n_m}}}$ denotes the finite sum over all nonnegative integer tuples $\mathbf{u}=(u_{q,p})_{\substack{q \in [\![ 1,m ]\!] \setminus \{ j \} \\ p \in [\![ 1,k ]\!]}}$ such that $u_{q,1}+\cdots+u_{q,k}=n_q \ (1 \leq q \neq j \leq m)$.

\begin{remark}
(i) If $k=0$ or $m=1,2$ (resp. $k=0$ or $m=1$), then the only admissible $\mathbf{u}$ in the first sum (resp. the second sum) corresponds to the empty tuple $\emptyset$.

(ii) If there exists an integer $q \neq j,\ell$ (resp. $q \neq j$) such that $n_q<0$, then the first sum symbol (resp. the second sum symbol) is just $0$ by convention.
\end{remark}

We now introduce some terms and functions that will appear in the main results. We denote the $n$-th harmonic number by $h_n := \sum_{k=1}^n \frac{1}{k}$. We denote the binomial and multinomial coefficients by
\begin{align*}
    \binom{s}{n} :=& \frac{s(s-1)\cdots(s-n+1)}{n!} &&(s \in \mathbb{C}, n \in \mathbb{Z}_{\geq 0}) \\
    \binom{n}{\mathbf{k}} :=& \frac{n!}{k_1!\cdots k_i!} &&(n \in \mathbb{Z}_{\geq 0}, \mathbf{k}=(k_1,\ldots,k_i) \in \mathbb{Z}^i_{\geq 0} \text{ s.t. } \sum_{j=1}^i k_j=n).
\end{align*}
We assume that $h_0=0$, and if $i=0$ then $\mathbf{k}$ is the empty tuple $\emptyset$, and $\binom{0}{\emptyset} = 1$.

\begin{definition}
The \textit{Lerch zeta function} $\phi(z,s,x)$ is the meromorphic continuation in $s$ to the whole complex plane $\mathbb{C}$ of the Dirichlet series $\sum_{n=0}^{+\infty} \frac{z^n}{(n+x)^s}, \ (|z| \leq 1, \Re(s)>1, \Re(x)>0)$. We denote by $\zeta(s,x):=\phi(1,s,x)$ the \textit{Hurwitz zeta function}.
\end{definition}

\begin{definition}
Let $\mathbf{c}=(c_1,\ldots,c_r) \in H_0^r$, $\mathbf{x}=(x_1,\ldots,x_r) \in H_0^r$, and $\mathbf{M}=(M_1,\ldots,M_r) \in \mathbb{Z}_{\geq 0}^r$. The \textit{generalized Barnes zeta function} is defined as the meromorphic continuation to $\mathbb{C}$ of the Dirichlet series
\begin{equation}
    \zeta_B(s,\mathbf{M},\mathbf{x}|\mathbf{c}) := \sum_{n_1=0}^{+\infty} \cdots \sum_{n_r=0}^{+\infty} \frac{(n_1+x_1)^{M_1}\cdots(n_r+x_r)^{M_r}}{(c_1 (n_1+x_1) + \cdots + c_r(n_r+x_r))^s}. \label{eq:def_generalized_barnes}
\end{equation}
Its derivative with respect to $s$ is denoted by $\zeta_B'(s,\mathbf{M},\mathbf{x}|\mathbf{c}) := \partial_s \zeta_B(\mathbf{M},s,\mathbf{x}|\mathbf{c})$. By setting $\mathbf{M}=\mathbf{0}:=(0,\ldots,0)$, we get the classical Barnes zeta function
$$ \zeta_B(s,\mathbf{0},\mathbf{x}|\mathbf{c}) = \sum_{n_1,\ldots,n_{r} \geq 0} \frac{1}{(c_1 (n_1+x_1) + \cdots + c_r(n_r+x_r))^s}.$$
\end{definition}

\begin{definition}
\label{def:incomplete_gamma_function}
We denote the \textit{Euler Gamma function} by $\Gamma$, and the \textit{Euler-Mascheroni constant} by $\gamma:=-\Gamma'(1)$. We define the \textit{upper} (resp. \textit{lower}) \textit{incomplete gamma function} by
\begin{align*}
    \Gamma(s,t,z) :=& \int_{t}^{+\infty} \phie^{-z y} y^{s-1} \dd y &&(s \in \mathbb{C}, \ t>0, \ \Re(z)>0), \\
    \gamma(s,t,z) :=& \int_0^{t} \phie^{-z y} y^{s-1} \dd y &&(\Re(s)>0, \ t>0, \ \Re(z)>0).
\end{align*}
\end{definition}

\begin{remark}
If $z = 1$, we recover the classical incomplete gamma functions used in the literature (see \cite[Chap.IX]{bateman1953higher}). We see that, for all $t>0$, the functions $(s,z) \in H_0 \times H_0 \mapsto \gamma(s,t,z)$ and $(s,z) \in \mathbb{C} \times H_0 \mapsto \Gamma(s,t,z)$ are holomorphic, and we have the identity
\begin{equation}
    \Gamma(s)z^{-s} = \Gamma(s,t,z) + \gamma(s,t,z) \qquad (\Re(s)>0, \ t>0, \ \Re(z)>0). \label{eq:relation_incomplete_gamma_functions}
\end{equation}
\end{remark}

\subsection{Main results}

Throughout this paper, we fix a tuple $-\mathbf{N}=(-N_1,\ldots,-N_{r+m}) \in \mathbb{Z}_{\leq 0}^{r+m}$, and a direction $\bm{\theta}=(\theta_1,\ldots,\theta_{r+m}) \in \mathbb{C}^{r+m}$ satisfying the following non-vanishing condition
\begin{equation}
    \sum_{a \in \mathcal{A}}\theta_a + \sum_{q=1}^m \theta_{r+q} \neq 0 \qquad (\mathcal{A} \subseteq [\![ 1,r ]\!]).\label{eq:direction_non_vanishing_condition}
\end{equation}
We show a meromorphic continuation to the whole complex space of the function $\mathbf{s} \mapsto \zeta(\mathbf{c},\mathbf{x},\mathbf{s})$ in §\ref{section:analytic_continuation_zeta}. In §\ref{subsection:regularity_directional_zeta_0} we prove that the single-variable function $z \mapsto \zeta(\mathbf{c},\mathbf{x},-\mathbf{N}+z\bm{\theta})$ is regular at $z=0$, and provide an explicit formula for it. Hence we can define and study the directional values and the directional derivative values
\begin{align}
    \zeta(\mathbf{c},\mathbf{x},\underset{\bm{\theta}}{-\mathbf{N}}) &:= \lim_{z \to 0} \zeta(\mathbf{c},\mathbf{x},-\mathbf{N}+z\bm{\theta}), \label{eq:def_directional_value} \\
    \zeta'(\mathbf{c},\mathbf{x},\underset{\bm{\theta}}{-\mathbf{N}}) &:= \lim_{z \to 0} \partial_z (\zeta(\mathbf{c},\mathbf{x},-\mathbf{N}+z\bm{\theta})). \label{eq:def_directional_derivative_value}
\end{align}
From the meromorphic continuation of $\zeta(\mathbf{c},\mathbf{x},\mathbf{s})$, we obtain closed explicit formulas for both the directional values and the directional derivative values in §\ref{section:computation_zeta}. Both formulas involve explicit coefficients $C^0_{j,\mathcal{A},\mathbf{k}}(\mathbf{c},\mathbf{N})$ and $C^1_{j,\mathcal{A},\mathbf{k}}(\mathbf{c},\underset{\bm{\theta}}{\mathbf{N}})$ (see (\ref{eq:formula_C0}) and (\ref{eq:formula_C1})), and arise as coefficients of a Taylor expansion of the parametric integral $f_{j,\mathcal{A},\mathbf{k}}(\mathbf{s})$ (see (\ref{eq:def_parameter_integral})). For further details about these coefficients, see §\ref{section:computations_C}.
\begin{maintheorem}
\label{th:directional_values_mzf}
We have
\begin{align}
    \zeta( \mathbf{c},\mathbf{x},\underset{\bm{\theta}}{-\mathbf{N}}) = \sum_{\substack{\emptyset \neq \mathcal{A} \subseteq [\![1,r]\!] \\ 1 \leq j \leq m}}& \frac{(-1)^{|\mathbf{N}|_{|\mathcal{B}}+ \beta}\theta_{r+j}}{|\bm{\theta}|_{|\mathcal{B}\cup[\![r+1,r+m]\!]}} \prod_{p=1}^{\beta} N_{b_p}! \label{eq:zeta_N} \\
    & \times \sum_{k_1+\cdots+k_{\alpha} = \beta + |\mathbf{N}|_{|\mathcal{B}\cup [\![ r+1,r+m ]\!]}} C^0_{j,\mathcal{A},\mathbf{k}}(\mathbf{c},\mathbf{N}) \prod_{p=1}^{\alpha} \frac{\zeta(-N_{a_p}-k_p,x_{a_p})}{k_p!} \nonumber
\end{align}
where we write $\mathcal{A}=\{ a_1,\ldots,a_{\alpha} \} \subseteq [\![ 1,r ]\!]$, and $\mathcal{B}= \{ b_1,\ldots,b_{\beta} \}:=[\![1,r]\!] \setminus \mathcal{A}$.
\end{maintheorem}

\begin{maintheorem}
\label{th:directional_derivative_values_mzf}
We have
\begin{align}
    & \zeta'(\mathbf{c},\mathbf{x},\underset{\bm{\theta}}{-\mathbf{N}}) = \sum_{j=1}^{m} \theta_{r+j} \sum^{\widehat{j}}_{\substack{u_{1,1}+\cdots+u_{1,r}=N_{r+1} \\ \cdots \\ u_{m,1}+\cdots+u_{m,r}=N_{r+m}}} \prod_{\substack{ q=1 \\ q \neq j}}^{m} \left( \binom{N_{r+q}}{\mathbf{u}_{q\bullet}} \prod_{p=1}^{r} c_{q,p}^{u_{q,p}} \right) \label{eq:zeta_prime_N} \\
    & \times \Bigg[ \zeta_B'(-N_{r+j},\mathbf{M}(\mathbf{N},j,\mathbf{u}),\mathbf{x}|\mathbf{c}_{j\bullet})- (\gamma-h_{N_{r+j}}) N_{r+j}! \sum_{\emptyset \neq \mathcal{A} \subseteq [\![1,r]\!]} \prod_{p=1}^{\beta} \frac{(N_{b_p}+|\mathbf{u}_{\bullet b_p}|)!}{(-c_{j,b_p})^{N_{b_p}+|\mathbf{u}_{\bullet b_p}|+1}} \nonumber \\
    & \qquad \qquad \quad \times \sum_{k_1+\cdots+k_{\alpha}= N_{r+j}+\sum_{p=1}^{\beta} (N_{b_p}+|\mathbf{u}_{\bullet b_p}|+1)} \prod_{p=1}^{\alpha} \frac{c_{j,a_p}^{k_p}\zeta(- N_{a_p}-|\mathbf{u}_{\bullet a_p}|-k_p,x_{a_p})}{k_p!} \Bigg] \nonumber \\
    & + \sum_{\substack{\emptyset \neq \mathcal{A} \subseteq [\![1,r]\!] \\ 1 \leq j \leq m}} \frac{(-1)^{|\mathbf{N}|_{|\mathcal{B}}+ \beta}\theta_{r+j}}{|\bm{\theta}|_{|\mathcal{B}\cup[\![r+1,r+m]\!]}} \prod_{p=1}^{\beta} N_{b_p}! \sum_{\substack{k_1+\cdots+k_{\alpha} \\ = \beta + |\mathbf{N}|_{|\mathcal{B}\cup [\![ r+1,r+m ]\!]}}} \sum_{i=1}^{\alpha} \prod_{\substack{p=1 \\ p \neq i}}^{\alpha} \frac{\zeta(-N_{a_p}-k_p,x_{a_p})}{k_p!} \nonumber \\
    & \qquad \qquad \times \left( \frac{1}{\alpha} C^1_{j,\mathcal{A},\mathbf{k}}(\mathbf{c},\underset{\bm{\theta}}{\mathbf{N}}) \frac{\zeta(-N_{a_i}-k_i,x_{a_i})}{k_i!} + \theta_{a_i} C^0_{j,\mathcal{A},\mathbf{k}}(\mathbf{c},\mathbf{N}) \frac{\zeta'(-N_{a_i}-k_i,x_{a_i})}{k_i!} \right) \nonumber
\end{align}
where $\mathbf{M}(\mathbf{N},j,\mathbf{u})$ corresponds to the tuple 
\begin{equation}
    \mathbf{M}(\mathbf{N},j,\mathbf{u}) := (N_p+|\mathbf{u}_{\bullet p}|)_{p \in [\![1,r]\!]} \quad (1 \leq j \leq m, \  \mathbf{u}=(u_{q,p})_{\substack{q \in [\![1,m]\!] \setminus \{ j \} \\ p \in [\![ 1,r ]\!]}}). \label{eq:def_tuple_M}
\end{equation}
\end{maintheorem}

\begin{definition}
\label{def:C0_C1}
Let $\emptyset \neq \mathcal{A}= \{ a_1,\ldots,a_{\alpha} \}\subseteq [\![ 1,r ]\!]$, $\mathcal{B}= \{ b_1,\ldots,b_{\beta} \} := [\![ 1,r ]\!] \setminus \mathcal{A}$, $j \in [\![ 1,m ]\!]$, and $\mathbf{k}=(k_1,\ldots,k_{\alpha}) \in \mathbb{Z}_{\geq 0}^{\alpha}$. We set
\begin{align}
    C&^0_{j,\mathcal{A},\mathbf{k}}(\mathbf{c},\mathbf{N}) = \prod_{q=1}^{m} N_{r+q}! \sum_{\substack{v_{1,1}+\cdots+v_{m,1}=k_1 \\ \cdots \\ v_{1,\alpha}+\cdots+v_{m,\alpha} = k_{\alpha}}} \sum^{\widehat{j}}_{\substack{u_{1,1}+\cdots+u_{1,\beta}=N_{r+1}-|\mathbf{v}_{1 \bullet}| \\ \cdots \\ u_{m,1}+\cdots+u_{m,\beta}=N_{r+m}-|\mathbf{v}_{m \bullet}|}} \label{eq:formula_C0} \\
    & \prod_{p=1}^{\alpha} \left( \binom{k_p}{\mathbf{v}_{\bullet p}} \prod_{q=1}^{m} c_{q,a_p}^{v_{q,p}} \right) \prod_{p=1}^{\beta} \left( \binom{-N_{b_p}-1}{|\mathbf{u}_{\bullet p}|} \binom{|\mathbf{u}_{\bullet p}|}{\mathbf{u}_{\bullet p}} c_{j,b_p}^{-N_{b_p}-1-|\mathbf{u}_{\bullet p}|} \prod_{\substack{q=1 \\ q \neq j}}^{m} c_{q,b_p}^{u_{q,p}} \right) \nonumber \\
    C&_{j,\mathcal{A},\mathbf{k}}^1(\mathbf{c},\underset{\bm{\theta}}{\mathbf{N}}) = \prod_{q=1}^{m} N_{r+q}! \sum_{\substack{v_{1,1}+\cdots+v_{m,1}=k_1 \\ \cdots \\ v_{1,\alpha}+\cdots+v_{m,\alpha} = k_{\alpha}}} \sum^{\widehat{j}}_{\substack{u_{1,1}+\cdots+u_{1,\beta}=N_{r+1}-|\mathbf{v}_{1 \bullet}| \\ \cdots \\ u_{m,1}+\cdots+u_{m,\beta}=N_{r+m}-|\mathbf{v}_{m \bullet}|}} \label{eq:formula_C1} \\
    & \qquad \left( \sum_{p=1}^{\beta} \theta_{b_p} (\gamma-h_{N_{b_p}+|\mathbf{u}_{\bullet p}|}+\ln c_{j,{b_p}}) + \sum_{q=1}^{m} \theta_{r+q} (\gamma-h_{N_{r+q}}) \right) \nonumber \\
    & \qquad \times \prod_{p=1}^{\alpha} \left( \binom{k_p}{\mathbf{v}_{\bullet p}} \prod_{q=1}^{m} c_{q,a_p}^{v_{q,p}} \right) \prod_{p=1}^{\beta} \left( \binom{-N_{b_p}-1}{|\mathbf{u}_{\bullet p}|} \binom{|\mathbf{u}_{\bullet p}|}{\mathbf{u}_{\bullet p}} c_{j,b_p}^{-N_{b_p}-1-|\mathbf{u}_{\bullet p}|} \prod_{\substack{q=1 \\ q \neq j}}^{m} c_{q,b_p}^{u_{q,p}} \right) \nonumber \\
    & + \prod_{q=1}^{m} N_{r+q}! \sum_{\substack{\ell=1 \\ \ell \neq j}}^{m} \theta_{r+\ell} \sum_{\substack{v_{1,1}+\cdots+v_{m,1}=k_1 \\ \cdots \\ v_{1,\alpha}+\cdots+v_{m,\alpha} = k_{\alpha}}} \sum^{\widehat{j,\ell}}_{\substack{u_{1,1}+\cdots+u_{1,\beta}=N_{r+1}-|\mathbf{v}_{1 \bullet}| \\ \cdots \\ u_{m,1}+\cdots+u_{m,\beta}=N_{r+m}-|\mathbf{v}_{m \bullet}|}} \nonumber  \\
    & \qquad \prod_{p=1}^{\alpha} \left( \binom{k_p}{\mathbf{v}_{\bullet p}} \prod_{q=1}^m c_{q,a_p}^{v_{q,p}} \right) \prod_{p=1}^{\beta} \left( \binom{-N_{b_p}-1}{|\mathbf{u}_{\bullet p}|} \binom{|\mathbf{u}_{\bullet p}|}{\mathbf{u}_{\bullet p}} \prod_{\substack{q=1 \\ q \neq j,\ell}}^m c_{q,b_p}^{u_{q,p}} \right) W_{\mathcal{B},j,\ell,\mathbf{n}(\mathbf{N},\mathbf{u}),N_{\ell+r}-|\mathbf{v}_{\ell \bullet}|}(\mathbf{c}) \nonumber
\end{align}
where $\mathbf{n}(\mathbf{N},\mathbf{u}):=(N_{b_1}+|\mathbf{u}_{\bullet 1}|,\ldots,N_{b_\beta}+|\mathbf{u}_{\bullet \beta}|)$, and $W$ is the constant defined in (\ref{eq:def_coefficient_W}).
\end{definition}

An important consequence of Theorems \ref{th:directional_values_mzf} and \ref{th:directional_derivative_values_mzf} is that we have explicit formulas for the directional values and the directional derivative values of the multiple zeta functions associated with the root systems $B_2$ and $G_2$. In particular, we obtain a closed and explicit formula for the two Witten zeta functions $\zeta_{\mathfrak{so}(5)}$ and $\zeta_{\mathfrak{g}_2}$ at nonpositive integers, as well as for their derivatives. We provide formulas for $\zeta_{\mathfrak{so}(5)}$ and $\zeta_{\mathfrak{g}_2}$ at $s=0,-1,\ldots,-4$ and their first derivatives at $s=0$.

\begin{maintheorem}
We have
\begin{table}[htbp]
\centering
\begin{NiceTabular}{|c|c|c|c|c|c|}[colortbl-like]
\toprule
$N$ & $0$ & $-1$ & $-2$ & $-3$ & $-4$ \\
\midrule
$\zeta_{\mathfrak{so}(5)}(N)$ & $\frac{3}{8}$ & $-\frac{11}{26880}$ & $0$ & $-\frac{509}{37847040}$ & $0$ \\
$\zeta_{\mathfrak{g}_2}(N)$ & $\frac{5}{12}$ & $\frac{6641}{62705664}$ & $0$ & $\frac{12522872818983257}{109242202556140093440}$ & $0$ \\
\bottomrule
\end{NiceTabular}
\end{table}
\begin{align}
    \zeta_{\mathfrak{so}(5)}'(0) =& \frac{3\ln 3}{8} + \frac{13\ln 2}{8} + \frac{3\ln \pi}{2} \approx 3.255438605434552\ldots \label{eq:zeta_so5_derivative_0} \\
    \zeta_{\mathfrak{g}_2}'(0) =& \frac{5\ln 5}{12} - \frac{\ln 3}{12} + \frac{13\ln 2}{4} + \frac{5\ln \pi}{2} \approx 5.693601157568522\ldots \label{eq:zeta_g_2_derivative_0}
\end{align}
\end{maintheorem}
Both Witten zeta functions vanish at $s=-2$, which is a special case of a recent result (see \cite{au2024vanishingwitten}) that shows the vanishing at negative even integers of Witten zeta functions associated with any semi-simple Lie algebra. It is worth noting that these results were rediscovered in \cite{au2025valueswitten}, and Au went further to compute the residues at positive poles of Witten zeta functions of rank $2$.

Note that the asymptotic formula for the number of representations $r_{\mathfrak{so}(5)}(n)$ (see \cite[Theorem 1.3]{bridges2024asymptotic}) involves a constant depending on $\zeta_{\mathfrak{so}(5)}'(0)$, and (\ref{eq:zeta_so5_derivative_0}) allows this asymptotic formula to be fully explicit. Applying a Meinardus-type theorem (see \cite[Theorem 4.4]{bridges2024asymptotic}), we establish an asymptotic formula for the number of $n$-dimensional representations of the exceptional Lie algebra $\mathfrak{g}_2$, denoted $r_{\mathfrak{g}_2}(n)$. This formula depends on special values of $\zeta_{\mathfrak{g}_2}$, and on its residues. We point out that these residues are not computed in this paper, but have been computed by Au in \cite[Example 10.4]{au2025valueswitten}.

\begin{maintheorem}
\label{th:representation_g_2}
For $N \geq 1$, there exist explicit constants $\widetilde{B}_2,\ldots,\widetilde{B}_{N+1}$ such that
$$ r_{\mathfrak{g}_2}(n) \underset{n \to +\infty}{=} \frac{C}{n^{\frac{9}{16}}} \exp \left( A_1 n^{\frac{1}{4}}+ A_2 n^{\frac{3}{20}}+ A_3 n^{\frac{1}{20}} \right) \left( 1+ \sum_{j=2}^{N+1} \frac{\widetilde{B}_j}{n^{\frac{j-1}{20}}} + O \left( n^{-\frac{N+1}{20}} \right)\right), $$
where $A_1,A_2,A_3,C$ are defined in §\ref{subsection:proof_number_representation_g2}.
\end{maintheorem}

\subsection{Paper's layout}

In §\ref{section:analytic_continuation_zeta}, we prove the meromorphic continuation of $\zeta(\mathbf{c},\mathbf{x},\mathbf{s})$ using Crandall's expansion techniques, following ideas of Onodera, and Borwein and Dilcher (see \cite{onodera2021multiple}, \cite{borwein2018tornheim}). More specifically, in §\ref{subsection:crandall_expansion}, we start by providing an expression of $\zeta(\mathbf{c},\mathbf{x},\mathbf{s})$ in terms of a series of Eulerian-type integrals. Then we split each range of integral into two sets depending on a free real variable $t>0$. The first part, denoted $F(\mathbf{c},\mathbf{x},\mathbf{s},t)$, corresponds to a series of integrals over the same underlying domain $[0,t]^r$. The second part, denoted $H(\mathbf{c},\mathbf{x},\mathbf{s},t)$, corresponds to the same series of integrals, but over the domain $\mathbb{R}_{\geq 0}^r \setminus [0,t]^r$. We show that $H(\mathbf{c},\mathbf{x},\mathbf{s},t)$ is an entire function over the whole complex space (with respect to $\mathbf{s}$) that vanishes at nonpositive integers, and hence does not contribute for the values at nonpositive integers of $\zeta(\mathbf{c},\mathbf{x},\mathbf{s})$. Using Erdélyi's formula (see Proposition \ref{prop:Erdelyi_formula}) we provide an expression for $F(\mathbf{c},\mathbf{x},\mathbf{s},t)$ involving only the Hurwitz zeta function, the Gamma function, and parameter-dependent integrals denoted by $f_{j,\mathcal{A},\mathbf{k}}(\mathbf{c},\mathbf{s})$ (see (\ref{eq:def_parameter_integral})). In §\ref{subsection:param_dependent_integral} we study the analytic properties of $f_{j,\mathcal{A},\mathbf{k}}(\mathbf{c},\mathbf{s})$, and we obtain that it is an entire function with respect to $\mathbf{s}$. In §\ref{subsection:analytic_continuation_F}, we establish a meromorphic continuation for $F(\mathbf{c},\mathbf{x},\mathbf{s},t)$ with respect to $\mathbf{s}$, using bounds on the Hurwitz zeta functions due to Onodera. In particular, we obtain that it has a meromorphic continuation with singularities belonging to a union of hyperplanes $\mathcal{S}_{r,m}$ which contains nonpositive integer tuples, therefore proving a meromorphic continuation given by
$$ \zeta(\mathbf{c},\mathbf{x},\mathbf{s})=F(\mathbf{c},\mathbf{x},\mathbf{s},t)+H(\mathbf{c},\mathbf{x},\mathbf{s},t) \qquad (\mathbf{s} \in \mathbb{C}^{r+m} \setminus \mathcal{S}_{r,m}, \ t>0 \text{ small enough}). $$
We then apply this formula in §\ref{subsection:regularity_directional_zeta_0} upon proving that the single-variable function $z \mapsto \zeta(\mathbf{c},\mathbf{x},-\mathbf{N}+z\bm{\theta})$ is regular at $z=0$.

In §\ref{section:computation_zeta}, we prove Theorems \ref{th:directional_values_mzf} and \ref{th:directional_derivative_values_mzf} using the meromorphic continuation established for $\zeta(\mathbf{c},\mathbf{x},\mathbf{s})$. More specifically, in §\ref{subsection:proof_theorem_value} we compute the directional value 
$$ \zeta(\mathbf{c},\mathbf{x},\underset{\bm{\theta}}{-\mathbf{N}})=\lim_{z \to 0} F(\mathbf{c},\mathbf{x},-\mathbf{N}+z\bm{\theta},t)+\lim_{z \to 0} H(\mathbf{c},\mathbf{x},-\mathbf{N}+z\bm{\theta},t). $$
This computation follows from the vanishing of $H(\mathbf{c},\mathbf{x},\mathbf{s})$ at nonpositive integers, and from the analytic continuation obtained for $F(\mathbf{c},\mathbf{x},\mathbf{s},t)$ in §\ref{subsection:analytic_continuation_F}. The computation of the directional derivative value 
$$ \zeta'(\mathbf{c},\mathbf{x},\underset{\bm{\theta}}{-\mathbf{N}})=\lim_{z \to 0} \partial_z \left(F(\mathbf{c},\mathbf{x},-\mathbf{N}+z\bm{\theta},t)\right)+\lim_{z \to 0} \partial_z \left(H(\mathbf{c},\mathbf{x},-\mathbf{N}+z\bm{\theta},t)\right) $$
presents additional difficulties, because the second limit $\lim_{z \to 0} \partial_z \left(H(\mathbf{c},\mathbf{x},-\mathbf{N}+z\bm{\theta},t)\right)$ does not vanish in general. In §\ref{subsection:computation_F_prime}, we evaluate the first limit using the meromorphic continuation of $F(\mathbf{c},\mathbf{x},\mathbf{s},t)$ obtained in §\ref{subsection:analytic_continuation_F}. We note that the formula for this limit contains a sum of $\ln t$ and a Laurent series in the variable $t$, with an explicit constant term. In §\ref{subsection:computation_H_prime}, we establish a relation between the second limit $\lim_{z \to 0} \partial_z \left(H(\mathbf{c},\mathbf{x},-\mathbf{N}+z\bm{\theta},t)\right)$ and derivative values at nonpositive integers of some generalized Barnes zeta function. This last limit can also be expressed in terms of a sum of $\ln t$ and an explicit Laurent series in the variable $t$. By summing the two formulas obtained for these two limits, we prove Theorem \ref{th:directional_derivative_values_mzf} in §\ref{subsection:proof_theorem_derivative_value}.

In §\ref{section:computations_C}, we study the values and derivatives of the parametric integral $f_{j,\mathcal{A},\mathbf{k}}(\mathbf{c},\mathbf{s})$ that appear in the meromorphic continuation of $F(\mathbf{c},\mathbf{x},\mathbf{s},t)$.

In §\ref{section:values_hurwitz_barnes}, we recall and establish some results to simplify the formulas obtained for $\zeta'(\mathbf{c},\mathbf{x},\underset{\bm{\theta}}{-\mathbf{N}})$. More particularly, in §\ref{subsection:generalized_barnes_formula}, we prove an explicit expression between generalized Barnes formula associated with rational coefficients, and the Hurwitz zeta function. In §\ref{subsection:values_particular_hurwitz}, we recall known formulas for the special values of the Hurwitz zeta function $\zeta(s,x)$. In §\ref{subsection:values_zeta_witten_0}, we compute the values $\zeta_{\mathfrak{so}(5)}'(0)$ and $\zeta_{\mathfrak{g}_2}'(0)$ using the formulas obtained in §\ref{subsection:generalized_barnes_formula} and §\ref{subsection:values_particular_hurwitz}, and thanks to Theorem \ref{th:directional_derivative_values_mzf}.

In §\ref{section:asymptotic_formula}, we give an overview of a generalized Meinardus Theorem (see \cite{bridges2024asymptotic}). More specifically, we recall that theorem in §\ref{subsection:meinardus_type_theorem}, and then in §\ref{subsection:proof_number_representation_g2} we apply it along with the analytic properties we established for $\zeta_{\mathfrak{g}_2}$ to obtain an asymptotic formula for the number of representations $r_{\mathfrak{g}_2}(n)$.

In §\ref{section:annex}, we use a partial fraction decomposition to study an integral that appears in the computations of the directional derivative values of the entire function $f_{j,\mathcal{A},\mathbf{k}}(\mathbf{c},\mathbf{s})$.

\section{Meromorphic continuation of \texorpdfstring{$\zeta(\mathbf{c},\mathbf{x},\mathbf{s})$}{zeta}}
\label{section:analytic_continuation_zeta}

In this section, we define the linear forms $l_q(\mathbf{y}):=\sum_{p=1}^r c_{q,p}y_p \ (1 \leq q \leq m)$. Our goal is to prove an explicit meromorphic continuation of $\zeta(\mathbf{c},\mathbf{x},\mathbf{s})$ in terms of the Hurwitz zeta function, the Euler gamma function, and certain integrals, which will later be evaluated at special values.

\begin{theorem}
\label{th:analytic_continuation_zeta}
The function $\zeta(\mathbf{c},\mathbf{x},\mathbf{s})$ has an explicit meromorphic continuation with respect to $\mathbf{s}=(s_1,\ldots,s_{r+m})$ to the whole complex space given by
$$ \zeta(\mathbf{c},\mathbf{x},\mathbf{s})=H(\mathbf{c},\mathbf{x},\mathbf{s},t)+F(\mathbf{c},\mathbf{x},\mathbf{s},t) \qquad (t>0 \text{ small enough}), $$
where the function $F(\mathbf{c},\mathbf{x},\mathbf{s},t)$ (resp. $H(\mathbf{c},\mathbf{x},\mathbf{s},t)$) satisfies (\ref{eq:analytic_continuation_F_bis}) (resp. (\ref{eq:def_H_primaire})). Moreover, its singularities belong to the set $\mathcal{S}_{r,m} \subset \mathbb{C}^{r+m}$ defined as a union of all the hyperplanes
\begin{align*}
    \mathcal{H}^{(1)}_{p,n} :=& \{ \mathbf{s} \in \mathbb{C}^{r+m}; s_{p}=n\} && (1 \leq p \leq r, \ n \in \mathbb{N}), \\
    \mathcal{H}^{(2)}_{\mathcal{A},n} :=& \left\{ \mathbf{s} \in \mathbb{C}^{r+m};s_{r+1}+\cdots+s_{r+m} + \sum_{a \in \mathcal{A}} s_{a}=n\right\} && (\mathcal{A} \subseteq [\![ 1,r ]\!], \ n \in \mathbb{Z}_{\leq m}).
\end{align*}
Furthermore, if $\mathbf{s} \in \mathcal{S}_{r,m} \setminus \left( \cup_{\substack{1 \leq p \leq r \\ n \geq 1}} \mathcal{H}^{(1)}_{p,n}\right)$ is a singularity, then its multiplicity is $1$.
\end{theorem}

\subsection{Erdélyi's formula}

Erdélyi’s formula \cite[§1.11]{bateman1953higher}, stated in the following proposition, is crucial for rewriting integrands involving the Lerch zeta function, which appear in the next subsection.

\begin{proposition}[\cite{onodera2021multiple}]
\label{prop:Erdelyi_formula}
Let $u \in \mathbb{C} \setminus (-\infty,0]$. We have
$$\phie^{-xu} \phi(\phie^{-u},s,x) = \Gamma(1-s) u^{s-1} + \sum_{k=0}^{+\infty} \frac{(-u)^{k}}{k!} \zeta(s-k,x) \quad ( |u|<2\pi, s \in \mathbb{C}\setminus \mathbb{N}, x>0).$$
\end{proposition}

To prove that the above series is absolutely convergent, Onodera established a bound using a formula of Hurwitz. This bound will also be used later in the paper.

\begin{lemma}[\cite{onodera2021multiple}]
\label{lem:bound_hurwitz_zeta}
Let $\overline{D}(\delta,R) := \left\{ s \in \mathbb{C}; \min_{n \in \mathbb{N}} |s-n| \geq \delta, |s| \leq R \right\}$ for any $\delta \in \mathbb{R}_{>0}$, $R \in \mathbb{N}$. We then have
$$ |\partial_s^n \zeta(s-k,x)| \underset{R,\varepsilon,\delta}{\ll} k!(k+1)^{R+\varepsilon}(2\pi)^{-k} \qquad (s \in \overline{D}(\delta,r), k,n \geq 0, \varepsilon>0). $$
\end{lemma}

\subsection{Crandall's expansion}
\label{subsection:crandall_expansion}

Let $\mathbf{s}=(s_1,\ldots, s_{r+m}) \in \mathbb{C}^{r+m}$ such that $\Re(s_{1}), \ldots, \Re(s_{r+m}) > 1$. Since for all tuples $\mathbf{n} \in \mathbb{Z}_{\geq 0}^{r}$ we have
$\Gamma(s_{r+q})(l_q(\mathbf{n}+\mathbf{x}))^{-s_{r+q}}=\int_0^{+\infty} \phie^{-l_q(\mathbf{n}+\mathbf{x})y} y^{s_{r+q}-1} \dd y$, then we can write
\begin{align*}
    \zeta(\mathbf{c},\mathbf{x},\mathbf{s}) =& \sum_{n_1,\ldots,n_r \geq 0} \frac{1}{\prod_{p=1}^r (n_p+x_p)^{s_p}}\int_{0}^{+\infty} \cdots \int_{0}^{+\infty} \prod_{q=1}^m \frac{\phie^{-l_q(\mathbf{n}+\mathbf{x})y_q} y_q^{s_{r+q}-1}}{\prod_{q=1}^m \Gamma(s_{r+q})} \dd y_1 \cdots \dd y_m.
\end{align*}
Next we split the range of each integral such that $[0,+\infty) = [0,t] \cup (t,+\infty)$ for any given $t>0$. By using (\ref{eq:relation_incomplete_gamma_functions}), one can then prove that the general term of the Dirichlet series associated with $\zeta(\mathbf{c},\mathbf{x},\mathbf{s})$ satisfies the following identity
\begin{align}
    &\frac{1}{\prod_{p=1}^r (n_p+x_p)^{s_p} \prod_{q=1}^m l_q(\mathbf{n}+\mathbf{x})^{s_{r+q}}} \nonumber \\
    & \quad = \frac{1}{\prod_{p=1}^r (n_p+x_p)^{s_p}} \int_{0}^{t} \cdots \int_{0}^{t} \prod_{q=1}^m \frac{\phie^{-l_q(\mathbf{n}+\mathbf{x})y_q} y_q^{s_{r+q}-1}}{\Gamma(s_{r+q})} \dd y_1 \cdots \dd y_m \label{eq:general_term_dirichlet_series_hurwitz_zeta}\\
    & \qquad + \sum_{\substack{\emptyset \neq \mathcal{A} \subseteq [\![1,m]\!] \\ \mathcal{C} \subseteq \mathcal{A}^c}} \frac{(-1)^{|\mathcal{A}^c \setminus \mathcal{C}|}}{\prod_{q \in [\![1,m]\!] \setminus \mathcal{C}} \Gamma(s_{r+q})} \prod_{p=1}^{r} (n_p+x_p)^{-s_p} \nonumber \\
    & \qquad \qquad \times \prod_{q \in [\![ 1,m ]\!] \setminus \mathcal{C}} \Gamma(s_{r+q},t,l_q(\mathbf{n}+\mathbf{x})) \prod_{q \in \mathcal{C}} l_q(\mathbf{n}+\mathbf{x})^{-s_{r+q}}. \nonumber
\end{align}
On the right-hand side, we obtain two types of terms. The first one depends on an integral over $[0,t]^m$, and the second term is a finite sum of a non-empty product involving upper incomplete gamma functions and linear forms. We now consider the sum over $n_1,\ldots,n_r \geq 0$ for each of these terms, and study the nature of each corresponding series.

\begin{proposition}
\label{prop:normally_convergent_incomplete_gamma}
Let $t>0$, $\mathcal{C} \subsetneq [\![ 1,m ]\!]$, and $K \subset \mathbb{C}^{r+m}$ be a compact set. Then the series
\begin{equation}
    \sum_{n_1,\ldots,n_{r} \geq 0} \frac{\prod_{q \in [\![ 1,m ]\!] \setminus \mathcal{C}} \Gamma(s_{r+q},t,l_q(\mathbf{n}+\mathbf{x}))}{\prod_{p=1}^{r} (n_p+x_p)^{s_p} \prod_{q \in \mathcal{C}} l_q(\mathbf{n}+\mathbf{x})^{s_{r+q}}} \label{eq:normally_convergent_incomplete_gamma}
\end{equation}
is normally convergent on $K$ in the variables $(s_1,\ldots,s_{r+m})$.
\end{proposition}

\begin{proof}
We proceed by bounding the general term of the series above. Let $s \in \mathbb{C}$, $t>0$, $z \in H_0$. By a change of variables, we have
$\Gamma(s,t,z) = 2^s \int_{t/2}^{+\infty} \phie^{-2z y} y^{s-1} \dd y$. This yields the following bound for the upper incomplete gamma function
\begin{equation}
    |\Gamma(s,t,z)| \leq 2^{\Re(s)} \phie^{-\Re(z) t/2} \Gamma(\Re(s),t/2,\Re(z)). \label{eq:bound_incomplete_gamma_function}
\end{equation}
By the previous inequality, we find
\begin{align*}
    & \left| \prod_{p=1}^{r} (n_p+x_p)^{-s_p} \prod_{q \in [\![ 1,m ]\!] \setminus \mathcal{C}} \Gamma(s_{r+q},t,l_q(\mathbf{n}+\mathbf{x})) \prod_{q \in \mathcal{C}} l_q(\mathbf{n}+\mathbf{x})^{-s_{r+q}} \right| \\
    & \quad \leq \frac{\prod_{q \in [\![ 1,m ]\!] \setminus \mathcal{C}} \left( 2^{\sigma_{r+q}} \phie^{-\frac{\Re(l_q(\mathbf{n}+\mathbf{x}))t}{2}} \Gamma(\sigma_{r+q},t/2,\Re(l_q(\mathbf{n}+\mathbf{x})))\right)}{\prod_{p=1}^{r} (n_p+x_p)^{\sigma_p} \prod_{q \in \mathcal{C}} |l_q(\mathbf{n}+\mathbf{x})|^{\sigma_{r+q}}}.
\end{align*}
Note that $x \in \mathbb{R}_{>0} \mapsto \Gamma(\sigma,t/2,x)$ is increasing for any $\sigma>0, t>0$. Since $\Re(l_q(\mathbf{n}+\mathbf{x})) \geq c$ $(1 \leq q \leq m)$ where $c:=\min_{1 \leq q \leq m, 1 \leq p \leq r}(\Re(c_{q,p}x_p))>0$, we find 
$$\Gamma(\sigma_{r+q},t/2,\Re(l_q(\mathbf{n}+\mathbf{x}))) \leq \Gamma(\sigma_{r+q},t/2,c). $$ 
Since $K$ is compact, we have
$$ |2^{\sigma_{r+q}}| \ll_K 1, \qquad \Gamma(\sigma_{r+q},t/2,c) \ll_K 1. $$
Since the set $[\![ 1,m ]\!] \setminus \mathcal{C}$ is non-empty, and since $\Re(c_{q,p})>0$ $(1 \leq q \leq m, 1 \leq p \leq r)$, there exist $d_1,\ldots,d_r>0$ such that
\begin{align*}
    & \left| \prod_{p=1}^{r} (n_p+x_p)^{-s_p} \prod_{q \in [\![ 1,m ]\!] \setminus \mathcal{C}} \Gamma(s_{r+q},t,l_q(\mathbf{n}+\mathbf{x})) \prod_{q \in \mathcal{C}} l_q(\mathbf{n}+\mathbf{x})^{-s_{r+q}} \right| \\
    & \quad \ll_K \frac{\prod_{q \in [\![ 1,m ]\!] \setminus \mathcal{C}} \phie^{-\frac{\Re(l_q(\mathbf{n}+\mathbf{x}))t}{2}}}{\prod_{p=1}^{r} (n_p+x_p)^{\sigma_p} \prod_{q \in \mathcal{C}} |l_q(\mathbf{n}+\mathbf{x})|^{\sigma_{r+q}}}=O(\phie^{-d_1(n_1+x_1)-\cdots-d_r(n_r+x_r)}).
\end{align*}
Therefore, (\ref{eq:normally_convergent_incomplete_gamma}) is normally convergent in the variables $(s_1,\ldots,s_{r+m})$ on $K$.
\end{proof}

\begin{remark}
Note that the series on the right-hand side of (\ref{eq:normally_convergent_incomplete_gamma}) may diverge if we no longer assume $\Re(c_{j,i})>0$ for all $1 \leq i \leq r, \ 1 \leq j \leq m$.
\end{remark}

\begin{proposition}
\label{prop:normally_convergent_H1}
The following series is normally convergent in the variables $s_1,\ldots,s_{r+m}$ on any given compact set $K \subset H_1^{r+m}$,
$$ \sum_{n_1,\ldots,n_r \geq 0} \frac{1}{\prod_{p=1}^r (n_p+x_p)^{s_p}} \int_{0}^{t} \cdots \int_{0}^{t} \prod_{q=1}^m \left( \phie^{-l_q(\mathbf{n}+\mathbf{x})y_q} y_q^{s_{r+q}-1} \right) \dd y_1 \cdots \dd y_m. $$
\end{proposition}

\begin{proof}
It follows from bounding the exponential term in the integrand by $1$.
\end{proof}

From the previous Proposition, we can perform a series-integral swap
\begin{align}
    &\sum_{n_1,\ldots,n_r \geq 0} \frac{1}{\prod_{p=1}^r (n_p+x_p)^{s_p}} \int_{0}^{t} \cdots \int_{0}^{t} \prod_{q=1}^m \left(\phie^{-l_q(\mathbf{n}+\mathbf{x})y_q} y_q^{s_{r+q}-1}\right) \dd y_1 \cdots \dd y_m \nonumber \\
    & \quad = \int_{0}^{t} \cdots \int_{0}^{t} \prod_{p=1}^r \left( \phie^{-l_p^*(\mathbf{y})x_p} \phi(\phie^{-l_p^*(\mathbf{y})},s_{p},x_p) \right) \prod_{q=1}^{m}y_q^{s_{r+q}-1} \dd y_1 \cdots \dd y_m, \label{eq:series_integral_swap}
\end{align}
where $\Re(s_1),\ldots,\Re(s_{r+m})>1$ and
\begin{equation}
    l_p^*(\mathbf{y}):=c_{1,p}y_1+\cdots+c_{m,p}y_p \qquad (1 \leq p \leq r). \label{eq:def_l_star}
\end{equation}
By (\ref{eq:general_term_dirichlet_series_hurwitz_zeta}), by (\ref{eq:series_integral_swap}), by Proposition \ref{prop:normally_convergent_H1}, and by Proposition \ref{prop:normally_convergent_incomplete_gamma}, we find
\begin{equation}   
    \zeta(\mathbf{c},\mathbf{x},\mathbf{s})= F(\mathbf{c},\mathbf{x},\mathbf{s},t)+H(\mathbf{c},\mathbf{x},\mathbf{s},t) \qquad (\Re(s_1),\ldots,\Re(s_{r+m}) > 1, \ t>0), \label{eq:crandall_expansion_zeta}
\end{equation}
where
\begin{align}
    F(\mathbf{c},\mathbf{x},\mathbf{s},t) :=& \int_{0}^{t} \cdots \int_{0}^{t} \prod_{p=1}^r \left( \phie^{-l_p^*(\mathbf{y})x_p} \phi(\phie^{-l_p^*(\mathbf{y})},s_{p},x_p) \right) \prod_{q=1}^{m} \frac{y_q^{s_{r+q}-1}}{\Gamma(s_{r+q})} \dd y_1 \cdots \dd y_m \label{eq:def_F_primaire}\\
    H(\mathbf{c},\mathbf{x},\mathbf{s},t) :=& \sum_{\substack{\emptyset \neq \mathcal{A} \subseteq [\![1,m]\!] \\ \mathcal{C} \subseteq \mathcal{A}^c}} \frac{(-1)^{|\mathcal{A}^c \setminus \mathcal{C}|}}{\prod_{q \in [\![1,m]\!] \setminus \mathcal{C}} \Gamma(s_{r+q})} \label{eq:def_H_primaire} \\
    & \qquad \sum_{n_1,\ldots,n_r \geq 0} \frac{\prod_{q \in [\![ 1,m ]\!] \setminus \mathcal{C}} \Gamma(s_{r+q},t,l_q(\mathbf{n}+\mathbf{x}))}{\prod_{p=1}^{r} (n_p+x_p)^{s_p} \prod_{q \in \mathcal{C}} l_q(\mathbf{n}+\mathbf{x})^{s_{r+q}}}. \nonumber
\end{align}
By Proposition \ref{prop:normally_convergent_incomplete_gamma}, we get that:

\begin{corollary}
\label{cor:H_entire}
Let $t>0$, then $\mathbf{s} \mapsto H(\mathbf{c},\mathbf{x},\mathbf{s},t)$ is entire and vanishes at all nonpositive integers $-\mathbf{N} \in \mathbb{Z}_{\leq 0}^{r+m}$.
\end{corollary}

By the previous corollary, $H(\mathbf{c},\mathbf{x},\mathbf{s},t)$ does not contribute to the directional values of $\zeta(\mathbf{c},\mathbf{x},\mathbf{s})$ at nonpositive integers. Thus, only $F(\mathbf{c},\mathbf{x},\mathbf{s},t)$ determines those values, which we ultimately compute using Erdélyi's formula (see Proposition \ref{prop:Erdelyi_formula}).

\begin{proposition}
\label{prop:analytic_continuation_F}
Let $\mathbf{s}=(s_1,\ldots,s_{r+m}) \in (\mathbb{C} \setminus \mathbb{N})^{r+m}$ such that $\Re(s_1),\ldots,\Re(s_{r+m})>1$. There exists $t_0>0$ such that, for all $t \in (0,t_0)$, we have
\begin{align}
    F(\mathbf{c},\mathbf{x},\mathbf{s},t) =& \sum_{\substack{\mathcal{A} \subseteq [\![1,r]\!] \\ k_1,\ldots,k_{\alpha}\geq 0}} (-1)^{|\mathbf{k}|} \prod_{p=1}^{\beta} \Gamma(1-s_{b_p}) \prod_{p=1}^{\alpha} \frac{\zeta(s_{a_p}-k_p,x_{a_p})}{k_p!} \label{eq:analytic_continuation_F} \\
    & \times \int_0^t \cdots \int_0^t \prod_{q=1}^{m} \frac{y_q^{s_{r+q}-1}}{\Gamma(s_{r+q})}\prod_{p=1}^{\beta} l^*_{b_p}(\mathbf{y})^{s_{b_p}-1} \prod_{p=1}^{\alpha} l^*_{a_p}(\mathbf{y})^{k_p} \dd y_1 \cdots \dd y_{m}. \nonumber
\end{align}
Recall that we write $\mathcal{A}=\{ a_1,\ldots,a_{\alpha} \} \subseteq [\![ 1,r ]\!]$, $\mathcal{B}=\{ b_1,\ldots,b_{\beta} \} := [\![ 1,r ]\!] \setminus \mathcal{A}$.
\end{proposition}

Note that $\mathcal{A}$ or $\mathcal{B}$ can be empty, thus $\alpha$ or $\beta$ can be $0$. In such a case, we understand that the product over an empty set is $1$. Also, if $\alpha=0$, then the tuple $\mathbf{k}=(k_1,\ldots,k_{\alpha})$ corresponds to the empty tuple $\emptyset$, thus $|\mathbf{k}|=0$.

\begin{proof}
Let 
\begin{equation}
    \overline{D}_{r,m}(\delta,R) := \left\{ \mathbf{s}\in {\overline{D}_0}(R)^{r+m}; \min_{n \in \mathbb{Z}}|s_{p}-n| \geq \delta \ (1 \leq p \leq r)\right\} \label{eq:definition_D_deltaR} 
\end{equation}
where $0 < \delta < 1$ and $R>0$ are both arbitrary. Using Erdélyi's formula, we will prove that (\ref{eq:analytic_continuation_F}) holds for all $\mathbf{s} \in \overline{D}_{r,m}(\delta,R)$ such that $\Re(s_1),\ldots,\Re(s_{r+m})>1$.

Let $\mathbf{s} \in \overline{D}_{r,m}(\delta,R)$ such that $\Re(s_1),\ldots,\Re(s_{r+m})>1$. In particular, we have $s_p \in \mathbb{C} \setminus \mathbb{N}$ $(1 \leq p \leq r)$. Note also that 
\begin{equation}
    \forall \mathbf{y}=(y_1,\ldots,y_m) \in [0,t]^{m}, \ |l_p^*(\mathbf{y})| \ll t \qquad (1 \leq p \leq r). \label{eq:bound_l*_p_theta}
\end{equation}
By Erdélyi's formula we obtain the following for $t>0$ sufficiently small
\begin{equation}
    \phie^{-l_p^*(\mathbf{y})x_p} \phi(\phie^{-l_p^*(\mathbf{y})},s_p,x_p) = \Gamma(1-s_p) l_p^*(\mathbf{y})^{s_p-1} + \sum_{k_p=0}^{+\infty} \frac{(-l_p^*(\mathbf{y}))^{k_p}}{k_p!} \zeta(s_p-k_p,x_p), \label{eq:erdelyi_formula_p}
\end{equation}
with $1 \leq p \leq r$ and $\mathbf{y} \in [0,t]^{m}$. We can then substitute the product $\prod_{p=1}^r \left( \phie^{-l_p^*(\mathbf{y})x_p} \phi(\phie^{-l_p^*(\mathbf{y})},s_p,x_p) \right)$ in the integrand of $F(\mathbf{c},\mathbf{x},\mathbf{s},t)$ using the right-hand side of (\ref{eq:erdelyi_formula_p}), and we will then perform another series-integral inversion. In order to justify such inversion, we first need to bound each term on the right-hand side of (\ref{eq:erdelyi_formula_p}) for all $\mathbf{s} \in \overline{D}_{r,m}(\delta,R)$. 

We first recall the following standard results:
\begin{enumerate}
	\item[i)] The function $\Gamma$ is meromorphic on $\mathbb{C}$ with simple poles at nonpositive integers, and the function $\frac{1}{\Gamma}$ is entire.
	\item[ii)] The Hurwitz zeta function is meromorphic on $\mathbb{C}$ with a simple pole at $s=1$.
	\item[iii)] As $s_p \mapsto \Gamma(1-s_p)$ is holomorphic on the compact set $\overline{D}_{r,m}(\delta,R)$, it follows that
\begin{equation}
    |\Gamma(1-s_p)| \underset{R,\delta}{\ll} 1 \qquad (\mathbf{s}=(s_1,\ldots,s_{r+m}) \in \overline{D}_{r,m}(\delta,R)). \label{eq:bound_gamma}
\end{equation}
	\item[iv)] Let $\varepsilon>0$. By Lemma \ref{lem:bound_hurwitz_zeta} we have
\begin{equation}
    \left|\frac{\zeta(s_p-k_p,x_p)}{k_p!}\right| \underset{R,\varepsilon,\delta}{\ll} (2\pi)^{-k_p}(k_p+1)^{R+\varepsilon} \quad (\mathbf{s} \in \overline{D}_{r,m}(\delta,R),k_p \in \mathbb{Z}_{\geq 0}). \label{eq:bound_zeta}
\end{equation}
	Moreover, the series $\sum_{k_p=0}^{+\infty} (k_p+1)^{R+\varepsilon}\left(\frac{t}{2\pi}\right)^{k_p}$ converges for all $t>0$ sufficiently small.
\end{enumerate}

Using iv) together with bounds (\ref{eq:bound_l*_p_theta}) and (\ref{eq:bound_gamma}), we can perform a series-integral inversion
\begin{align}
    \int_{0}^{t}& \cdots \int_{0}^{t} \prod_{p=1}^r \left(\Gamma(1-s_p) l_p^*(\mathbf{y})^{s_p-1} + \sum_{k_p=0}^{+\infty} \frac{(-l_p^*(\mathbf{y}))^{k_p}}{k_p!} \zeta(s_p-k_p,x_p) \right) \prod_{q=1}^{m}\left( y_q^{s_{r+q}-1} \dd y_q \right) \nonumber \\
    &= \sum_{\substack{\mathcal{A} \subseteq [\![1,r]\!] \\ k_1,\ldots,k_{\alpha}\geq 0}} (-1)^{|\mathbf{k}|} \prod_{p=1}^{\beta} \Gamma(1-s_{b_p}) \prod_{p=1}^{\alpha} \frac{\zeta(s_{a_p}-k_p,x_{a_p})}{k_p!} \label{eq:analytic_continuation_F_proof_s1}\\
    & \qquad \qquad \times \int_0^t \cdots \int_0^t \prod_{q=1}^{m} y_q^{s_{r+q}-1}\prod_{p=1}^{\beta} l^*_{b_p}(\mathbf{y})^{s_{b_p}-1} \prod_{p=1}^{\alpha} l^*_{a_p}(\mathbf{y})^{k_p} \dd y_1 \cdots \dd y_{m}. \nonumber
\end{align}
Dividing both sides of (\ref{eq:analytic_continuation_F_proof_s1}) by $\prod_{q=1}^{m}\Gamma(s_{r+q})$, we obtain (\ref{eq:analytic_continuation_F}).
\end{proof}

We now want to extend (\ref{eq:analytic_continuation_F}) to the whole complex space. One major obstruction is that the integral on the right-hand side of (\ref{eq:analytic_continuation_F}) is not convergent when we no longer assume $\Re(s_1),\ldots,\Re(s_{r+m}) > 1$, due to a singularity at the origin of the range of the integral. For this reason, we shall perform a suitable blow-up in order to obtain an analytic continuation of this parameter-dependent integral.

\subsection{Analytic continuation of a parameter-dependent integral}
\label{subsection:param_dependent_integral}

In this subsection, we fix $\mathcal{A}= \{ a_1,\ldots,a_{\alpha} \} \subseteq [\![ 1,r ]\!]$, $\mathcal{B}=\{ b_1,\ldots,b_{\beta} \} := [\![ 1,r ]\!] \setminus \mathcal{A}$, and a tuple $\mathbf{k}=(k_1,\ldots,k_{\alpha}) \in \mathbb{Z}_{\geq 0}^{\alpha}$. We set 
$$ V_j(t):=\big\{ \mathbf{y} \in [0,t]^{m};y_j \geq y_q \ (1 \leq q \neq j \leq m) \big\} \qquad (1 \leq j \leq m). $$
The range of the integral on the right-hand side of (\ref{eq:analytic_continuation_F}) can be decomposed as $[0,t]^m = \bigcup_{j=1}^m V_j(t)$ with pairwise intersections of measure zero. Let $\mathbf{s}=(s_1,\ldots,s_{r+m}) \in \mathbb{C}^{r+m}$ such that $\Re(s_1),\ldots,\Re(s_{r+m})>1$. By the additivity property of the integral, and using the change of variables
\[
    \varphi_j :
    \begin{cases}
        [0,1]^{j-1} \times [0,t] \times [0,1]^{m-j} \to V_j(t) \\
        \mathbf{y}=(y_1,\ldots,y_{m}) \mapsto (y_1y_j,\ldots,y_j,\ldots,y_{m}y_j)
    \end{cases}
    \qquad (1 \leq j \leq m)
\]
we obtain
\begin{align}
    \int_0^t \cdots \int_0^t \prod_{q=1}^{m} \frac{y_q^{s_{r+q}-1}}{\Gamma(s_{r+q})}& \prod_{p=1}^{\beta} l^*_{b_p}(\mathbf{y})^{s_{b_p}-1} \prod_{p=1}^{\alpha} l^*_{a_p}(\mathbf{y})^{k_p} \dd y_1 \cdots \dd y_{m} \nonumber \\
    =& \sum_{j=1}^{m} \frac{t^{|\mathbf{k}|-\beta+|\mathbf{s}|_{|\mathcal{B}\cup [\![ r+1,r+m ]\!]}}}{\Gamma(s_{r+j})(|\mathbf{k}|-\beta+|\mathbf{s}|_{|\mathcal{B}\cup [\![ r+1,r+m ]\!]})} f_{j,\mathcal{A},\mathbf{k}}(\mathbf{c},\mathbf{s}) \label{eq:relation_integral_after_blowup}
\end{align}
where $l^*_1,\ldots,l^*_r$ are defined in (\ref{eq:def_l_star}), and where we have set
\begin{equation}
    f_{j,\mathcal{A},\mathbf{k}}(\mathbf{c},\mathbf{s}) := \int_{0}^1 \cdots \int_{0}^1 \prod_{\substack{q=1 \\ q \neq j}}^{m} \frac{y_q^{s_{r+q}-1}}{\Gamma(s_{r+q})} \prod_{p=1}^{\beta} l^*_{b_p}(\widehat{\mathbf{y}}^j)^{s_{b_p}-1} \prod_{p=1}^{\alpha} l^*_{a_p}(\widehat{\mathbf{y}}^j)^{k_p} \dd y_1 \cdots \widehat{\dd y_j} \cdots  \dd y_{m} \label{eq:def_parameter_integral}
\end{equation}
with $\dd y_1 \cdots \widehat{\dd y_{j}}\cdots \dd y_{m} := \dd y_{1} \cdots \dd y_{j-1} \dd y_{j+1} \cdots \dd y_{m}$, and $\widehat{\mathbf{y}}^j := (y_1,\ldots,y_{j-1},1,y_{j+1},\ldots,y_{m})$. It is clear that the function $\mathbf{s} \mapsto f_{j,\mathcal{A},\mathbf{k}}(\mathbf{c},\mathbf{s})$ is holomorphic over $H_1^{r+m}$.

In the next proposition, we establish a holomorphic continuation for $f_{j,\mathcal{A},\mathbf{k}}(\mathbf{c},\mathbf{s})$ $(\mathbf{s} \in \mathbb{C}^{r+m})$ involving a normally convergent series with respect to the variables $s_1,\ldots,s_{r+m}$. We also provide a bound for this function, which will allow us to prove that the right-hand side of (\ref{eq:analytic_continuation_F}) converges for all tuples $\mathbf{s} \in \mathbb{C}^{r+m}$, except for some tuples $\mathbf{s}$ contained in a union of hyperplanes.

\begin{proposition}
\label{prop:expression_f_Ajepsilon_s}
The function $\mathbf{s} \mapsto f_{j,\mathcal{A},\mathbf{k}}(\mathbf{c},\mathbf{s})$ has a holomorphic continuation on $\mathbb{C}^{r+m}$ given by
\begin{align}
    f_{j,\mathcal{A},\mathbf{k}}(&\mathbf{c},\mathbf{s}) = \sum_{\mathcal{A}' \subseteq [\![ 1,m ]\!] \setminus \{ j \}} \sum_{\substack{ u_{1,1},\ldots,u_{\alpha',\beta} \geq 0 \\ v_{1,1}+\cdots+v_{m,1}=k_1 \\ \cdots \\ v_{1,\alpha}+\cdots+v_{m,\alpha}=k_{\alpha}}} \frac{\varepsilon^{ \sum_{q=1}^{\alpha'} \left(s_{r+a'_q}+|\mathbf{u}_{q \bullet}| + |\mathbf{v}_{q \bullet}|\right)}}{\prod_{q=1}^{\alpha'} \left( \Gamma(s_{r+a'_q})(s_{r+a'_q}+ |\mathbf{u}_{q \bullet}| + |\mathbf{v}_{q \bullet}|)\right)} \nonumber \\
    & \times \prod_{p=1}^{\alpha} \left(\binom{k_p}{\mathbf{v}_{\bullet p}} \prod_{q=1}^m c_{q,a_p}^{v_{q,p}} \right) \prod_{p=1}^{\beta} \left( \binom{s_{b_p}-1}{|\mathbf{u}_{\bullet p}|} \binom{|\mathbf{u}_{\bullet p}|}{\mathbf{u}_{\bullet p}} \prod_{q=1}^{\alpha'} c_{a'_q,b_p}^{u_{q,p}} \right) \label{eq:expression_f_Ajepsilon_s} \\
    & \times \int_{\varepsilon}^{1} \cdots \int_{\varepsilon}^1 \prod_{q=1}^{\beta'} \frac{y_q^{s_{r+b'_q}-1+|\mathbf{v}_{q \bullet}|}}{\Gamma(s_{r+b'_q})} \prod_{p=1}^{\beta} \left( c_{j,b_p}+\sum_{q=1}^{\beta'} c_{b'_q,b_p} y_{q} \right)^{s_{a_p}-1-|\mathbf{u}_{\bullet p}|} \dd y_1 \cdots \dd y_{\beta'} \nonumber
\end{align}
where $\varepsilon \in (0,1)$ is sufficiently small, and where we write $\mathcal{A}'=(a'_1,\ldots,a'_{\alpha'}) \subseteq [\![ 1,m ]\!] \setminus \{ j \}$ and $\mathcal{B}'=(b'_1,\ldots,b'_{\alpha'}) := [\![ 1,m ]\!] \setminus ( \mathcal{A}' \cup \{ j \})$. Moreover, for any $R>0$, we have
\begin{equation}
    |f_{j,\mathcal{A},\mathbf{k}}(\mathbf{c},\mathbf{s})| \underset{R,\varepsilon}{\ll} (m \max_{1 \leq p \leq r, 1 \leq q \leq m}|c_{q,p}|)^{|\mathbf{k}|} \qquad (|s_1|,\ldots,|s_{r+m}| \leq R). \label{eq:bound_f_Ajk}
\end{equation}
\end{proposition}

Before proving the proposition, we need the following elementary uniform bound.

\begin{lemma}
\label{lem:bound_1_over_gamma}
For all $R>0$, we have
$$ \left| \frac{1}{\Gamma(s)} \right| \underset{R}{\ll} 1, \qquad \left| \frac{1}{\Gamma(s)(s+n)} \right| \underset{R}{\ll} 1 \qquad (n \in \mathbb{Z}_{\geq 0}, \ s \in \overline{D}_0(R)). $$
\end{lemma}

\begin{proof}[Proof of Proposition \ref{prop:expression_f_Ajepsilon_s}]
Firstly, we prove that the series on the right-hand side of (\ref{eq:expression_f_Ajepsilon_s}) is normally convergent on any given compact subset of $\mathbb{C}^{r+m}$ (in the variables $\mathbf{s}$) by establishing bounds for each term of this series. Secondly, we show that (\ref{eq:expression_f_Ajepsilon_s}) holds for all $s_1,\ldots,s_{r+m} \in \mathbb{C}$ such that $\Re(s_1),\ldots,\Re(s_{r+m})>1$, which will conclude the proof.

i) We begin by bounding the general term, denoted $T_{\mathcal{A}',\mathbf{u},\mathbf{v}}(\mathbf{s})$, of the series on the right-hand side of (\ref{eq:expression_f_Ajepsilon_s}), on the compact set $K:=\{ \mathbf{s} \in \mathbb{C}^{r+m};\forall k \in [\![ r+m ]\!], s_k \in \overline{D}_0(R) \}$ with $R \in \mathbb{N}$. Let $u_{1,1},\ldots,u_{\alpha',\beta} \geq 0$ and $v_{1,1},\ldots,v_{m,\alpha} \geq 0$ be nonnegative integers such that $v_{p,1}+\cdots+v_{p,m}=k_p$ $(1 \leq p \leq \alpha)$. For all $\mathbf{s} \in K$, we have:
\begin{enumerate}
\item[$\bullet$] From Lemma \ref{lem:bound_1_over_gamma}, we obtain
\begin{align}
    & \left|\frac{1}{\Gamma(s_{r+q})(s_{r+q}+n)}\right| \underset{R}{\ll} 1 && (1 \leq q \leq m, \ n \in \mathbb{Z}_{\geq 0}) \label{eq:bound_gamma_inverse_uniform} \\
    & \left| \frac{1}{\Gamma(s_{r+q})} \right| \underset{R}{\ll} 1 && (1 \leq q \leq m). \label{eq:bound_gamma_inverse}
\end{align}

\item[$\bullet$] By compactness, it is clear that
\begin{align}
    &\left|x^{s_{r+q}-1+|\mathbf{v}_{q \bullet}|} \right| \underset{R,\varepsilon}{\ll} 1 && (1 \leq q \leq m, \ x \in [\varepsilon,1]), \label{eq:bound_x_power} \\
    &\left|\varepsilon^{\sum_{q=1}^{\alpha'} (s_{r+a'_q}+|\mathbf{u}_{q \bullet}|+|\mathbf{v}_{q \bullet}|)} \right| \underset{R,\varepsilon}{\ll} \varepsilon^{\sum_{q=1}^{\alpha'} |\mathbf{u}_{q \bullet}|}. \label{eq:bound_epsilon_power}
\end{align}

\item[$\bullet$] By setting $L:= \max_{1 \leq p \leq r, 1 \leq q \leq m} |c_{q,p}|$, we clearly have that 
\begin{equation}
    \left| \binom{k_p}{\mathbf{v}_{\bullet p}} \prod_{q=1}^m c_{q,a_p}^{v_{q,p}} \right| \leq \binom{k_p}{\mathbf{v}_{\bullet p}} L^{k_p} \qquad (1 \leq p \leq \alpha). \label{eq:bound_binomial}
\end{equation}

\item[$\bullet$] Using the bound
$$ \left| \binom{z-1}{k} \right| \leq \binom{R+k}{k} \qquad (k \in \mathbb{Z}_{\geq 0}, \ z \in \overline{D}_0(R)) $$ 
we obtain
\begin{equation}
    \left| \prod_{p=1}^{\beta} \left( \binom{s_{b_p}-1}{|\mathbf{u}_{\bullet p}|} \binom{|\mathbf{u}_{\bullet p}|}{\mathbf{u}_{\bullet p}} \prod_{q=1}^{\alpha'} c_{a'_q,b_p}^{u_{q,p}} \right) \right| \underset{r,\varepsilon}{\ll} L^{|\mathbf{u}|} \prod_{p=1}^{\beta} \left( \binom{R+|\mathbf{u}_{\bullet p}|}{|\mathbf{u}_{\bullet p}|} \binom{|\mathbf{u}_{\bullet p}|}{\mathbf{u}_{\bullet p}} \right). \label{eq:bound_multinomial}
\end{equation}

\item[$\bullet$] Let $1 \leq p \leq \beta$ and $(y_1,\ldots,y_{\beta'}) \in [\varepsilon,1]^{\beta'}$. It is clear that
$$ \left|\left( c_{j,b_p}+\sum_{q=1}^{\beta'} c_{b_q',b_p} y_{q} \right)^{s_{b_p}-1-|\mathbf{u}_{\bullet p}|}\right| \leq \left|c_{j,b_p}+\sum_{q=1}^{\beta'} c_{b'_q,b_p} y_q \right|^{\sigma_{b_p}-1-|\mathbf{u}_{\bullet p}|} \phie^{R \frac{\pi}{2}}. $$
Note that we have the following bound:
$$ l \leq \left( \Re(c_{j,b_p})+\sum_{q=1}^{\beta'} \Re(c_{b'_q,b_p}) \varepsilon \right) \leq \left| c_{j,b_p}+\sum_{q=1}^{\beta'} c_{b'_q,b_p} y_{q} \right| \leq \sum_{q=1}^{m} |c_{q,b_p}| \leq mL $$ 
with $L := \max_{1 \leq p \leq r, 1 \leq q \leq m}|c_{q,p}|>0, \ l := \min_{1 \leq p \leq r, 1 \leq q \leq m}{\Re(c_{q,p})}>0$. Thus
\[
    \left| c_{j,b_p}+\sum_{q=1}^{\beta'} c_{b'_q,b_p} y_{q} \right|^{\sigma_{b_p}-1-|\mathbf{u}_{\bullet p}|} \leq \left\{
    \begin{aligned}
        & (mL)^{\sigma_{b_p}-|\mathbf{u}_{\bullet p}|-1} \text{ if } \left|c_{j,b_p}+\sum_{q=1}^{\beta'} c_{b'_q,b_p} y_q \right| \geq 1 \\
        & \qquad \qquad \qquad \text{ and } \sigma_{b_p}-|\mathbf{u}_{\bullet p}| \geq 1 \\
        & 1 \text{ if } \left|c_{j,b_p}+\sum_{q=1}^{\beta'} c_{b'_q,b_p} y_q \right| \geq 1, \ \sigma_{b_p}-|\mathbf{u}_{\bullet p}| < 1 \\
        & 1 \text{ if } \left|c_{j,b_p}+\sum_{q=1}^{\beta'} c_{b'_q,b_p} y_q \right| < 1, \ \sigma_{b_p}-|\mathbf{u}_{\bullet p}| \geq 1 \\
        & K^{\sigma_{b_p}-|\mathbf{u}_{\bullet p}|-1} \text{ else}.
    \end{aligned}
    \right. 
\]

By compactness, we also have:
$$ (mL)^{\sigma_p-|\mathbf{u}_{\bullet p}|-1} \underset{R}{\ll} (mL)^{-|\mathbf{u}_{\bullet p}|}, \quad l^{\sigma_p-|\mathbf{u}_{\bullet p}|-1} \underset{R}{\ll} l^{-|\mathbf{u}_{\bullet p}|}. $$
Finally, we find that
\begin{equation}
    \left| \left(c_{j,b_p}+\sum_{q=1}^{\beta'} c_{b'_q,b_p} y_{q} \right)^{s_{b_p}-1-|\mathbf{u}_{\bullet p}|} \right| \underset{R}{\ll} \min(l,mL,1)^{-|\mathbf{u}_{\bullet p}|}. \label{eq:bound_integrand}
\end{equation}
\end{enumerate}
From (\ref{eq:bound_x_power}) and (\ref{eq:bound_integrand}), we get:
\begin{align*}
    & \left| \int_{\varepsilon}^{1} \cdots \int_{\varepsilon}^1 \prod_{q=1}^{\beta'} \frac{y_q^{s_{r+b'_q}-1+ |\mathbf{v}_{q \bullet}|}}{\Gamma(s_{r+b'_q})} \prod_{p=1}^{\beta} \left( c_{j,b_p}+\sum_{q=1}^{\beta'} c_{b'_q,b_p} y_{q} \right)^{s_{a_p}-1-|\mathbf{u}_{\bullet p}|} \dd y_1 \cdots \dd y_{\beta'} \right| \\ 
    & \qquad \underset{R,\varepsilon}{\ll} \min(l,mL,1)^{-|\mathbf{u}_{\bullet p}|}.
\end{align*}
From the above inequality and from (\ref{eq:bound_gamma_inverse_uniform}), (\ref{eq:bound_gamma_inverse}), (\ref{eq:bound_epsilon_power}), (\ref{eq:bound_binomial}), (\ref{eq:bound_multinomial}), we find that 
$$ |T_{\mathcal{A}',\mathbf{u},\mathbf{v}}(\mathbf{s})| \underset{R,\varepsilon}{\ll} \left(\frac{L\varepsilon}{\min(l,mL,1)} \right)^{|\mathbf{u}|} \prod_{p=1}^{\alpha} \left( \binom{k_p}{\mathbf{v}_{\bullet p}} L^{k_p} \right) \prod_{p=1}^{\beta} \left( \binom{R+|\mathbf{u}_{\bullet p}|}{|\mathbf{u}_{\bullet p}|} \binom{|\mathbf{u}_{\bullet p}|}{\mathbf{u}_{\bullet p}} \right).$$
By Newton's multinomial, we have $\displaystyle{\sum_{\substack{v_{1,1}+\cdots+v_{m,1}=k_1 \\ \cdots \\ v_{1,\alpha}+\cdots+v_{m,\alpha}=k_{\alpha}}} \prod_{p=1}^{\alpha} \left( \binom{k_p}{\mathbf{v}_{\bullet p}} L^{k_p} \right)=(mL)^{|\mathbf{k}|}}$. It is straightforward to check that the series 
$$ (mL)^{|\mathbf{k}|} \sum_{u_{1,1},\ldots,u_{\alpha',\beta} \geq 0} \left(\frac{L\varepsilon}{\min(l,mL,1)} \right)^{|\mathbf{u}|} \prod_{p=1}^{\beta} \left( \binom{R+|\mathbf{u}_{\bullet p}|}{|\mathbf{u}_{\bullet p}|} \binom{|\mathbf{u}_{\bullet p}|}{\mathbf{u}_{\bullet p}} \right) $$
converges absolutely for all $0 < \varepsilon < \frac{\min(l,mL,1)}{L}$. Therefore, the series on the right-hand side of (\ref{eq:analytic_continuation_F}) is normally convergent on the compact set $K$, for $\varepsilon>0$ sufficiently small.

ii) Let $\mathbf{s}=(s_1,\ldots,s_{r+m}) \in \mathbb{C}^{r+m}$ such that $\Re(s_1),\ldots,\Re(s_{r+m})>1$. We now want to prove that (\ref{eq:analytic_continuation_F}) holds for such $\mathbf{s}$. Let $\mathcal{A}'=(a'_1,\ldots,a'_{\alpha'}) \subseteq [\![ 1,m ]\!] \setminus \{ j \}$ and $\mathcal{B}'=(b'_1,\ldots,b'_{\alpha'}) := [\![ 1,m ]\!] \setminus ( \mathcal{A}' \cup \{ j \})$. We set
\[
    V_{\mathcal{A}',j}(\varepsilon):= \left\{ (y_1,\ldots,\widehat{y_{j}},\ldots,y_{m}) \in [0,1]^{m-1} ; \
    \begin{aligned} 
        &y_{a'_q} \leq \varepsilon \quad (1 \leq q \leq \alpha') \\
        &y_{b'_q} \geq \varepsilon \quad (1 \leq q \leq \beta')
    \end{aligned}
    \right\} \qquad (1>\varepsilon>0). 
\]
We can then partition the following integration range:
$$ \{ (y_1,\ldots,\widehat{y_j},\ldots,y_{m}); \forall q \neq j, y_q \in [0,1] \} = \bigcup_{\substack{\mathcal{A}' \subseteq [\![1,m]\!] \setminus \{ j \}}} V_{\mathcal{A}',j}(\varepsilon), $$
with pairwise intersections of measure zero. By the additivity property of the integral, we get
\begin{align}
    &f_{j,\mathcal{A},\mathbf{k}}(\mathbf{c},\mathbf{s}) \nonumber \\
    &= \sum_{\mathcal{A}' \subseteq [\![1,m]\!] \setminus \{ j \}} \int_{V_{\mathcal{A}',j}(\varepsilon)} \prod_{\substack{q=1 \\ q \neq j}}^{m} \frac{y_q^{s_{r+q}-1}}{\Gamma(s_{r+q})} \prod_{p=1}^{\beta} l^*_{b_p}(\widehat{\mathbf{y}}^j)^{s_{b_p}-1} \prod_{p=1}^{\alpha} l^*_{a_p}(\widehat{\mathbf{y}}^j)^{k_p} \dd y_1 \cdots \widehat{\dd y_{j}} \cdots \dd y_{m}. \label{eq:expression_chasles_f_Ajepsilon_(s)}
\end{align}
We now aim to expand the integrand of each integral into power series around the origin. Note that these power series converge uniformly on a sufficiently small compact disk. Therefore we are able to perform a series-integral inversion.

Let $\mathcal{A}'=(a'_1,\ldots,a'_{\alpha'}) \subseteq [\![1,m]\!] \setminus \{ j \}$ and $\varepsilon \in (0,\frac{1}{m \max_{1 \leq q \leq m, 1 \leq p \leq r}|c_{q,p}|})$. By Taylor's expansion
$$ (1+X_1+\cdots+X_d)^{-s} = \sum_{k_1,\ldots,k_d \geq 0} \binom{-s}{|\mathbf{k}|} \binom{|\mathbf{k}|}{\mathbf{k}} X_1^{k_1} \cdots X_{d}^{k_d} \quad (|X_1|+\cdots+|X_d|<1), $$
we obtain the following expansion, uniformly in $\widehat{x}^j \in F_{\mathcal{A}',j}(\varepsilon)$,
\begin{align*}
    l^*_{b_p}&(\widehat{\mathbf{y}}^j)^{s_{b_p}-1} \\
    &= \sum_{u_{p,1}, \ldots, u_{p,\alpha'} \geq 0} \binom{s_{b_p}-1}{|\mathbf{u}_{\bullet p}|} \binom{|\mathbf{u}_{\bullet p}|}{\mathbf{u}_{\bullet p}} \left( c_{j,b_p}+\sum_{q=1}^{\beta'} c_{b'_q,b_p} y_q \right)^{s_{b_p}-1-|\mathbf{u}_{\bullet p}|} \prod_{q=1}^{\alpha'} \left(c_{a'_q,b_p}^{u_{q,p}} y_{q}^{u_{q,p}}\right).
\end{align*}
Let $p \in [\![ 1, \alpha ]\!]$ and $k_p \in \mathbb{Z}_{\geq 0}$. By Newton's multinomial, we have
$$ l^*_{a_p}(\widehat{\mathbf{y}}^j)^{k_p} = \sum_{\substack{v_{1,1}+\cdots+v_{m,1}=k_1 \\ \cdots \\ v_{1,\alpha}+\cdots+v_{m,\alpha}=k_{\alpha}}}  \binom{k_p}{\mathbf{v}_{\bullet p}} c_{j,p}^{v_{j,p}} \prod_{\substack{q=1 \\ q \neq j}}^m c_{q,p}^{v_{q,p}} y_{q}^{v_{q,p}}. $$
Plugging these last two expressions into the integrand on the right-hand side of (\ref{eq:expression_chasles_f_Ajepsilon_(s)}), we obtain
\begin{align}
    &\int_{V_{\mathcal{A}',j}(\varepsilon)} \prod_{\substack{q=1 \\ q \neq j}}^{m} \frac{y_q^{s_{r+q}-1}}{\Gamma(s_{r+q})} \prod_{p=1}^{\beta} l^*_{b_p}(\widehat{\mathbf{y}}^j)^{s_{b_p}-1} \prod_{p=1}^{\alpha} l^*_{a_p}(\widehat{\mathbf{y}}^j)^{k_p} \dd y_1 \cdots \widehat{\dd y_{j}} \cdots \dd y_{m} \nonumber \\
    & \ = \sum_{\substack{ u_{1,1},\ldots,u_{\alpha',\beta} \geq 0 \\ v_{1,1}+\cdots+v_{m,1}=k_1 \\ \cdots \\ v_{1,\alpha}+\cdots+v_{m,\alpha}=k_{\alpha}}} \prod_{p=1}^{\alpha} \left( \binom{k_p}{\mathbf{v}_{\bullet p}} c_{j,a_p}^{v_{j,p}} \right) \prod_{p=1}^{\beta} \left( \binom{s_{b_p}-1}{|\mathbf{u}_{\bullet p}|} \binom{|\mathbf{u}_{\bullet p}|}{\mathbf{u}_{\bullet p}} \right) \label{eq:taylor_series_injection_f_AJepsilon_s} \\
    & \quad \times \int_{V_{\mathcal{A}',j}(\varepsilon)} \prod_{\substack{q=1 \\ q \neq j}}^{m} \left( \frac{y_q^{s_{r+q}-1}}{\Gamma(s_{r+q})} \prod_{p=1}^{\alpha} \left( c_{q,a_p}^{v_{q,p}} y_{q}^{v_{q,p}} \right) \right) \nonumber \\
    & \qquad \times \prod_{p=1}^{\beta} \left( \left( c_{j,b_p}+\sum_{q=1}^{\beta'} c_{b'_q,b_p} y_{q} \right)^{s_{b_p}-1-|\mathbf{u}_{\bullet p}|} \prod_{q=1}^{\alpha'} \left( c_{a'_q,b_p}^{u_{q,p}} y_{q}^{u_{q,p}} \right) \right) \dd y_1 \cdots \widehat{\dd y_{j}}\cdots \dd y_{m}. \nonumber
\end{align}
Using Fubini's theorem, we then obtain that the integral inside that last expression is equal to
\begin{align*}
    &\prod_{q=1}^{\alpha'} \frac{\varepsilon^{s_{r+a'_q}+|\mathbf{u}_{q \bullet}|+|\mathbf{v}_{q \bullet}|}}{\Gamma(s_{r+a'_q})(s_{r+a'_q} +|\mathbf{u}_{q \bullet}|+|\mathbf{v}_{q \bullet}|)} \prod_{q=1}^{\alpha'} \left(\prod_{p=1}^{\beta} \left( c_{a'_q,b_p}^{u_{q,p}} y_{q}^{u_{q,p}} \right) \right) \prod_{\substack{q=1 \\ q \neq j}}^{m} \left( \prod_{p=1}^{\alpha} c_{q,a_p}^{v_{q,p}} \right)  \\
    & \qquad \times \int_{\varepsilon}^1 \cdots \int_{\varepsilon}^1 \prod_{q=1}^{\beta'} \frac{y_q^{s_{r+b'_q}-1 + |\mathbf{v}_{q \bullet}|}}{\Gamma(s_{r+b'_q})} \prod_{p=1}^{\beta}  \left( c_{j,b_p}+\sum_{q=1}^{\beta} c_{b'_q,b_p} y_{q} \right)^{s_{b_p}-1-|\mathbf{u}_{\bullet p}|} \dd y_1 \cdots \dd y_{\beta'}.
\end{align*}
Plugging the previous equality into (\ref{eq:taylor_series_injection_f_AJepsilon_s}), and using (\ref{eq:expression_chasles_f_Ajepsilon_(s)}), we get (\ref{eq:expression_f_Ajepsilon_s}).
\end{proof}

\begin{remark}
We proved that the series (with respect to $(\mathbf{u},\mathbf{v})$) on the right-hand side of (\ref{eq:expression_f_Ajepsilon_s}) is normally convergent on any compact subset of $\mathbb{C}^{r+m}$ in the variables $s_1,\ldots,s_{r+m}$. Therefore, one can compute the limit of $f_{j,\mathcal{A},\mathbf{k}}(\mathbf{c},\mathbf{s})$ at $\mathbf{s}=(-N_1,\ldots,-N_{r+m})$ by interchanging the limit and the series on the right-hand side of (\ref{eq:expression_f_Ajepsilon_s}). One can also compute the directional derivative values of $f_{j,\mathcal{A},\mathbf{k}}$ by taking the directional derivative values of the general term of the series on the right-hand side of (\ref{eq:expression_f_Ajepsilon_s}).
\end{remark}

\subsection{Meromorphic continuation of \texorpdfstring{$F(\mathbf{c},\mathbf{x},\mathbf{s},t)$}{analyticcontinuationF}}
\label{subsection:analytic_continuation_F}

We established that $\zeta(\mathbf{c},\mathbf{x},\mathbf{s}) = F(\mathbf{c},\mathbf{x},\mathbf{s},t)+H(\mathbf{c},\mathbf{x},\mathbf{s},t)$ $(t>0,\Re(s_1),\ldots,\Re(s_{r+m})>1)$, where $\mathbf{s}\mapsto H(\mathbf{c},\mathbf{x},\mathbf{s},t)$ is entire and vanishes when $\mathbf{s} \in \mathbb{Z}_{\leq 0}^{r+m}$. We also showed in (\ref{eq:analytic_continuation_F}) an explicit formula for $F(\mathbf{c},\mathbf{x},\mathbf{s},t)$, where $\Re(s_1),\ldots,\Re(s_{r+m})>1$. Here, we extend this formula to $\mathbb{C}^{r+m}$, which introduces singularities along a union of hyperplanes.

\begin{proposition}
\label{prop:analytic_continuation_F_bis}
Let $\mathbf{s}=(s_1,\ldots,s_{r+m}) \in \mathbb{C}^{r+m} \setminus \mathcal{S}_{r,m}$. There exists $t_0>0$ such that for all $t \in (0,t_0)$, we have
\begin{align}
    F(\mathbf{c},\mathbf{x},\mathbf{s},t) = \sum_{\substack{\mathcal{A} \subseteq [\![1,r]\!] \\ 1 \leq j \leq m}} \prod_{p=1}^{\beta}& \Gamma(1-s_{b_p}) \sum_{k_1,\ldots,k_{\alpha}\geq 0} (-1)^{|\mathbf{k}|} f_{j,\mathcal{A},\mathbf{k}}(\mathbf{s}) \label{eq:analytic_continuation_F_bis} \\
    & \times \frac{t^{|\mathbf{k}|-\beta+|\mathbf{s}|_{|\mathcal{B}\cup [\![ r+1,r+m ]\!]}}}{\Gamma(s_{r+j})(|\mathbf{k}|-\beta+|\mathbf{s}|_{|\mathcal{B}\cup [\![ r+1,r+m ]\!]})} \prod_{p=1}^{\alpha} \frac{\zeta(s_{a_p}-k_p,x_{a_p})}{k_p!} \nonumber
\end{align}
where we recall that we write $\mathcal{A}=\{ a_1,\ldots,a_{\alpha}\} \subseteq [\![ 1,r ]\!]$, $\mathcal{B}=\{ b_1,\ldots,b_{\beta}\} := [\![ 1,r ]\!] \setminus \mathcal{A}$. Moreover, the singularities of $\mathbf{s} \mapsto F(\mathbf{c},\mathbf{x},\mathbf{s},t)$ belong to $\mathcal{S}_{r,m}$.
\end{proposition}

\begin{proof}
By plugging (\ref{eq:relation_integral_after_blowup}) into (\ref{eq:analytic_continuation_F}), we obtain that (\ref{eq:analytic_continuation_F_bis}) holds for all $\mathbf{s}\in (\mathbb{C}\setminus \mathbb{N})^{r+m}$ such that $\Re(s_1),\ldots,\Re(s_{r+m})>1$. We now prove that the series on the right-hand side of (\ref{eq:analytic_continuation_F_bis}) is normally convergent on a suitable compact set that does not intersect the set $\mathcal{S}_{r,m}$, and we prove such a result by using bounds obtained for the Hurwitz zeta function and for $f_{j,\mathcal{A},\mathbf{k}}(\mathbf{s})$. We need to choose a suitable compact set by accounting for the possible singularities arising from the right-hand side of (\ref{eq:analytic_continuation_F_bis}).

Let $R>0$ and $\delta \in (0,1)$. We set
\[
    K_{r,m}(\delta,R) := \left\{ \mathbf{s} \in \overline{D}_0(R)^{r+m} ;
    \begin{aligned}
        &\bullet \min_{n \in \mathbb{N}}|s_{p}-n| \geq \delta && (1 \leq p \leq r) \\
        &\bullet \min_{n \in \mathbb{Z}_{\leq m}}\left|\sum_{a \in \mathcal{A}} s_{a} + \sum_{q=1}^m s_{r+q}-n\right| \geq \delta && (\mathcal{A} \subseteq [\![ 1,r ]\!])
    \end{aligned}
    \right\}.
\]
Note that the set $K_{r,m}(\delta,R)$ contains the compact set $D_{r,m}(\delta,R)$ as defined in (\ref{eq:definition_D_deltaR}). Observe that each term of the series on the right-hand side of (\ref{eq:analytic_continuation_F_bis}) is holomorphic on $K_{r,m}(\delta,R)$. Let $\mathcal{A}=
\{ a_1,\ldots,a_{\alpha} \} \subseteq [\![1,r]\!]$, $\mathcal{B}=\{ b_1,\ldots,b_{\beta} \} := [\![1,r]\!]\setminus \mathcal{A}$, $j \in [\![ 1,m ]\!]$. For all $\mathbf{k}=(k_1,\ldots,k_{\alpha}) \in \mathbb{Z}_{\geq 0}^{\alpha}$, we have:
\begin{enumerate}
\item[$\bullet$] If $k_1+\cdots+k_{\alpha} \geq \beta$, then by Lemma \ref{lem:bound_1_over_gamma} we find that
$$ \left|\frac{1}{\Gamma(s_{r+j})(|\mathbf{k}|-\beta+|\mathbf{s}|_{|\mathcal{B}\cup [\![ r+1,r+m ]\!]})}\right| \underset{R,\delta}{\ll} 1 \qquad (\mathbf{s} \in K_{r,m}(\delta,R)). $$
\item[$\bullet$] If $k_1+\cdots+k_{\alpha} < \beta$, we know that the function 
$$\mathbf{s} \in K_{r,m}(\delta,R) \mapsto \frac{1}{\Gamma(s_{r+j})(|\mathbf{k}|-\beta+|\mathbf{s}|_{|\mathcal{B}\cup [\![ r+1,r+m ]\!]})}$$
is entire. By compactness, we have
$$ \left|\frac{1}{\Gamma(s_{r+j})(|\mathbf{k}|-\beta+|\mathbf{s}|_{|\mathcal{B}\cup [\![ r+1,r+m ]\!]})}\right| \underset{R,\delta,\mathbf{k}}{\ll} 1 \qquad (\mathbf{s} \in K_{r,m}(\delta,R)). $$
\end{enumerate}
Since there's only a finite number of nonnegative tuples $\mathbf{k}$ satisfying $k_1+\cdots+k_{\alpha} \leq \beta$, we get that
\begin{equation}
    \left|\frac{1}{\Gamma(s_{r+j})(|\mathbf{k}|-\beta+|\mathbf{s}|_{|\mathcal{B}\cup [\![ r+1,r+m ]\!]})}\right| \underset{R,\delta}{\ll} 1 \qquad (\mathbf{s} \in K_{r,m}(\delta,R), \mathbf{k} \in \mathbb{Z}_{\geq 0}^{\alpha}). \label{eq:bound_1_over_gamma_continuation_F}
\end{equation}

We see that the function $\mathbf{s} \mapsto \prod_{p=1}^{\beta} \Gamma(1-s_{b_p})$ is entire on $K_{r,m}(\delta,R)$. By compactness we have
\begin{equation}
    \left|\prod_{p=1}^{\beta} \Gamma(1-s_{b_p}) \right| \underset{R,\delta}{\ll} 1 \qquad (\mathbf{s} \in K_{r,m}(\delta,R)). \label{eq:bound_product_gamma}
\end{equation}
One can also prove that
\begin{equation}
    \left|t^{|\mathbf{k}|-\beta+|\mathbf{s}|_{|\mathcal{B}\cup [\![ r+1,r+m ]\!]}}\right| \underset{R,\delta}{\ll} t^{|\mathbf{k}|} \qquad (\mathbf{s} \in K_{r,m}(\delta,R), \mathbf{k} \in \mathbb{Z}_{\geq 0}^{\alpha}). \label{eq:ineq_theta_power}
\end{equation}

By (\ref{eq:bound_1_over_gamma_continuation_F}), (\ref{eq:bound_product_gamma}), (\ref{eq:bound_f_Ajk}), (\ref{eq:ineq_theta_power}), and using Lemma \ref{lem:bound_hurwitz_zeta}, we obtain that
\begin{align}
    &\left| f_{j,\mathcal{A},\mathbf{k}}(\mathbf{s}) \frac{t^{|\mathbf{k}|-\beta+|\mathbf{s}|_{|\mathcal{B}\cup [\![ r+1,r+m ]\!]}}}{\Gamma(s_{r+j})(|\mathbf{k}|-\beta+|\mathbf{s}|_{|\mathcal{B}\cup [\![ r+1,r+m ]\!]})} \prod_{p=1}^{\beta} \Gamma(1-s_{b_p}) \prod_{p=1}^{\alpha} \frac{\zeta(s_{a_p}-k_p,x_{a_p})}{k_p!} \right| \nonumber \\
    & \qquad \underset{R,\varepsilon,\delta}{\ll} \left(\frac{m \max_{1 \leq p \leq r, 1 \leq q \leq m}|c_{q,p}|}{2\pi} t\right)^{|\mathbf{k}|} \prod_{p=1}^{\alpha} (k_p+1)^{R+\varepsilon} \quad (\varepsilon>0), \label{eq:bound_general_term_series_F}
\end{align}
Note that the series $\sum_{k_1,\ldots,k_{\alpha} \geq 0} (\frac{m \max_{1 \leq p \leq r, 1 \leq q \leq m}|c_{q,p}|}{2\pi} t)^{|\mathbf{k}|} \prod_{p=1}^{\alpha} (k_p+1)^{R+\varepsilon}$ converges for all $0 < t < \frac{2\pi}{m \max_{1 \leq p \leq r, 1 \leq q \leq m}|c_{q,p}|} =: t_0$. Combining the previous fact with the bound (\ref{eq:bound_general_term_series_F}), we conclude that the series on the right-hand side of (\ref{eq:analytic_continuation_F_bis}) is normally convergent on $K_{r,m}(\delta,R)$ in the variables $\mathbf{s}=(s_1,\ldots,s_{r+m})$, for any $t \in (0,t_0)$.
\end{proof}

Combining the expansion (\ref{eq:crandall_expansion_zeta}) with the convergence result shown above, we obtain a meromorphic continuation for $\zeta(\mathbf{c},\mathbf{x},\mathbf{s})$ to the whole complex space $\mathbb{C}^{r+m}$.

\begin{corollary}
\label{cor:crandall_expansion_zeta}
Let $t>0$ sufficiently small. We have
\begin{equation}
    \zeta(\mathbf{c},\mathbf{x},\mathbf{s}) = F(\mathbf{c},\mathbf{x},\mathbf{s},t) + H(\mathbf{c},\mathbf{x},\mathbf{s},t) \qquad (\mathbf{s} \in \mathbb{C}^{r+m} \setminus \mathcal{S}_{r,m}).  \label{eq:analytic_continuation_zeta}
\end{equation}
\end{corollary}

From Corollary \ref{cor:crandall_expansion_zeta}, Proposition \ref{prop:analytic_continuation_F_bis}, and Corollary \ref{cor:H_entire}, we obtain Theorem \ref{th:analytic_continuation_zeta}.

\subsection{Regularity of \texorpdfstring{$F(\mathbf{c},\mathbf{x},-\mathbf{N}+z\protect \bm{\theta},t)$}{zetaalongdirection} at \texorpdfstring{$z=0$}{z0}}
\label{subsection:regularity_directional_zeta_0}

Let $\mathbf{N} \in \mathbb{Z}_{\geq 0}^{r+m}$ and $\bm{\theta} \in \mathbb{C}^{r+m}$ such that (\ref{eq:direction_non_vanishing_condition}) holds. We prove in this subsection that the single-variable function $z \mapsto \zeta(\mathbf{c},\mathbf{x},-\mathbf{N}+z\bm{\theta})$ is meromorphic on $\mathbb{C}$ and regular at $z=0$, thereby justifying the notation (\ref{eq:def_directional_value}) and (\ref{eq:def_directional_derivative_value}). Thanks to Theorem \ref{th:analytic_continuation_zeta} and Corollary \ref{cor:H_entire}, it suffices to show that $z \mapsto F(\mathbf{c},\mathbf{x},-\mathbf{N}+z\bm{\theta},t)$ is regular at $z=0$. To that end, we will show that (\ref{eq:analytic_continuation_F_bis}) holds around $z=0$ after setting $\mathbf{s}=-\mathbf{N}+z\bm{\theta}$.

\begin{proposition}
\label{prop:analytic_continuation_directional_F}
Let $t>0$ and $\eta>0$ sufficiently small. For all $z \in D_0(\eta)$, we have
\begin{align}
    F(\mathbf{c},\mathbf{x},-\mathbf{N}+z\bm{\theta},t)& = \sum_{\substack{\mathcal{A} \subseteq [\![1,r]\!] \\ 1 \leq j \leq m}} \prod_{p=1}^{\beta} \Gamma(1+N_{b_p}-z\theta_{b_p}) \sum_{k_1,\ldots,k_{\alpha}\geq 0} (-1)^{|\mathbf{k}|} f_{j,\mathcal{A},\mathbf{k}}(-\mathbf{N}+z\bm{\theta}) \nonumber \\
    & \times \frac{t^{|\mathbf{k}|-\beta-|\mathbf{N}|_{|\mathcal{B}\cup [\![ r+1,r+m ]\!]}+z|\bm{\theta}|_{|\mathcal{B}\cup [\![ r+1,r+m ]\!]}}}{\Gamma(-N_{r+j}+z\theta_{r+j})(|\mathbf{k}|-\beta-|\mathbf{N}|_{|\mathcal{B}\cup [\![ r+1,r+m ]\!]}+z|\bm{\theta}|_{|\mathcal{B}\cup [\![ r+1,r+m ]\!]})} \nonumber \\
    & \times \prod_{p=1}^{\alpha} \frac{\zeta(-N_{a_p}+z\theta_{a_p}-k_p,x_{a_p})}{k_p!}. \label{eq:analytic_continuation_directional_F}
\end{align}
\end{proposition}

\begin{proof}
Let $t>0$ be sufficiently small. By definition of $\mathcal{S}_{r,m}$ in Theorem \ref{th:analytic_continuation_zeta}, we have
$$ -\mathbf{N}+z\bm{\theta} \in \mathcal{S}_{r,m} \iff z \in \left( \bigcup_{\mathcal{A} \subseteq [\![ 1,r ]\!]} \frac{1}{|\bm{\theta}|_{|\mathcal{A} \cup [\![ r+1,r+m ]\!]}} \mathbb{Z}_{\leq |\mathbf{N}|_{|\mathcal{A} \cup [\![ r+1,r+m ]\!]}} \right) \cup \left( \bigcup_{\substack{p \in [\![ 1,r ]\!] \\ \text{s.t } \theta_p \neq 0}} \frac{1}{\theta_p} \mathbb{N}\right). $$
We take a radius $\eta$ so that the function $z \in \overline{D}_0(\eta) \mapsto F(\mathbf{c},\mathbf{x},-\mathbf{N}+z\bm{\theta},t)$ has at most one singularity in the disk, at $z=0$. Let $\eta>0$ and $R \in \mathbb{N}$ such that
\begin{align*}
    \eta :=& \frac{1}{2} \min\left( \min_{p \in [\![1,r]\!] \text{ s.t. } \theta_p \neq 0} \left| \frac{1}{\theta_p} \right| ,\min_{\mathcal{A} \subseteq [\![ 1,r ]\!]} \left| \frac{1}{\sum_{a \in \mathcal{A}}\theta_{a}+\sum_{q=1}^m \theta_{r+q}} \right|\right) \\
    R :=& \left\lceil \max_{1 \leq k \leq r+m}(N_{k}+|\theta_k\eta|)\right\rceil \in \mathbb{N}.
\end{align*}
Such $\eta$ exists because of the non-vanishing condition (\ref{eq:direction_non_vanishing_condition}) on $\bm{\theta}$. We clearly have
$$ \forall z \in \overline{D}_0(\eta) \setminus \{ 0 \}, \ -\mathbf{N}+z\bm{\theta} \not \in \mathcal{S}_{r,m}, $$
and with our choice for $\eta$ and $R$, we also have
$$ \{ -\mathbf{N}+z\bm{\theta}; z \in \overline{D}_0(\eta) \} \subseteq \overline{D}_0(R). $$
Therefore, for $\delta>0$ sufficiently small, the compact set
$$ K(\eta,\delta,R):=\{ -\mathbf{N}+z\bm{\theta}; z \in \overline{D}_0(\eta) \} \cap K(\delta,R) $$
contains a non-empty open connected subset of $K(\delta,R)$. We observe that (\ref{eq:analytic_continuation_directional_F}) holds for all $z \in \overline{D}_0(\eta)$ such that $-\mathbf{N}+z\bm{\theta} \in K(\delta,R)$. Therefore it remains to show that the formula on the right-hand side of (\ref{eq:analytic_continuation_directional_F}) is normally convergent on $\overline{D}_0(\eta)$. We will then bound the general term of the series on the right-hand side of (\ref{eq:analytic_continuation_directional_F}), and we will use the same bounds as those used in the proof of Proposition \ref{prop:analytic_continuation_F_bis}.

Let $\mathcal{A}=\{ a_1,\ldots,a_{\alpha} \} \subseteq [\![ 1,r ]\!]$, $\mathcal{B}= \{ b_1,\ldots,b_{\beta} \} := [\![ 1,r ]\!] \setminus \mathcal{A}$, $\mathbf{k}=(k_1,\ldots,k_{\alpha}) \in \mathbb{Z}_{\geq 0}^{\alpha}$, $1 \leq j \leq m$. By the definition of $\eta$, no singularity arises from the terms $\zeta(-N_{a_p}+z\theta_{a_p}-k_p,x_{a_p})$ $(1 \leq p \leq \alpha)$ for $z \in \overline{D}_0(\eta)$. Thanks to (\ref{eq:bound_f_Ajk}), Lemma \ref{lem:bound_1_over_gamma}, and Lemma \ref{lem:bound_hurwitz_zeta}, we obtain:
\begin{align*}
    &|f_{j,\mathcal{A},\mathbf{k}}(-\mathbf{N}+z\bm{\theta})| \underset{\eta}{\ll} \left( m \max_{1 \leq p \leq r, \ 1 \leq q \leq m} |c_{q,p}|\right)^{|\mathbf{k}|} \\
    &\left|\frac{t^{|\mathbf{k}|-\beta-|\mathbf{N}|_{|\mathcal{B}\cup [\![ r+1,r+m ]\!]}+z|\bm{\theta}|_{|\mathcal{B}\cup [\![ r+1,r+m ]\!]}}}{\Gamma(-N_{r+j}+z\theta_{r+j})(|\mathbf{k}|-\beta-|\mathbf{N}|_{|\mathcal{B}\cup [\![ r+1,r+m ]\!]}+z|\bm{\theta}|_{|\mathcal{B}\cup [\![ r+1,r+m ]\!]})} \right| \underset{\eta}{\ll} t^{|\mathbf{k}|} \\
    &\left| \prod_{p=1}^{\alpha} \frac{\zeta(-N_{a_p}+z\theta_{a_p}-k_p,x_{a_p})}{k_p!} \right| \underset{\varepsilon,\eta}{\ll} (2\pi)^{-|\mathbf{k}|}\prod_{p=1}^{\alpha}(k_p+1)^{\eta+\varepsilon} &&(\varepsilon>0)
\end{align*}
uniformly for $z \in \overline{D}_0(\eta)$. By compactness, we also have 
$$ \left| \prod_{p=1}^{\beta} \Gamma(1+N_{b_p}-z\theta_{b_p})\right| \underset{\eta}{\ll} 1 \qquad (z \in \overline{D}_0(\eta)). $$
Therefore, the general term of the series on the right-hand side of (\ref{eq:analytic_continuation_directional_F}) is bounded by $(\frac{m \max_{1 \leq p \leq r, 1 \leq q \leq m} |c_{q,p}|}{2\pi}t)^{|\mathbf{k}|}\prod_{p=1}^{\alpha} (k_p+1)^{\eta+\varepsilon}$. Assuming that $t>0$ is sufficiently small, we get that the series on the right-hand side of (\ref{eq:analytic_continuation_directional_F}) converges normally on $\overline{D}_0(\eta)$.
\end{proof}

\section{Computation of \texorpdfstring{$\zeta(\mathbf{c},\mathbf{x},\protect \underset{\bm{\theta}}{-\mathbf{N}})$}{zeta} and \texorpdfstring{$\zeta'(\mathbf{c},\mathbf{x},\protect \underset{\bm{\theta}}{-\mathbf{N}})$}{zetader}}
\label{section:computation_zeta}

In this section, we prove both Theorem \ref{th:directional_values_mzf} and Theorem \ref{th:directional_derivative_values_mzf}. Both proofs rely on Crandall's expansion established in the previous section, as well as on the meromorphic continuations of the functions $F$ and $H$. By Corollary \ref{cor:crandall_expansion_zeta}, we have
\begin{align}
    \zeta(\mathbf{c},\mathbf{x},\underset{\bm{\theta}}{-\mathbf{N}})=&F(\mathbf{c},\mathbf{x},\underset{\bm{\theta}}{-\mathbf{N}},t)+H(\mathbf{c},\mathbf{x},\underset{\bm{\theta}}{-\mathbf{N}},t) \label{eq:value_zeta_nonpositive_primary} \\
    \zeta'(\mathbf{c},\mathbf{x},\underset{\bm{\theta}}{-\mathbf{N}})=& F'(\mathbf{c},\mathbf{x},\underset{\bm{\theta}}{-\mathbf{N}},t)+H'(\mathbf{c},\mathbf{x},\underset{\bm{\theta}}{-\mathbf{N}},t) \label{eq:value_zeta_prime_nonpositive_primary}
\end{align}
for $t>0$ sufficiently small, where we set
\begin{align*}
    F(\mathbf{c},\mathbf{x},\underset{\bm{\theta}}{-\mathbf{N}},t) :=& \lim_{z \to 0} F(\mathbf{c},\mathbf{x},-\mathbf{N}+z\bm{\theta},t) \\
    H(\mathbf{c},\mathbf{x},\underset{\bm{\theta}}{-\mathbf{N}},t) :=& \lim_{z \to 0} H(\mathbf{c},\mathbf{x},-\mathbf{N}+z\bm{\theta},t) \\
    F'(\mathbf{c},\mathbf{x},\underset{\bm{\theta}}{-\mathbf{N}},t) :=& \lim_{z \to 0} \partial_z \left( F(\mathbf{c},\mathbf{x},-\mathbf{N}+z\bm{\theta},t) \right) \\
    H'(\mathbf{c},\mathbf{x},\underset{\bm{\theta}}{-\mathbf{N}},t) :=& \lim_{z \to 0} \partial_z \left( H(\mathbf{c},\mathbf{x},-\mathbf{N}+z\bm{\theta},t)\right).
\end{align*}
Thanks to Corollary \ref{cor:H_entire}, we get that $H(\mathbf{c},\mathbf{x},\underset{\bm{\theta}}{-\mathbf{N}},t)=0$. In the following subsections, all the remaining directional limits written above need to be computed to prove both Theorems \ref{th:directional_values_mzf} and \ref{th:directional_derivative_values_mzf}.

\subsection{Lemmata}
\label{subsection:lemmata_values}

We state two key lemmata to compute the two special values $\zeta(\mathbf{c},\mathbf{x},\underset{\bm{\theta}}{-\mathbf{N}})$ and $\zeta'(\mathbf{c},\mathbf{x},\underset{\bm{\theta}}{-\mathbf{N}})$.

\begin{lemma}
\label{lem:inverse_gamma_formula}
Let $a,b \in \mathbb{C} \setminus \{ 0 \}$, $N \in \mathbb{Z}_{\geq 0}$ and $n \in \mathbb{Z}$. Then we have
\[
    \frac{1}{\Gamma(az-N)(bz+n)} \underset{z \to 0}{=}
    \begin{cases}
        \frac{(-1)^{N}aN!}{b} + \frac{(-1)^{N}a^2N!}{b}\left(\gamma-h_N \right)z + O(z^2) & \text{ if } n=0 \\
        \frac{(-1)^{N}aN!}{n}z+O(z^2) & \text{ if } n \neq 0
    \end{cases}
\]
where $h_N$ is the $N$-th harmonic number. Moreover, we have $( \frac{1}{\Gamma})'(-N)=(-1)^{N} N!$.
\end{lemma}

\begin{proof}
It follows from applying Gamma's functional relation, and using the first-order Taylor expansion of $\frac{1}{\Gamma}$ at nonpositive integers.
\end{proof}

\begin{lemma}
\label{lem:derivative_value_binomial}
Let $(n_1,\ldots,n_{\alpha}) \in \mathbb{Z}_{\geq 0}^{\alpha}$, we have
\begin{equation}
    \label{eq:formula_derivative_binom_coeff}
    \lim_{z \to 0} \partial_z \left( \prod_{i=1}^{\alpha} \binom{-N_i-1+z\theta_i}{n_i} \right) = \left( \sum_{i=1}^{\alpha} \theta_i (h_{N_i}-h_{N_i+n_i}) \right) \prod_{i=1}^{\alpha} \binom{-N_i-1}{n_i}.
\end{equation}
\end{lemma}

\begin{proof}
It follows from Leibniz's rule and from
\begin{align*}
    \lim_{z \to 0} \partial_z \left( \binom{-N_i-1+z\theta_i}{n_i} \right) =& \theta_i \left( \sum_{j=0}^{n_i-1} \frac{(-N_i-1) \cdots \reallywidehat{(-N_i-1-j)} \cdots (-N_i-n_i) }{n_i!} \right) \\
    =& \theta_i \binom{-N_i-1}{n_i} \sum_{j=0}^{n_i-1} \frac{1}{-N_i-1-j},
\end{align*}
where $\reallywidehat{(-N_i-1-j)}$ indicates that the factor $(-N_i-1-j)$ is omitted from the product. Note that the last sum yields a difference between harmonic numbers.
\end{proof}

\subsection{Proof of Theorem \ref{th:directional_values_mzf}}
\label{subsection:proof_theorem_value}

From Corollary \ref{cor:H_entire}, we know that $H(\mathbf{c},\mathbf{x},\underset{\bm{\theta}}{-\mathbf{N}},t)=0$. Therefore, the only contribution to the special value $\zeta(\mathbf{c},\mathbf{x},\underset{\bm{\theta}}{-\mathbf{N}})$ comes from the term $F(\mathbf{c},\mathbf{x},\underset{\bm{\theta}}{-\mathbf{N}},t)$. Using Proposition \ref{prop:analytic_continuation_directional_F}, we obtain
\begin{align}
    F(\mathbf{c},\mathbf{x},&\underset{\bm{\theta}}{-\mathbf{N}},t) = \sum_{\substack{\mathcal{A} \subseteq [\![1,r]\!] \\ 1 \leq j \leq m}} \prod_{p=1}^{\beta} N_{b_p}! \sum_{k_1,\ldots,k_{\alpha}\geq 0} (-1)^{|\mathbf{k}|} f_{j,\mathcal{A},\mathbf{k}}(\mathbf{c},-\mathbf{N}) \prod_{p=1}^{\alpha} \frac{\zeta(-N_{a_p}-k_p,x_{a_p})}{k_p!} \nonumber \\
    & \times \lim_{z \to 0}\frac{t^{|\mathbf{k}|-\beta-|\mathbf{N}|_{|\mathcal{B}\cup [\![ r+1,r+m ]\!]}+z|\bm{\theta}|_{|\mathcal{B}\cup [\![ r+1,r+m ]\!]}}}{\Gamma(-N_{r+j}+z\theta_{r+j})(|\mathbf{k}|-\beta-|\mathbf{N}|_{|\mathcal{B}\cup [\![ r+1,r+m ]\!]}+z|\bm{\theta}|_{|\mathcal{B}\cup [\![ r+1,r+m ]\!]})}. \label{eq:directional_formula_F_proof_thA}
\end{align}

Let $\mathcal{A}=\{ a_1,\ldots,a_{\alpha} \} \subseteq [\![1,r]\!]$, $j \in [\![ 1,m ]\!]$, $\mathbf{k}=(k_1,\ldots,k_{\alpha}) \in \mathbb{Z}_{\geq 0}^{\alpha}$. By Lemma \ref{lem:inverse_gamma_formula} we find that
\begin{align}
    \lim_{z \to 0}&\frac{1}{\Gamma(-N_{r+j}+z\theta_{r+j})(|\mathbf{k}|-\beta-|\mathbf{N}|_{|\mathcal{B}\cup [\![ r+1,r+m ]\!]}+z|\bm{\theta}|_{|\mathcal{B}\cup [\![ r+1,r+m ]\!]})} \nonumber \\
    &\qquad \qquad \qquad \qquad =
    \begin{cases}
        0 & \text{ if } |\mathbf{k}| \neq \beta+|\mathbf{N}|_{|\mathcal{B}\cup [\![ r+1,r+m ]\!]} \\
        \frac{(-1)^{N_{r+j}}\theta_{r+j}N_{r+j}!}{|\bm{\theta}|_{|\mathcal{B}\cup [\![ r+1,r+m ]\!]}} & \text{ otherwise.}
    \end{cases}
    \label{eq:computation_limit_th_A}
\end{align}
In particular, if $\mathcal{A}=\emptyset$, then $\mathbf{k}$ corresponds to the empty tuple $\emptyset$, and so $|\mathbf{k}|=0$. Since $0<\beta=m$, we conclude that the limit on the left-hand side of (\ref{eq:computation_limit_th_A}) vanishes when $\mathcal{A}=\emptyset$.

Let 
\begin{equation}
    C^0_{j,\mathcal{A},\mathbf{k}}(\mathbf{c},\mathbf{N}) := (-1)^{N_{r+1}+\cdots+\widehat{N_{r+j}}+\cdots+N_{r+m}} N_{r+j}! f_{j,\mathcal{A},\mathbf{k}}(-\mathbf{N}). \label{eq:def_C_0}
\end{equation}
Using (\ref{eq:computation_limit_th_A}) and splitting the sum on the right-hand side of (\ref{eq:directional_formula_F_proof_thA}) as follows 
$$ \sum_{k_1,\ldots,k_{\alpha} \geq 0} = \sum_{k_1+\cdots+k_{\alpha} = \beta+|\mathbf{N}|_{|\mathcal{B}\cup [\![ r+1,r+m ]\!]}} + \sum_{k_1+\cdots+k_{\alpha} \neq \beta+|\mathbf{N}|_{|\mathcal{B}\cup [\![ r+1,r+m ]\!]}}, $$
we obtain Theorem \ref{th:directional_values_mzf}. \qed

Note that the definition of the coefficient $C^0_{j,\mathcal{A},\mathbf{k}}(\mathbf{c},\mathbf{N})$ simplifies the expression for $\zeta(\mathbf{c},\mathbf{x},\underset{\bm{\theta}}{-\mathbf{N}})$ in Theorem \ref{th:directional_values_mzf}. The computation of this coefficient will be carried out in §\ref{subsection:computation_C_0}.

\subsection{Computation of \texorpdfstring{$F'(\mathbf{c},\mathbf{x},\underset{\bm{\theta}}{-\mathbf{N}},t)$}{directionalderivativef}}
\label{subsection:computation_F_prime}

One can evaluate at $z=0$ the derivative with respect to $z$ termwise, via (\ref{eq:analytic_continuation_directional_F}).

\begin{proposition}
\label{prop:formule_F_prime}
There exists a function with no constant term $F^*(\mathbf{c},\mathbf{x},\underset{\bm{\theta}}{-\mathbf{N}},t) = \mu \ln t + \sum_{n=\ell}^{+\infty} \lambda_n t^{n}$ defined for all $0 < t \ll 1$ where $\ell \in \mathbb{Z}_{\leq 0}$, such that:
\begin{align}
    F'&(\mathbf{c},\mathbf{x},\underset{\bm{\theta}}{-\mathbf{N}},t) = F^*(\mathbf{c},\mathbf{x},\underset{\bm{\theta}}{-\mathbf{N}},t) \label{eq:directional_derivative_formula_F_s0} \\
    & +\sum_{\substack{\emptyset \neq \mathcal{A} \subseteq [\![1,r]\!] \\ 1 \leq j \leq m}} \frac{(-1)^{|\mathbf{N}|_{|\mathcal{B}}+ \beta}\theta_{r+j}}{|\bm{\theta}|_{|\mathcal{B}\cup[\![r+1,r+m]\!]}} \prod_{p=1}^{\beta} N_{b_p}! \sum_{\substack{k_1+\cdots+k_{\alpha} \\ = \beta + |\mathbf{N}|_{|\mathcal{B}\cup [\![ r+1,r+m ]\!]}}} \sum_{i=1}^{\alpha} \prod_{\substack{p=1 \\ p \neq i}}^{\alpha} \frac{\zeta(-N_{a_p}-k_p,x_{a_p})}{k_p!} \nonumber \\
    & \qquad \qquad \times \left(\frac{1}{\alpha} C^1_{j,\mathcal{A},\mathbf{k}}(\mathbf{c},\underset{\bm{\theta}}{\mathbf{N}}) \frac{\zeta(-N_{a_i}-k_i,x_{a_i})}{k_i!} + \theta_{a_i} C^0_{j,\mathcal{A},\mathbf{k}}(\mathbf{c},\mathbf{N}) \frac{\zeta'(-N_{a_i}-k_i,x_{a_i})}{k_i!} \right) \nonumber
\end{align}
where $C^0_{j,\mathcal{A},\mathbf{k}}(\mathbf{c},\mathbf{N})$ (resp. $C^1_{j,\mathcal{A},\mathbf{k}}(\mathbf{c},\underset{\bm{\theta}}{\mathbf{N}})$) is as defined in (\ref{eq:def_C_0}) (resp. in (\ref{eq:def_C_1})).
\end{proposition}

\begin{proof}
Differentiating (\ref{eq:analytic_continuation_directional_F}) with respect to $z$ and taking the limit at $z=0$, we obtain
\begin{align}
    F'(\mathbf{c},\mathbf{x},\underset{\bm{\theta}}{-\mathbf{N}},t) =& \sum_{\substack{\mathcal{A} \subseteq [\![1,r]\!] \\ 1 \leq j \leq m}} \prod_{p=1}^{\beta} N_{b_p}! \sum_{k_1,\ldots,k_{\alpha}\geq 0} (-1)^{|\mathbf{k}|} t^{|\mathbf{k}|-\beta-|\mathbf{N}|_{|\mathcal{B}\cup [\![ r+1,r+m ]\!]}} \label{eq:formula_F_prime_computations} \\
    & \qquad \qquad \times \big(T^{(1)}_{j,\mathcal{A},\mathbf{k}}+T^{(2)}_{j,\mathcal{A},\mathbf{k}}+T^{(3)}_{j,\mathcal{A},\mathbf{k}} \ln t\big), \nonumber
\end{align}
with 
\begin{align*}
    T^{(1)}_{j,\mathcal{A},\mathbf{k}}=& \lim_{z \to 0} \left( \frac{1}{\Gamma(-N_{r+j}+z\theta_{r+j})(|\mathbf{k}|-\beta-|\mathbf{N}|_{|\mathcal{B}\cup [\![ r+1,r+m ]\!]}+z|\bm{\theta}|_{|\mathcal{B}\cup [\![ r+1,r+m ]\!]})} \right) \\
    & \quad \times \sum_{i=1}^{\alpha} \Biggl[ \frac{\zeta(-N_{a_i}-k_i,x_{a_i})}{k_{i}!} \left( f'_{j,\mathcal{A},\mathbf{k}}(\mathbf{c},\underset{\bm{\theta}}{-\mathbf{N}}) + \sum_{p=1}^{\beta} \theta_{b_{p}}(h_{N_{b_p}}-\gamma) f_{j,\mathcal{A},\mathbf{k}}(\mathbf{c},-\mathbf{N}) \right) \\
    & \qquad \qquad \qquad + \theta_{a_i}\frac{\zeta'(-N_{a_i}-k_i,x_{a_i})}{k_{i}!} f_{j,\mathcal{A},\mathbf{k}}(\mathbf{c},-\mathbf{N}) \Biggr] \prod_{\substack{p=1 \\ p \neq i}}^{\alpha} \frac{\zeta(-N_{a_p}-k_p,x_{a_p})}{k_p!} \\
    T^{(2)}_{j,\mathcal{A},\mathbf{k}}=& \lim_{z \to 0} \partial_z \left( \frac{1}{\Gamma(-N_{r+j}+z\theta_{r+j})(|\mathbf{k}|-\beta-|\mathbf{N}|_{|\mathcal{B}\cup [\![ r+1,r+m ]\!]}+z|\bm{\theta}|_{|\mathcal{B} \cup [\![ r+1,r+m ]\!]})} \right) \\
    & \quad \times f_{j,\mathcal{A},\mathbf{k}}(\mathbf{c},-\mathbf{N}) \prod_{p=1}^{\alpha} \frac{\zeta(-N_{a_p}-k_p,x_{a_p})}{k_p!} \\
    T^{(3)}_{j,\mathcal{A},\mathbf{k}}=& \lim_{z \to 0} \left( \frac{1}{\Gamma(-N_{r+j}+z\theta_{r+j})(|\mathbf{k}|-\beta-|\mathbf{N}|_{|\mathcal{B}\cup [\![ r+1,r+m ]\!]}+z|\bm{\theta}|_{|\mathcal{B}\cup [\![ r+1,r+m ]\!]})} \right) \\
    & \quad \times |\bm{\theta}|_{|\mathcal{B}\cup [\![ r+1,r+m ]\!]} f_{j,\mathcal{A},\mathbf{k}}(\mathbf{c},-\mathbf{N}) \prod_{p=1}^{\alpha} \frac{\zeta(-N_{a_p}-k_p,x_{a_p})}{k_p!}.
\end{align*}

It remains to compute the limits for each term. Let us fix $\mathcal{A} = \{ a_1,\ldots,a_{\alpha} \} \subseteq [\![ 1,r ]\!]$, $j \in [\![ 1,m ]\!]$, $\mathcal{B} = \{ b_1,\ldots,b_{\beta} \} := [\![ 1,r ]\!] \setminus \mathcal{A}$, and $\mathbf{k}=(k_1,\ldots,k_{\alpha})$. Thanks to Lemma \ref{lem:inverse_gamma_formula}, we have
\begin{align*}
    &\frac{1}{\Gamma(-N_{r+j}+z\theta_{r+j})(|\mathbf{k}|-\beta-|\mathbf{N}|_{|\mathcal{B}\cup [\![ r+1,r+m ]\!]}+z|\bm{\theta}|_{|\mathcal{B}\cup [\![ r+1,r+m ]\!]})} \\
    & \qquad \underset{z \to 0}{=}
    \begin{cases}
        &\frac{(-1)^{N_{r+j}}\theta_{r+j}N_{r+j}!}{|\mathbf{k}|-\beta-|\mathbf{N}|_{|\mathcal{B}\cup [\![ r+1,r+m ]\!]}}z+O(z^2) \qquad \text{if } |\mathbf{k}| \neq \beta+|\mathbf{N}|_{|\mathcal{B}\cup [\![ r+1,r+m ]\!]} \\
        &\frac{(-1)^{N_{r+j}}\theta_{r+j}N_{r+j}!}{|\bm{\theta}|_{|\mathcal{B}\cup [\![ r+1,r+m ]\!]})} + \frac{(-1)^{N_{r+j}}\theta_{r+j}^2N_{r+j}!}{|\bm{\theta}|_{|\mathcal{B}\cup [\![ r+1,r+m ]\!]})}\left(\gamma-h_{N_{r+j}} \right)z + O(z^2) \qquad \text{else.}
    \end{cases}
\end{align*}
Therefore, many terms actually vanish depending on an equality satisfied by $|\mathbf{k}|$.
\begin{enumerate}
\item[i)] If $|\mathbf{k}| \neq \beta+|\mathbf{N}|_{|\mathcal{B}\cup [\![ r+1,r+m ]\!]}$ we find that $T^{(1)}_{j,\mathcal{A},\mathbf{k}}=T^{(3)}_{j,\mathcal{A},\mathbf{k}}=0$. In particular, if $\mathcal{A}=\emptyset$ then $\mathbf{k}$ corresponds to the empty tuple $\emptyset$, thus $|\mathbf{k}|=0<\beta=m$ and then $T^{(1)}_{j,\emptyset,\mathbf{k}}=T^{(3)}_{j,\emptyset,\mathbf{k}}=0$.
\item[ii)] If $|\mathbf{k}| = \beta+|\mathbf{N}|_{|\mathcal{B}\cup [\![ r+1,r+m ]\!]}$, we find
\begin{align}
    &T^{(1)}_{j,\mathcal{A},\mathbf{k}}=\frac{(-1)^{N_{r+j}}\theta_{r+j}N_{r+j}!}{|\bm{\theta}|_{|\mathcal{B}\cup [\![ r+1,r+m ]\!]})} \label{eq:formula_T_1} \\
    & \quad \sum_{i=1}^{\alpha} \Biggl[ \left( f'_{j,\mathcal{A},\mathbf{k}}(\mathbf{c},\underset{\bm{\theta}}{-\mathbf{N}}) + f_{j,\mathcal{A},\mathbf{k}}(\mathbf{c},-\mathbf{N})\sum_{p=1}^{\beta} \theta_{b_{p}}(h_{N_{b_p}}-\gamma) \right) \frac{\zeta(-N_{a_i}-k_i,x_{a_i})}{k_{i}!} \nonumber \\
    & \qquad \qquad \qquad + \theta_{a_i} f_{j,\mathcal{A},\mathbf{k}}(\mathbf{c},-\mathbf{N}) \frac{\zeta'(-N_{a_i}-k_i,x_{a_i})}{k_{i}!} \Biggr] \prod_{\substack{p=1 \\ p \neq i}}^{\alpha} \frac{\zeta(-N_{a_p}-k_p,x_{a_p})}{k_p!} \nonumber \\
    &T^{(2)}_{j,\mathcal{A},\mathbf{k}}=\frac{(-1)^{N_{r+j}}\theta_{r+j}^2N_{r+j}!}{|\bm{\theta}|_{|\mathcal{B}\cup [\![ r+1,r+m ]\!]})}\left(\gamma-h_{N_{r+j}} \right) f_{j,\mathcal{A},\mathbf{k}}(\mathbf{c},-\mathbf{N}) \prod_{p=1}^{\alpha} \frac{\zeta(-N_{a_p}-k_p,x_{a_p})}{k_p!}. \label{eq:formula_T_2}
\end{align}
\end{enumerate}

Using i) and ii), and splitting the series $\sum_{k_1,\ldots,k_{\alpha} \geq 0}$ in (\ref{eq:formula_F_prime_computations}) we obtain
\begin{align}
    F'(\mathbf{c},\mathbf{x},\underset{\bm{\theta}}{-\mathbf{N}},t) = F^*(\mathbf{c},\mathbf{x},\underset{\bm{\theta}}{-\mathbf{N}},t&) + \sum_{\substack{\emptyset \neq \mathcal{A} \subseteq [\![1,r]\!] \\ 1 \leq j \leq m}} (-1)^{\beta+|\mathbf{N}|_{|\mathcal{B}\cup [\![ r+1,r+m ]\!]}} \prod_{p=1}^{\beta} N_{b_p}! \label{eq:formula_F_prime_computations_bis} \\
    & \ \times \sum_{k_1+\cdots+k_{\alpha}=\beta+|\mathbf{N}|_{|\mathcal{B}\cup [\![ r+1,r+m ]\!]}} \big(T^{(1)}_{j,\mathcal{A},\mathbf{k}}+T^{(2)}_{j,\mathcal{A},\mathbf{k}}\big) \nonumber
\end{align}
where
\begin{align*}
    &F^*(\mathbf{c},\mathbf{x},\underset{\bm{\theta}}{-\mathbf{N}},t) := \sum_{\substack{\emptyset \neq \mathcal{A} \subseteq [\![1,r]\!] \\ 1 \leq j \leq m}} (-1)^{\beta+|\mathbf{N}|_{|\mathcal{B}\cup [\![ r+1,r+m ]\!]}} \prod_{p=1}^{\beta} N_{b_p}! \\
    & \qquad \qquad \qquad \times \sum_{k_1+\cdots+k_{\alpha}=\beta+|\mathbf{N}|_{|\mathcal{B}\cup [\![ r+1,r+m ]\!]}} T^{(3)}_{j,\mathcal{A},\mathbf{k}} \ln t \\
    & + \sum_{\substack{\mathcal{A} \subseteq [\![1,r]\!] \\ 1 \leq j \leq m}} \prod_{p=1}^{\beta} N_{b_p}! \sum_{k_1+\cdots+k_{\alpha} \neq \beta+|\mathbf{N}|_{|\mathcal{B}\cup [\![ r+1,r+m ]\!]}} (-1)^{|\mathbf{k}|} T^{(2)}_{j,\mathcal{A},\mathbf{k}} t^{|\mathbf{k}|-\beta-|\mathbf{N}|_{|\mathcal{B}\cup [\![ r+1,r+m ]\!]}}.
\end{align*}
By setting 
\begin{align}
    C_{j,\mathcal{A},\mathbf{k}}^1(\mathbf{c},\underset{\bm{\theta}}{\mathbf{N}})&:=N_{r+j}! \prod^m_{\substack{q=1 \\ q \neq j}}(-1)^{N_{r+q}} \label{eq:def_C_1} \\
    & \times \left[ f'_{j,\mathcal{A},\mathbf{k}}(\mathbf{c},\underset{\bm{\theta}}{-\mathbf{N}}) + f_{j,\mathcal{A},\mathbf{k}}(\mathbf{c},-\mathbf{N}) \left( \theta_{r+j} (h_{N_{r+j}}-\gamma) + \sum_{p=1}^{\beta} \theta_{b_{p}}(h_{N_{b_p}}-\gamma) \right) \right] \nonumber \\
    &= (-1)^{N_{r+1}+\cdots+\widehat{N_{r+j}}+\cdots+N_{r+m}} N_{r+j}! f'_{j,\mathcal{A},\mathbf{k}}(\mathbf{c},\underset{\bm{\theta}}{-\mathbf{N}}) \label{eq:def_C_1_bis} \\
    & \quad + C_{j,\mathcal{A},\mathbf{k}}^0(\mathbf{c},\underset{\bm{\theta}}{\mathbf{N}}) \left( \theta_{r+j} (h_{N_{r+j}}-\gamma) + \sum_{p=1}^{\beta} \theta_{b_{p}}(h_{N_{b_p}}-\gamma) \right), \nonumber	
\end{align}
and by plugging formulas (\ref{eq:formula_T_1}) and (\ref{eq:formula_T_2}) into (\ref{eq:formula_F_prime_computations_bis}) we obtain the required formula.
\end{proof}

\subsection{Computation of \texorpdfstring{$H'(\mathbf{c},\mathbf{x},\protect \underset{\bm{\theta}}{-\mathbf{N}},t)$}{directionalderivativeh}}
\label{subsection:computation_H_prime}

By (\ref{eq:def_H_primaire}), we have
\begin{align}
    H(\mathbf{c},\mathbf{x},-\mathbf{N}+z\bm{\theta},t) =& \sum_{\substack{\emptyset \neq \mathcal{A} \subseteq [\![1,m]\!] \\ \mathcal{C} \subseteq \mathcal{A}^c}} \frac{(-1)^{|\mathcal{A}^c \setminus \mathcal{C}|}}{\prod_{q \in [\![1,m]\!] \setminus \mathcal{C}} \Gamma(-N_{r+q}+z\theta_{r+q})} \label{eq:def_H_direction} \\
    & \times \sum_{n_1,\ldots,n_r \geq 0} \frac{\prod_{q \in [\![ 1,m ]\!] \setminus \mathcal{C}} \Gamma(-N_{r+q}+z\theta_{r+q},t,l_q(\mathbf{n}+\mathbf{x}))}{\prod_{p=1}^{r} (n_p+x_p)^{-N_p+z\theta_p} \prod_{q \in \mathcal{C}} l_q(\mathbf{n}+\mathbf{x})^{-N_{r+q}+z\theta_{r+q}}} \nonumber
\end{align}
for any $t>0$ and any $z \in \mathbb{C}$. Thanks to Proposition \ref{prop:normally_convergent_incomplete_gamma}, we know that all the series above are normally convergent with respect to $z$ over any compact subset of $\mathbb{C}$. In particular, we can differentiate the right-hand side of (\ref{eq:def_H_direction}) termwise with respect to $z$. By taking the limit at $z=0$ of such derivatives, many terms vanish because $\frac{1}{\Gamma}$ has a zero of order $1$ at nonpositive integers. Therefore, we obtain the following relation
\begin{align*}
    H'(\mathbf{c}&,\mathbf{x},\underset{\bm{\theta}}{-\mathbf{N}},t) = \sum_{j=1}^{m} \lim_{z \to 0} \partial_z \left( \frac{1}{\Gamma(-N_{r+j}+z\theta_{r+j})} \right) \nonumber \\
    & \times \sum_{n_1,\ldots,n_r \geq 0} \Gamma(-N_{r+j},t,l_j(\mathbf{n}+\mathbf{x})) \prod_{p=1}^{r} (n_p+x_p)^{N_p} \prod_{\substack{q=1 \\ q \neq j}}^m (l_q(\mathbf{n}+\mathbf{x}))^{N_{r+q}}. \nonumber
\end{align*}
By Lemma \ref{lem:inverse_gamma_formula} and by Newton's multinomial, we get
\begin{align}
    H'(\mathbf{c}&,\mathbf{x},\underset{\bm{\theta}}{-\mathbf{N}},t) = \sum_{j=1}^m \theta_{r+j} \sum^{\widehat{j}}_{\substack{u_{1,1}+\cdots+u_{1,r}=N_{r+1} \\ \cdots \\ u_{m,1}+\cdots+u_{m,r}=N_{r+m}}} \prod_{\substack{ q=1 \\ q \neq j}}^{m} \left( \binom{N_{r+q}}{\mathbf{u}_{q \bullet}} \prod_{p=1}^{r} c_{q,p}^{u_{q,p}} \right) \label{eq:formula_H_prime_series} \\
    & \times (-1)^{N_{r+j}}N_{r+j}! \sum_{n_1,\ldots,n_{r} \geq 0} \Gamma(-N_{r+j},t,l_j(\mathbf{n}+\mathbf{x})) \prod_{p=1}^{r} (n_p+x_p)^{N_p+|\mathbf{u}_{\bullet p}|}. \nonumber
\end{align}
We observe that this last expression is not very explicit due to the presence of the series on the right-hand side of (\ref{eq:formula_H_prime_series}). In §\ref{subsubsection:crandall_expansion_barnes}, we will show via another Crandall's expansion that this series can be written as a derivative value of a generalized Barnes zeta function, plus an explicit constant term.

\subsubsection{Crandall's expansion for generalized Barnes zeta function}
\label{subsubsection:crandall_expansion_barnes}

We fix $\mathbf{M}=(M_1,\ldots,M_r) \in \mathbb{Z}_{\geq 0}^r$, $j \in [\![ 1,m ]\!]$. By applying Corollary \ref{cor:crandall_expansion_zeta}, Proposition \ref{prop:analytic_continuation_F_bis} and (\ref{eq:def_H_primaire}) with suitable data for the generalized Barnes zeta function
$$ \zeta_B(s,\mathbf{M},\mathbf{x}|\mathbf{c}_{j\bullet}) = \sum_{n_1,\ldots,n_r \geq 0} \frac{(n_1+x_1)^{M_1}\cdots(n_r+x_r)^{M_r}}{(c_{j,1}(n_1+x_1)+\cdots+c_{j,r}(n_r+x_r))^s} \qquad (\Re(s) \gg 1), $$
we obtain the following Crandall's expansion for the generalized Barnes zeta function $\zeta_B(s,\mathbf{M},\mathbf{x}|\mathbf{c}_{j\bullet})$:
\begin{equation}
    \label{eq:crandall_expansion_generalized_barnes}
    \zeta_B(s,\mathbf{M},\mathbf{x}|\mathbf{c}_{j\bullet}) = H_{\mathbf{M}}(\mathbf{c}_{j \bullet},\mathbf{x},s,t)+F_{\mathbf{M}}(\mathbf{c}_{j \bullet},\mathbf{x},s,t), 
\end{equation}
where $t>0$ is sufficiently small, and
\begin{align}
    H_{\mathbf{M}}(\mathbf{c}_{j \bullet},\mathbf{x},s,t)& := \frac{1}{\Gamma(s)} \sum_{n_1,\ldots,n_{r} \geq 0} \Gamma(s,t,l_j(\mathbf{n}+\mathbf{x})) \prod_{p=1}^{r} (n_p+x_p)^{M_p} \qquad (s \in \mathbb{C}), \label{eq:h_generalized_barnes_formula} \\
    F_{\mathbf{M}}(\mathbf{c}_{j \bullet},\mathbf{x},s,t)& := \sum_{\mathcal{A} \subseteq [\![1,r]\!]} \prod_{p=1}^{\beta} \frac{M_{b_p}!}{c_{j,b_p}^{M_{b_p}+1}} \sum_{k_1,\ldots,k_{\alpha} \geq 0} (-1)^{|\mathbf{k}|} \prod_{p=1}^{\alpha} \frac{c_{j,a_p}^{k_p}\zeta(-M_{a_p}-k_p,x_{a_p})}{k_p!} \label{eq:j_generalized_barnes_formula} \\
    & \times \frac{t^{s-|\mathbf{M}|_{|\mathcal{B}}-\beta+|\mathbf{k}|}}{\Gamma(s)(s-|\mathbf{M}|_{|\mathcal{B}}-\beta+|\mathbf{k}|)}  \qquad \qquad \qquad (s \neq 1,\ldots,|\mathbf{M}|+r). \nonumber
\end{align}
We denote by $H'_{\mathbf{M}}$ and $F'_{\mathbf{M}}$ the derivatives with respect to $s$ of the two functions defined above. We now give an expression for both of these functions at nonpositive integers.

\begin{proposition}
\label{prop:derivative_values_generalized_barnes}
The function $s \mapsto \zeta_B(s,\mathbf{M},\mathbf{x}|\mathbf{c}_{j\bullet})$ is meromorphic over $\mathbb{C}$, and is regular at nonpositive integers. Moreover, there exists a function with no constant term $F^*_{\mathbf{M}}(\mathbf{c}_{j \bullet},\mathbf{x},-N_{r+j},t) = \mu \ln t + \sum_{n=\ell}^{+\infty} \lambda_n t^{n}$ defined for all $0 < t \ll 1$ where $\ell \in \mathbb{Z}_{\leq 0}$, such that
\begin{align}
    F'_{\mathbf{M}}(\mathbf{c}&_{j \bullet},\mathbf{x},-N_{r+j},t) = F^*_{\mathbf{M}}(\mathbf{c}_{j \bullet},\mathbf{x},-N_{r+j},t) +N_{r+j}! (\gamma-h_{N_{r+j}}) \label{eq:derivative_values_f_generalized_barnes} \\
    & \times \sum_{\emptyset \neq \mathcal{A} \subseteq [\![1,r]\!]} \prod_{p=1}^{\beta} \frac{M_{b_p}!}{(-c_{j,b_p})^{M_{b_p}+1}} \sum_{k_1+\cdots+k_{\alpha}=N_{r+j}+|\mathbf{M}|_{|\mathcal{B}}+\beta} \prod_{p=1}^{\alpha} \frac{c_{j,a_p}^{k_p}\zeta(-M_{a_p}-k_p,x_{a_p})}{k_p!} \nonumber \\
    H'_{\mathbf{M}}(\mathbf{c}&_{j \bullet},\mathbf{x},-N_{r+j},t) = (-1)^{N_{r+j}} N_{r+j}! \sum_{n_1,\ldots,n_r \geq 0} \Gamma(-N_{r+j},t,l_j(\mathbf{n}+\mathbf{x})) \prod_{p=1}^{r} (n_p+x_p)^{M_p}. \label{eq:derivative_values_h_generalized_barnes}
\end{align}
\end{proposition}

\begin{proof}
Its meromorphic continuation is given by (\ref{eq:crandall_expansion_generalized_barnes}). Differentiating (\ref{eq:h_generalized_barnes_formula}) with respect to $s$, and using Lemma \ref{lem:inverse_gamma_formula}, we get (\ref{eq:derivative_values_h_generalized_barnes}). The formula (\ref{eq:derivative_values_f_generalized_barnes}) follows from similar arguments as in the proof of Proposition \ref{prop:formule_F_prime}.
\end{proof}

\subsubsection{Explicit formula for \texorpdfstring{$H'(\mathbf{c},\mathbf{x},\underset{\bm{\theta}}{-\mathbf{N}},t)$}{directionalderivativeh}}

Via Proposition \ref{prop:derivative_values_generalized_barnes} we eliminate the non-explicit series on the right-hand side of (\ref{eq:formula_H_prime_series})
\begin{align}
    H'(\mathbf{c},\mathbf{x},\underset{\bm{\theta}}{-\mathbf{N}},t) = \sum_{j=1}^m & \theta_{r+j} \sum^{\widehat{j}}_{\substack{u_{1,1}+\cdots+u_{1,r}=N_{r+1} \\ \cdots \\ u_{m,1}+\cdots+u_{m,r}=N_{r+m}}} \prod_{\substack{ q=1 \\ q \neq j}}^{m} \left( \binom{N_{r+q}}{\mathbf{u}_{q \bullet}} \prod_{p=1}^{r} c_{q,p}^{u_{q,p}} \right) \label{eq:formula_H_prime_computation_bis} \\
    & \times \left( \zeta_B'(-N_{r+j},\mathbf{M}(\mathbf{N},j,\mathbf{u}),\mathbf{x}|\mathbf{c}) - F'_{\mathbf{M}(\mathbf{N},j,\mathbf{u})}(-N_{r+j},\mathbf{c}_{j \bullet},\mathbf{x},t) \right) \nonumber
\end{align}
where $\mathbf{M}(\mathbf{N},j,\mathbf{u})$ is the tuple defined in (\ref{eq:def_tuple_M}). By plugging (\ref{eq:derivative_values_f_generalized_barnes}) into (\ref{eq:formula_H_prime_computation_bis}), we obtain a function with no constant term $H^*_{\mathbf{M}}(\mathbf{c}_{j \bullet},\mathbf{x},-N_{r+j},t) = \mu \ln t + \sum_{n=\ell}^{+\infty} \lambda_n t^{n}$ defined for all $0 < t \ll 1$ where $\ell \in \mathbb{Z}_{\leq 0}$, such that
\begin{align}
    H'&(\mathbf{c},\mathbf{x},\underset{\bm{\theta}}{-\mathbf{N}},t) = H^*(\mathbf{c},\mathbf{x},\underset{\bm{\theta}}{-\mathbf{N}},t) \label{eq:formula_H_prime} \\
    & + \sum_{j=1}^{m} \theta_{r+j} \sum^{\widehat{j}}_{\substack{u_{1,1}+\cdots+u_{1,r}=N_{r+1} \\ \cdots \\ u_{m,1}+\cdots+u_{m,r}=N_{r+m}}} \prod_{\substack{ q=1 \\ q \neq j}}^{m}  \left( \binom{N_{r+q}}{\mathbf{u}_{q\bullet}} \prod_{p=1}^{r} c_{q,p}^{u_{q,p}} \right) \Bigg[ \zeta_B'(-N_{r+j},\mathbf{M}(\mathbf{N},j,\mathbf{u}),\mathbf{x}|\mathbf{c}_{j\bullet}) \nonumber \\
    & - (\gamma-h_{N_{r+j}}) N_{r+j}! \sum_{\emptyset \neq \mathcal{A} \subseteq [\![1,r]\!]} \prod_{p=1}^{\beta} \frac{(N_{b_p}+|\mathbf{u}_{\bullet b_p}|)!}{(-c_{j,b_p})^{N_{b_p}+|\mathbf{u}_{\bullet b_p}|+1}} \nonumber \\
    & \quad \times \sum_{k_1+\cdots+k_{\alpha}= N_{r+j}+\sum_{p=1}^{\beta} (N_{b_p}+|\mathbf{u}_{\bullet b_p}|+1)} \prod_{p=1}^{\alpha} \frac{c_{j,a_p}^{k_p}\zeta(-N_{a_p}-|\mathbf{u}_{\bullet a_p}|-k_p,x_{a_p})}{k_p!} \Bigg]. \nonumber
\end{align}

\subsection{Proof of Theorem \ref{th:directional_derivative_values_mzf}}
\label{subsection:proof_theorem_derivative_value}

By differentiating Crandall's expansion (\ref{eq:crandall_expansion_zeta}), recall that we obtain
$$ \zeta'(\mathbf{c},\mathbf{x},\underset{\bm{\theta}}{-\mathbf{N}})=F'(\mathbf{c},\mathbf{x},\underset{\bm{\theta}}{-\mathbf{N}},t) + H'(\mathbf{c},\mathbf{x},\underset{\bm{\theta}}{-\mathbf{N}},t), $$
for $t>0$ sufficiently small. A crucial step for the proof of Theorem \ref{th:directional_derivative_values_mzf} relies on the linear independence of $(\ln t,(t^n)_{n \geq \ell})$.

\begin{lemma}
\label{lem:linear_independence_laurent_series}
Let $\ell \in \mathbb{Z}_{\leq 0}$, $(\lambda_n)_{n \in \mathbb{Z}_{\geq \ell}}$ and $(\mu_n)_{n \in \mathbb{Z}_{\geq \ell}}$ be two sequences of complex numbers such that the power series $\sum_{n=0}^{+\infty} \lambda_n y^n$ and $\sum_{n=0}^{+\infty}\mu_n y^n$ have a positive radius of convergence $R$. Let us also assume that
$$ \forall y \in (0,R), \quad \sum_{n=\ell}^{+\infty} \lambda_n y^n + \sum_{n=\ell}^{+\infty} \mu_n y^{n} \ln y = 0. $$
Then for all $n \in \mathbb{Z}_{\geq \ell}$, we have $\lambda_n=0$ and $\mu_n=0$.
\end{lemma}

\begin{proof}
If all $\mu_n$ vanish (resp. all $\lambda_n$ vanish) for all $n \geq \ell$, the result is trivial. Otherwise, let us assume that there exist two integers $i, j \geq \ell$ such that $\lambda_i \neq 0$ and $\mu_j \neq 0$. Suppose $i = \min \{n; \lambda_n \neq 0 \}$ and $j = \min \{n; \mu_n \neq 0 \}$, then $\lambda_i y^i+\mu_j y^j \ln y + o(y^j\ln y)+o(y^i)\underset{y \to 0}{=}0$, thus $\mu_j y^{j-i}\ln y \underset{y \to 0}{\sim} \lambda_i$, hence the contradiction.
\end{proof}

By adding (\ref{eq:directional_derivative_formula_F_s0}) and (\ref{eq:formula_H_prime}), we obtain that $\zeta'(\mathbf{c},\mathbf{x},\underset{\bm{\theta}}{-\mathbf{N}})$ is equal to a constant term (with respect to $t$) that corresponds to the right-hand side of (\ref{eq:zeta_prime_N}), plus a function in the variable $t$ with no constant term $H^*(\mathbf{c},\mathbf{x},\underset{\bm{\theta}}{-\mathbf{N}},t)+F^*(\mathbf{c},\mathbf{x},\underset{\bm{\theta}}{-\mathbf{N}},t)$. This equality holds for all $t>0$ sufficiently small, yet $\zeta'(\mathbf{c},\mathbf{x},\underset{\bm{\theta}}{-\mathbf{N}})$ does not depend on $t$, therefore by Lemma \ref{lem:linear_independence_laurent_series} we obtain for all $t>0$ sufficiently small,
$$ H^*(\mathbf{c},\mathbf{x},\underset{\bm{\theta}}{-\mathbf{N}},t)+F^*(\mathbf{c},\mathbf{x},\underset{\bm{\theta}}{-\mathbf{N}},t)=0. $$
\qed

\section{Computations of \texorpdfstring{$C^0_{\protect \lowercase{j}, \mathcal{A},\mathbf{\lowercase{k}}}(\lowercase{c},\mathbf{N})$}{C0} and \texorpdfstring{$C^1_{\protect \lowercase{j}, \mathcal{A},\mathbf{\lowercase{k}}}(\lowercase{c},\underset{\bm{\theta}}{\mathbf{N}})$}{C1}}
\label{section:computations_C}

Throughout this section, we fix a set $\emptyset \neq \mathcal{A}=\{ a_1,\ldots,a_{\alpha} \} \subseteq [\![ 1,r ]\!]$, a tuple $\mathbf{k}=(k_1,\ldots,k_{\alpha}) \in \mathbb{Z}_{\geq 0}^{\alpha}$, and for convenience we also set $\mathcal{B}=\{ b_1,\ldots,b_{\beta}\} := [\![ 1,r ]\!] \setminus \mathcal{A}$.

We recall that $C^0_{j,\mathcal{A},\mathbf{k}}(\mathbf{c},\mathbf{N})$ (resp. $C^1_{j,\mathcal{A},\mathbf{k}}(\mathbf{c},\underset{\bm{\theta}}{\mathbf{N}})$) is defined in (\ref{eq:def_C_0}) (resp. (\ref{eq:def_C_1})) by the value $f_{j,\mathcal{A},\mathbf{k}}(\mathbf{c},-\mathbf{N})$ (resp. by the values $f_{j,\mathcal{A},\mathbf{k}}(\mathbf{c},-\mathbf{N})$ and $f'_{j,\mathcal{A},\mathbf{k}}(\mathbf{c},\underset{\bm{\theta}}{-\mathbf{N}})$). Also recall that we can derive from (\ref{eq:expression_f_Ajepsilon_s}) two explicit expressions of $f_{j,\mathcal{A},\mathbf{k}}(\mathbf{c},-\mathbf{N})$ and $f'_{j,\mathcal{A},\mathbf{k}}(\mathbf{c},\underset{\bm{\theta}}{-\mathbf{N}})$.

\begin{notation}
Let $\mathcal{A}'=\{ a'_1,\ldots,a'_{\alpha'} \} \subseteq [\![ 1,m ]\!] \setminus \{ j \}$, $\mathcal{B}'=\{ b'_1,\ldots,b'_{\beta'} \} := [\![ 1,m ]\!] \setminus ( \mathcal{A}' \cup \{ j \})$. We consider two tuples 
\begin{align*}
    \mathbf{u}=&(u_{1,1},\ldots,u_{\alpha',1},\ldots,u_{1,\beta},\ldots,u_{\alpha',\beta}), \\
    \mathbf{v}=&(v_{1,1},\ldots,v_{m,1},\ldots,v_{1,\alpha},\ldots,v_{m,\alpha}),
\end{align*}
where every component is a nonnegative integer. We denote
\begin{align}
    T&_{j,\mathcal{A},\mathbf{k},\mathcal{A}',\mathbf{u},\mathbf{v}}(\mathbf{c},\mathbf{s},\varepsilon) := \frac{1}{\prod_{\substack{q=1 \\ q \neq j}}^{m} \Gamma(s_{r+q})}\frac{\varepsilon^{ \sum_{q=1}^{\alpha'} \left(s_{r+a'_q}+|\mathbf{u}_{q \bullet}| + |\mathbf{v}_{q \bullet}|\right)}}{\prod_{q=1}^{\alpha'} (s_{r+a'_q}+ |\mathbf{u}_{q \bullet}| + |\mathbf{v}_{q \bullet}|)} \nonumber \\
    & \quad \times \prod_{p=1}^{\alpha} \left(\binom{k_p}{\mathbf{v}_{\bullet p}} \prod_{q=1}^m c_{q,a_p}^{v_{q,p}} \right) \prod_{p=1}^{\beta} \left( \binom{s_{b_p}-1}{|\mathbf{u}_{\bullet p}|} \binom{|\mathbf{u}_{\bullet p}|}{\mathbf{u}_{\bullet p}} \prod_{q=1}^{\alpha'} c_{a'_q,b_p}^{u_{q,p}} \right) \label{eq:expression_T_Ajepsilon_s} \\
    & \quad \times \int_{\varepsilon}^{1} \cdots \int_{\varepsilon}^1 \prod_{p=1}^{\beta} \left( c_{j,b_p}+\sum_{q=1}^{\beta'} c_{b'_q,b_p} y_{q} \right)^{s_{a_p}-1-|\mathbf{u}_{\bullet p}|} \prod_{q=1}^{\beta'} y_q^{s_{r+b'_q}-1+ |\mathbf{v}_{q \bullet}|} \dd y_1 \cdots \dd y_{\beta'}. \nonumber
\end{align}
For convenience, since we already fixed the data $\mathbf{c}$, $\mathcal{A}$, $j$ and $\mathbf{k}$ for this section, we simply write $T_{\mathcal{A}',\mathbf{u},\mathbf{v}}(\mathbf{s},\varepsilon):=T_{j,\mathcal{A},\mathbf{k},\mathcal{A}',\mathbf{u},\mathbf{v}}(\mathbf{c},\mathbf{s},\varepsilon)$.
\end{notation}

Recall that we proved in Proposition \ref{prop:expression_f_Ajepsilon_s} that the series on the right-hand side of (\ref{eq:expression_f_Ajepsilon_s}) is normally convergent on any compact subset of $\mathbb{C}^{r+m}$ in the variables $(s_1,\ldots,s_{r+m})$. Therefore, by exchanging the limits and series, and by plugging the resulting formulas for $f_{j,\mathcal{A},\mathbf{k}}(\mathbf{c},-\mathbf{N})$ and $f'_{j,\mathcal{A},\mathbf{k}}(\mathbf{c},\underset{\bm{\theta}}{-\mathbf{N}})$ into (\ref{eq:def_C_0}) and (\ref{eq:def_C_1_bis}), we find
\begin{align}
    &C^0_{j,\mathcal{A},\mathbf{k}}(\mathbf{c},\mathbf{N}) = (-1)^{N_{r+1}+\cdots+\widehat{N_{r+j}}+\cdots+N_{r+m}}N_{r+j}! \label{eq:expression_C_0_series} \\
    & \qquad \times \sum_{\mathcal{A}' \subseteq [\![ 1,m ]\!] \setminus \{ j \}} \sum_{\substack{ u_{1,1},\ldots,u_{\alpha',\beta} \geq 0 \\ v_{1,1}+\cdots+v_{m,1}=k_1 \\ \cdots \\ v_{1,\alpha}+\cdots+v_{m,\alpha}=k_{\alpha}}} T_{\mathcal{A}',\mathbf{u},\mathbf{v}}(-\mathbf{N},\varepsilon) \nonumber \\
    &C^1_{j,\mathcal{A},\mathbf{k}}(\mathbf{c},\underset{\bm{\theta}}{\mathbf{N}}) = C^0_{j,\mathcal{A},\mathbf{k}}(\mathbf{c},\mathbf{N})\left( \theta_{r+j} (h_{N_{r+j}}-\gamma) + \sum_{p=1}^{\beta} \theta_{b_{p}}(h_{N_{b_p}}-\gamma)\right) \label{eq:expression_C_1_series} \\
     & \quad + (-1)^{N_{r+1}+\cdots+\widehat{N_{r+j}}+\cdots+N_{r+m}} N_{r+j}! \sum_{\mathcal{A}' \subseteq [\![ 1,m ]\!] \setminus \{ j \}} \sum_{\substack{ u_{1,1},\ldots,u_{\alpha',\beta} \geq 0 \\ v_{1,1}+\cdots+v_{m,1}=k_1 \\ \cdots \\ v_{1,\alpha}+\cdots+v_{m,\alpha}=k_{\alpha}}} T'_{\mathcal{A}',\mathbf{u},\mathbf{v}}(\underset{\bm{\theta}}{-\mathbf{N}},\varepsilon) \nonumber
\end{align}
where we set $T'_{\mathcal{A}',\mathbf{u},\mathbf{v}}(\underset{\bm{\theta}}{-\mathbf{N}},\varepsilon) = \lim_{z \to 0} \partial_z (T_{\mathcal{A}',\mathbf{u},\mathbf{v}}(-\mathbf{N}+z \bm{\theta},\varepsilon))$. We observe that both values $T_{\mathcal{A}',\mathbf{u},\mathbf{v}}(-\mathbf{N},\varepsilon)$ and $T'_{\mathcal{A}',\mathbf{u},\mathbf{v}}(\underset{\bm{\theta}}{-\mathbf{N}},\varepsilon)$ vanish for most of the tuples $\mathbf{u}$ and $\mathbf{v}$.

\begin{lemma}
\label{lem:asymptotic_behavior_summand_h}
Let $\varepsilon \in (0,1)$, we have
\begin{align*}
    T&_{\mathcal{A}',\mathbf{u},\mathbf{v}}(-\mathbf{N}+z\bm{\theta},\varepsilon) \\
    & \underset{z \to 0}{=}
    \begin{cases}
        O(1) & \text{if } \mathcal{A}' =[\![ 1,m ]\!]\setminus \{ j \}, \text{ and } \forall q \in \mathcal{A}', |\mathbf{u}_{q \bullet}|+|\mathbf{v}_{q \bullet}|=N_{r+q} \\ 
        O(z) & \text{if } \mathcal{A}' =[\![ 1,m ]\!]\setminus \{ j \} \text{ and } \exists \ell \in \mathcal{A}' \setminus \{ j \} \text{ s.t. } \\
        & \hspace{48pt} |\mathbf{u}_{\ell \bullet}|+|\mathbf{v}_{\ell \bullet}| \neq N_{r+f} \text{ and } \forall q \in \mathcal{A}' \setminus \{ \ell \}, |\mathbf{u}_{q \bullet}|+|\mathbf{v}_{q \bullet}|=N_{r+q} \\
        O(z) & \text{if } \mathcal{A}' =[\![ 1,m ]\!]\setminus \{ j,\ell \} \text{ and } \forall q \in \mathcal{A}', |\mathbf{u}_{q \bullet}|+|\mathbf{v}_{q \bullet}|=N_{r+q} \\
        O(z^2) & \text{else}.
    \end{cases}
\end{align*}
\end{lemma}

\begin{proof}
Observe that the second and third lines on the right-hand side of (\ref{eq:expression_T_Ajepsilon_s}) define entire functions in the variables $(s_1,\ldots,s_{r+m})$. By substituting $\mathbf{s}=-\mathbf{N}+z\bm{\theta}$ in (\ref{eq:expression_T_Ajepsilon_s}), and recalling that the function $\frac{1}{\Gamma}$ has a zero of order $1$ at nonpositive integers, we find that the function $z \mapsto T_{\mathcal{A}',\mathbf{u},\mathbf{v}}(-\mathbf{N}+z\bm{\theta},\varepsilon)$ has a zero at $z=0$ of order
$$ m-1 - \Big|\{ q \in [\![ 1,\alpha' ]\!] ; |\mathbf{u}_{q \bullet}|+|\mathbf{v}_{q \bullet}| = N_{r+a'_q} \}\Big|. $$
\end{proof}

\subsection{Computation of \texorpdfstring{$C^0_{\protect \lowercase{j},\mathcal{A},\mathbf{\lowercase{k}}}(\mathbf{\lowercase{c}},\mathbf{N})$}{C0}}
\label{subsection:computation_C_0}

Considering (\ref{eq:expression_C_0_series}), it is enough to study the value $T_{\mathcal{A}',\mathbf{u},\mathbf{v}}(-\mathbf{N},\varepsilon)$ to obtain (\ref{eq:formula_C0}). By taking $\mathbf{s}=(-N_1,\ldots,-N_{r+m})$ in (\ref{eq:expression_T_Ajepsilon_s}), we obtain two possible types of values for $T_{\mathcal{A}',\mathbf{u},\mathbf{v}}(-\mathbf{N},\varepsilon)$, depending on whether the tuples $\mathbf{v}$ and $\mathbf{u}$ satisfy certain conditions on their components.

\begin{proposition}
Let $\mathcal{A}' \subseteq [\![ 1,m ]\!] \setminus \{ j \}$.
\begin{itemize}
\item[i)] If $\mathcal{A}'= \{ 1,\ldots,\widehat{j},\ldots,m \}$ and $|\mathbf{u}_{q \bullet}|+|\mathbf{v}_{q \bullet}|=N_{r+q} \ (1 \leq q \neq j \leq m)$, then
\begin{align}
    T&_{\mathcal{A}',\mathbf{u},\mathbf{v}}(-\mathbf{N},\varepsilon) = \prod_{\substack{q=1 \\ q \neq j}}^m \left( (-1)^{N_{r+q}}N_{r+q}!\right) \label{eq:value_T_Ajepsilon_1st_case} \\
    & \prod_{p=1}^{\alpha} \left( \binom{k_p}{\mathbf{v}_{\bullet p}} \prod_{q=1}^{m} c_{q,a_p}^{v_{q,p}} \right) \prod_{p=1}^{\beta} \left( \binom{-N_{b_p}-1}{|\mathbf{u}_{\bullet p}|} \binom{|\mathbf{u}_{\bullet p}|}{\mathbf{u}_{\bullet p}} c_{j,b_p}^{-N_{b_p}-1-|\mathbf{u}_{\bullet p}|} \prod_{\substack{q=1 \\ q \neq j}}^{m} c_{q,b_p}^{u_{q,p}} \right). \nonumber
\end{align}

\item[ii)] Otherwise, we have $T_{\mathcal{A}',\mathbf{u},\mathbf{v}}(-\mathbf{N},\varepsilon) = 0$.
\end{itemize}
\end{proposition}

\begin{proof}
i) Taking the limit of (\ref{eq:expression_T_Ajepsilon_s}) when $\mathbf{s} \to (-N_1,\ldots,-N_{r+m})$, we get an expression of $T_{\mathcal{A}',\mathbf{u},\mathbf{v}}(-\mathbf{N},\varepsilon)$. Note that this expression contains an integral over a domain given by an empty product, so by convention, the corresponding integral is $1$. The rest of the terms are straightforward to compute using Lemma \ref{lem:inverse_gamma_formula}.

ii) The second case follows from Lemma \ref{lem:asymptotic_behavior_summand_h}.
\end{proof}

This last proposition shows that there is only finitely many terms contributing to the series on the right-hand side of (\ref{eq:expression_C_0_series}). By plugging (\ref{eq:value_T_Ajepsilon_1st_case}) into (\ref{eq:expression_C_0_series}), we get (\ref{eq:formula_C0}).

\subsection{Computation of \texorpdfstring{$C^1_{\protect \lowercase{j},\mathcal{A},\mathbf{\lowercase{k}}}(\mathbf{\lowercase{c}},{\protect \underset{\bm{\theta}}{\mathbf{N}}})$}{C1}}

From (\ref{eq:expression_C_1_series}), it is enough to study the derivative value $T'_{\mathcal{A}',\mathbf{u},\mathbf{v}}(\underset{\bm{\theta}}{-\mathbf{N}},\varepsilon)$ to obtain (\ref{eq:formula_C1}). The computations are done in a similar fashion as in §\ref{subsection:computation_C_0}.

\begin{proposition}
\label{prop:computations_general_term_f_derivative}
Let $\mathcal{A}' \subseteq [\![ 1,m ]\!] \setminus \{ j \}$.
\begin{itemize}
\item[i)] If $\mathcal{A}'=\{ 1,\ldots,\widehat{j},\ldots,m \}$ and $|\mathbf{u}_{q \bullet}|+|\mathbf{v}_{q \bullet}|=N_{r+q} \ (1 \leq q \neq j \leq m)$, then there exists $T^*_{\mathcal{A}',\mathbf{u},\mathbf{v}}(\varepsilon)=\lambda \ln \varepsilon$ with $\lambda \in \mathbb{C}$ such that
\begin{align}
    T'_{\mathcal{A}',\mathbf{u},\mathbf{v}}&(\underset{\bm{\theta}}{-\mathbf{N}},\varepsilon) = T^*_{\mathcal{A}',\mathbf{u},\mathbf{v}}(\varepsilon) + \prod_{\substack{q=1 \\ q \neq j}}^{m} \left( (-1)^{N_{r+q}} N_{r+q}! \right) \prod_{p=1}^{\alpha} \left( \binom{k_p}{\mathbf{v}_{\bullet p}} \prod_{q=1}^m c_{q,a_p}^{v_{q,p}} \right) \label{eq:expression_general_term_C_1_i} \\
    & \qquad \times \prod_{p=1}^{\beta} \left( \binom{-N_{a_p}-1}{|\mathbf{u}_{\bullet p}|} \binom{|\mathbf{u}_{\bullet p}|}{\mathbf{u}_{\bullet p}} c_{j,a_p}^{-N_{a_p}-1-|\mathbf{u}_{\bullet p}|} \prod_{\substack{q=1 \\ q \neq j}}^{m} c_{q,a_p}^{u_{q,p}} \right) \nonumber \\
    & \qquad \times \left( \sum_{p=1}^{\beta} \theta_{b_p} (h_{N_{b_p}}-h_{N_{b_p}+|\mathbf{u}_{\bullet p}|} + \ln c_{j,b_p} ) + \sum_{\substack{q=1 \\ q \neq j}}^{m} \theta_{r+q} (\gamma-h_{N_{r+q}}) \right). \nonumber
\end{align}

\item[ii)] If $\mathcal{A}'=\{ 1,\ldots,\widehat{j},\ldots,m \}$, and if there exists an integer $\ell \in \mathcal{A}'$ such that $|\mathbf{u}_{\ell \bullet}|+|\mathbf{v}_{\ell \bullet}| \neq N_{r+\ell}$ and $|\mathbf{u}_{q \bullet}|+|\mathbf{v}_{q \bullet}|=N_{r+q} \ (1 \leq q \neq j,\ell \leq m)$, then $T'_{\mathcal{A}',\mathbf{u},\mathbf{v}}(\underset{\bm{\theta}}{-\mathbf{N}},\varepsilon) \in \mathbb{C}\varepsilon^{|\mathbf{u}_{\ell \bullet}|+|\mathbf{v}_{\ell \bullet}|-N_{r+\ell}}$ (as a polynomial in $\varepsilon$).

\item[iii)] If $\mathcal{A}'=[\![ 1,m ]\!] \setminus \{ j,\ell \}$ and $|\mathbf{u}_{q \bullet}|+|\mathbf{v}_{q \bullet}|=N_{r+q}\ (1 \leq q \neq j,\ell \leq m)$, then there exists a function with no constant term $T^*_{\mathcal{A}',\mathbf{u},\mathbf{v}}(\varepsilon)=\mu \ln \varepsilon +\sum_{n=i}^{+\infty} \lambda_n \varepsilon^n$ defined for all $0<\varepsilon \ll 1$ where $i \in \mathbb{Z}_{\leq 0}$, and such that
\begin{align}
     T'_{\mathcal{A}',\mathbf{u},\mathbf{v}}(\underset{\bm{\theta}}{-\mathbf{N}},&\varepsilon) = T^*_{\mathcal{A}',\mathbf{u},\mathbf{v}}(\varepsilon) \label{eq:expression_general_term_C_1_iii} \\
    & + \theta_{r+\ell} \prod_{\substack{q=1 \\ q \neq j}}^m \left( (-1)^{N_{r+q}} N_{r+q}!\right)\prod_{p=1}^{\alpha} \left( \binom{k_p}{\mathbf{v}_{\bullet p}} \prod_{q=1}^m c_{q,a_p}^{v_{q,p}} \right) \nonumber \\
    & \quad \times \prod_{p=1}^{\beta} \left( \binom{-N_{b_p}-1}{|\mathbf{u}_{\bullet p}|} \binom{|\mathbf{u}_{\bullet p}|}{\mathbf{u}_{\bullet p}} \prod_{\substack{q=1 \\ q \neq j,\ell}}^m c_{q,b_p}^{u_{q,p}} \right) W_{\mathcal{B},j,\ell,\mathbf{n}(\mathbf{N},\mathbf{u}),N_{\ell+r}-|\mathbf{v}_{\ell \bullet}|}(\mathbf{c}) \nonumber
\end{align}
where $W$ corresponds to the constant defined in (\ref{eq:def_coefficient_W}), and where $\mathbf{n}(\mathbf{N},\mathbf{u})$ is the tuple defined in Definition \ref{def:C0_C1}.

\item[iv)] Otherwise, we have $T'_{\mathcal{A}',\mathbf{u},\mathbf{v}}(\underset{\bm{\theta}}{-\mathbf{N}},\varepsilon)=0$.
\end{itemize}
\end{proposition}

\begin{proof}
By (\ref{eq:expression_T_Ajepsilon_s}), we obtain a formula for the derivative $\partial_z T_{\mathcal{A}',\mathbf{u},\mathbf{v}}(-\mathbf{N}+z\bm{\theta},\varepsilon)$. Taking the limit of the resulting formula at $z=0$, and using Lemmata \ref{lem:inverse_gamma_formula}, \ref{lem:derivative_value_binomial}, and \ref{lem:integral_computations_epsilon_1}, we prove the four claims of the proposition.
\end{proof}

\begin{remark}
To simplify notation in Proposition \ref{prop:computations_general_term_f_derivative} we made the following choices:
\begin{enumerate}
\item[1)] In i) and ii), the tuple $\mathbf{u}$ is such that $\mathbf{u}=(u_{q,p})_{\substack{1 \leq q \neq j \leq m \\ 1 \leq p \leq \beta}}$.
\item[2)] In iii), the tuple $\mathbf{u}$ is such that $\mathbf{u}=(u_{q,p})_{\substack{1 \leq q \neq j,\ell \leq m \\ 1 \leq p \leq \beta}}$.
\end{enumerate}
\end{remark}

We now compute $C^1_{j,\mathcal{A},\mathbf{k}}(\mathbf{c},\underset{\bm{\theta}}{\mathbf{N}})$. By (\ref{eq:expression_C_1_series}) and Proposition \ref{prop:computations_general_term_f_derivative}, we have
\begin{align}
    &C^1_{j,\mathcal{A},\mathbf{k}}(\mathbf{c},\underset{\bm{\theta}}{\mathbf{N}}) = C^0_{j,\mathcal{A},\mathbf{k}}(\mathbf{c},\mathbf{N})\left( \theta_{r+j} (h_{N_{r+j}}-\gamma) + \sum_{p=1}^{\beta} \theta_{b_{p}}(h_{N_{b_p}}-\gamma)\right)+G^*(\varepsilon) \label{eq:proof_C_1_formula} \\
     & + \sum_{\substack{ v_{1,1}+\cdots+v_{m,1}=k_1 \\ \cdots \\ v_{1,\alpha}+\cdots+v_{m,\alpha}=k_{\alpha}}} (-1)^{N_{r+1}+\cdots+\widehat{N_{r+j}}+\cdots+N_{r+m}} N_{r+j}! \nonumber \\
     & \qquad \times \left( \sum^{\widehat{j}}_{\substack{u_{1,1}+\cdots+u_{1,\beta}=N_{r+1}-|\mathbf{v}_{1 \bullet}| \\ \cdots \\ u_{m,1}+\cdots+u_{m,\beta}=N_{r+m}-|\mathbf{v}_{m \bullet}|}} + \sum^{\widehat{j,\ell}}_{\substack{u_{1,1}+\cdots+u_{1,\beta}=N_{r+1}-|\mathbf{v}_{1 \bullet}| \\ \cdots \\ u_{m,1}+\cdots+u_{m,\beta}=N_{r+m}-|\mathbf{v}_{m \bullet}|}} \right) \nonumber \\
     & \qquad \qquad \qquad \qquad \qquad \qquad \qquad \qquad \qquad \qquad \qquad \left(T'_{\mathcal{A}',\mathbf{u},\mathbf{v}}(\underset{\bm{\theta}}{-\mathbf{N}},\varepsilon)-T^*_{\mathcal{A}',\mathbf{u},\mathbf{v}}(\underset{\bm{\theta}}{-\mathbf{N}},\varepsilon)\right) \nonumber
\end{align}
where $G^*(\varepsilon)=\mu \ln \varepsilon + \sum_{n \geq i} \lambda_n \varepsilon^n$ is a function with no constant term, defined for $0<\varepsilon \ll 1$, with $i \in \mathbb{Z}_{\leq 0}$. Since $C^1_{j,\mathcal{A},\mathbf{k}}(\mathbf{c},\underset{\bm{\theta}}{\mathbf{N}})$ is independent of $\varepsilon$, it follows from Lemma \ref{lem:linear_independence_laurent_series} that $G^*(\varepsilon)=0$. Therefore the constant term (with respect to $\varepsilon$) of the right-hand side of (\ref{eq:proof_C_1_formula}) is equal to the coefficient $C^1_{j,\mathcal{A},\mathbf{k}}(\mathbf{c},\underset{\bm{\theta}}{\mathbf{N}})$. By replacing $T'_{\mathcal{A}',\mathbf{u},\mathbf{v}}(\underset{\bm{\theta}}{-\mathbf{N}},\varepsilon)-T^*_{\mathcal{A}',\mathbf{u},\mathbf{v}}(\underset{\bm{\theta}}{-\mathbf{N}},\varepsilon)$ with their corresponding values obtained in Proposition \ref{prop:computations_general_term_f_derivative}, we get (\ref{eq:formula_C1}).

\section{Some values and derivative values of Hurwitz and generalized Barnes zeta functions}
\label{section:values_hurwitz_barnes}

In this section, we simplify several of the terms appearing in (\ref{eq:zeta_prime_N}) in specific cases. We start by recalling some well-known formulas for the Hurwitz zeta function in §\ref{subsection:values_particular_hurwitz}. In §\ref{subsection:generalized_barnes_formula}, we prove a formula for the generalized Barnes zeta function when the coefficients are rational numbers.

\subsection{Particular Hurwitz zeta values and derivative values}
\label{subsection:values_particular_hurwitz}

Let us first recall that the values at nonpositive integers of the Hurwitz zeta function are given by Bernoulli polynomials,
$$ \zeta(-N,x)=-\frac{B_{N+1}(x)}{N+1} \qquad (\Re(x)>0, \ N \in \mathbb{Z}_{\geq 0}). $$
As shown in \cite[Formula (12)]{miller1998derivhurwitz}, the Hurwitz zeta function satisfies the following multiplication formula
\begin{equation}
    \sum_{n=0}^{k-1}\zeta\left(s,x+\frac{n}{k}\right) = k^s \zeta(s,kx) \qquad (\Re(x)>0, \ s \neq 1, \ k \in \mathbb{N}). \label{eq:multiplication_formula}
\end{equation}

The derivative values do not have a simple formula in general. For particular cases,
$$ \zeta'(-1,1) = \frac{1}{12}-\ln \glaisher, \qquad \zeta'(0,x) = \ln \Gamma(x)-\frac{\ln(2\pi)}{2} \qquad (\Re(x)>0) $$
where $\glaisher$ is the Glaischer-Kinkelin constant. Also, by writing $\zeta(s,x)=\zeta(s,x+1)+\frac{1}{x^s}$, we obtain
\begin{equation}
    \zeta'(-N,x+1)=\zeta'(-N,x)+x^N \ln x \qquad (\Re(x)>0, \ N \in \mathbb{Z}_{\geq 0}). \label{eq:relation_prime_zeta_hurwitz_x_1}
\end{equation}

\subsection{Explicit formula for generalized Barnes zeta functions with rational coefficients}
\label{subsection:generalized_barnes_formula}

In this subsection, we fix $\mathbf{c}=(c_1,\ldots,c_{r}) \in \mathbb{Q}_{>0}^r$, $\mathbf{x}=(x_1,\ldots,x_r) \in H_0^r$, and $\mathbf{M}=(M_1,\ldots,M_{r}) \in \mathbb{Z}_{\geq 0}^r$. We establish an explicit relation between the generalized Barnes zeta function $\zeta_B(s,\mathbf{M},\mathbf{x}|\mathbf{c})$ and the Hurwitz zeta function $\zeta(s,x)$. To that end, we adapt a strategy used by Aoki and Sakane (see \cite{sakane2022values}), which originally led to formulas for higher derivative values of classical Barnes zeta functions with rational coefficients.

\begin{notation}
Let $c_i=\frac{p_i}{q_i}$ with $p_i>0$ and $q_i>0$ being coprime integers ($1 \leq i \leq r$). We set $\tilde{c}:=\frac{\ppcm(p_1,\ldots,p_r)}{\pgcd(q_1,\ldots,q_r)}$, and $\tilde{c}_i := \frac{\tilde{c}}{c_i} \in \mathbb{N}$ ($1 \leq i \leq r$).
\end{notation}

\begin{proposition}
We have
\begin{equation}
    \zeta_B(s,\mathbf{M},\mathbf{x}|\mathbf{c}) = \frac{\prod_{i=1}^r \tilde{c}_i^{M_i}}{\tilde{c}^{s}} \sum_{\substack{0 \leq u_1 \leq \tilde{c}_1-1 \\ \cdots \\ 0 \leq u_r \leq \tilde{c}_r-1}} \zeta_B\left(s,\mathbf{M},\left(\frac{x_i+u_i}{\tilde{c}_i}\right)_{i \in [\![ 1,r ]\!]} | (1,\ldots,1) \right). \label{eq:formula_zeta_barnes_s}
\end{equation}
\end{proposition}

\begin{proof}
By the analytic continuation principle, it is enough to check that (\ref{eq:formula_zeta_barnes_s}) holds for all $s \in \mathbb{C}$ such that $\Re(s) \gg 1$. Let $q = \ppcm(q_1, \ldots, q_r)$ and $p = \ppcm(q c_1,\ldots,q c_r)$. By \cite[p.3]{sakane2022values}, we know that $\frac{p}{q}=\tilde{c}$. Using the decomposition $\mathbb{Z}_{\geq 0} = \bigsqcup_{u_i=0}^{\tilde{c}_i-1} \left( \tilde{c}_i \mathbb{Z}_{\geq 0}+u_i \right)$ for each $1 \leq i \leq r$, we get
\begin{align*}
    \zeta_B&(s,\mathbf{M},\mathbf{x}|\mathbf{c}) \\
    & =\tilde{c}^{-s} \sum_{\substack{0 \leq u_1 \leq \tilde{c}_1-1 \\ \cdots \\ 0 \leq u_r \leq \tilde{c}_r-1}} \sum_{\substack{n_1 = u_1 \mod \tilde{c}_1 \\ \cdots \\ n_r = u_r \mod \tilde{c}_r}}  \frac{\prod_{i=1}^r (n_i+x_i)^{M_p}}{\left( \frac{q}{p} \sum_{i=1}^r (c_i x_i + c_i u_i + c_i(n_i-u_i)) \right)^s}.
\end{align*}
By a change of variables $n_i \to \tilde{c}_i n_i+u_i$ ($1 \leq i \leq r$), we get
$$ \zeta_B(s,\mathbf{M},\mathbf{x}|\mathbf{c})= \tilde{c}^{-s} \sum_{\substack{0 \leq u_1 \leq \tilde{c}_1-1 \\ \cdots \\ 0 \leq u_r \leq \tilde{c}_r-1}} \sum_{\substack{n_1 \geq 0 \\ \cdots \\ n_r \geq 0}} \frac{\prod_{i=1}^r (\tilde{c}_i n_i+u_i+x_i)^{M_i}}{\left( \frac{q}{p} \sum_{i=1}^r (c_i x_i + c_i u_i + c_i \tilde{c}_i n_i) \right)^s}. $$
By factoring $\tilde{c}_i^{M_i}$ in the numerator, and by noticing that $\frac{q c_i \tilde{c}_i}{p}=1$ ($1 \leq i \leq r$), we get
$$ \zeta_B(s,\mathbf{M},\mathbf{x}|\mathbf{c})= \frac{\prod_{i=1}^{r} \tilde{c}_i^{M_p}}{\tilde{c}^{s}} \sum_{\substack{0 \leq u_1 \leq \tilde{c}_1-1 \\ \cdots \\ 0 \leq u_r \leq \tilde{c}_r-1}} \sum_{\substack{n_1 \geq 0 \\ \cdots \\ n_r \geq 0}} \frac{\prod_{i=1}^r \left(n_i +  \frac{u_i +x_i}{\tilde{c}_i} \right)^{M_i}}{\left( \sum_{i=1}^r \left( n_i + \frac{q}{p} (c_i x_i + c_i u_i) \right) \right) ^s}. $$
Since $\frac{u_i+x_i}{\tilde{c}_i} = \frac{q}{p} (c_i x_i + c_i u_i) \ (1 \leq i \leq r$), we find that (\ref{eq:formula_zeta_barnes_s}) holds when $\Re(s) \gg 1$.
\end{proof}

We now establish a formula for $\zeta_B(s,\mathbf{M},\mathbf{x}|(1,\ldots,1))$, by slightly modifying Onodera's original proof of a formula for certain Barnes zeta functions with numerator terms (see \cite[Proposition 4.1]{onodera2021multiple}).

\begin{proposition}
\label{prop:formula_generalized_barnes_x1}
Let $s \in \mathbb{C} \setminus \mathbb{N}$. We have
\begin{align}
    \zeta_B(s,\mathbf{M},\mathbf{x}|&\mathbf{1}) = \sum_{\emptyset \neq \mathcal{A} \subsetneq [\![1,r]\!]} (-1)^{\beta+|\mathbf{M}|_{|\mathcal{B}}-1} \prod_{i=1}^{\beta} M_{b_i}!  \label{eq:formula_zeta_barnes_1} \\
    & \quad \times \sum_{k'+k_1+\cdots+k_{\alpha} = |\mathbf{M}|_{|\mathcal{B}}+\beta-1} (-1)^{k'} \frac{\zeta(s-k',|\mathbf{x}|)}{k'!} \prod_{i=1}^{\alpha} \frac{\zeta(-M_{a_i}-k_i,x_{a_i})}{k_i!}. \nonumber
\end{align}
\end{proposition}

Before proving this proposition, we first state a lemma provided by Onodera.

\begin{lemma}
\cite[Lemma 4.2]{onodera2021multiple}
\label{lem:generalized_faulhaber}
Let $\ell \in \mathbb{Z}$. We have
\begin{align*}
    & \sum_{\substack{n_1,\ldots,n_r \geq 0 \\ n_1+\cdots+n_r=\ell}} \prod_{i=1}^r (n_i+x_i)^{M_i} + (-1)^{r-1} \sum_{\substack{n_1,\ldots,n_r \leq -1 \\ n_1+\cdots+n_{r}=\ell}} \prod_{i=1}^r (n_i+x_i)^{M_i} \\
    & = \quad \sum_{\emptyset \neq \mathcal{A} \subsetneq [\![ 1,r]\!]} \prod_{i=1}^{\beta} M_{b_i}! \sum_{k'+k_1+\cdots+k_{\alpha}=|\mathbf{M}|_{|\mathcal{B}}+\beta-1} \frac{(|\mathbf{x}|+\ell)^{k'}}{k'!} \prod_{i=1}^{\alpha} \frac{(-1)^{k_i}\zeta(-M_{a_i}-k_i,x_{a_i})}{k_i!}.
\end{align*}
\end{lemma}

\begin{proof}[Proof of Proposition \ref{prop:formula_generalized_barnes_x1}]
On the convergence domain of $\zeta_B(s,\mathbf{M},\mathbf{x}|\mathbf{1})$ we have 
$$ \zeta_B(s,\mathbf{M},\mathbf{x}|\mathbf{1}) = \sum_{n'=0}^{+\infty} \frac{1}{(n'+|\mathbf{x}|)^s} \sum_{\substack{n_1,\ldots,n_r \geq 0 \\ n_1+\cdots+n_r=n'}} \prod_{i=1}^r (n_i+x_i)^{M_i}.$$
By Lemma \ref{lem:generalized_faulhaber}, we get
\begin{align*}
    \zeta_B(s,\mathbf{M},\mathbf{x}|\mathbf{1}) = \sum_{\emptyset \neq \mathcal{A} \subsetneq [\![ 1,r]\!]} \prod_{i=1}^{\beta} M_{b_i}! & \sum_{\substack{k_1+\cdots+k_{\alpha}+k'=|\mathbf{M}|_{|\mathcal{B}}+\beta-1}} \prod_{i=1}^{\alpha} \frac{(-1)^{k_i}\zeta(-M_{a_i}-k_i,x_i)}{k_i!} \\
    & \qquad \qquad \times \sum_{n'=0}^{+\infty} \frac{(n'+|\mathbf{x}|)^{k'}}{k'!(n'+|\mathbf{x}|)^s}.
\end{align*}
Since $\zeta(s-k',|\mathbf{x}|)=\sum_{n'=0}^{+\infty} \frac{(n'+|\mathbf{x}|)^{k'}}{k'!(n'+|\mathbf{x}|)^s}$, we obtain (\ref{eq:formula_zeta_barnes_1}).
\end{proof}

By plugging (\ref{eq:formula_zeta_barnes_1}) into (\ref{eq:formula_zeta_barnes_s}), we finally get
\begin{align*}
    \zeta_B&(s,\mathbf{M},\mathbf{x}|\mathbf{c}) = \frac{\prod_{i=1}^r \tilde{c}_i^{M_i}}{\tilde{c}^{s}} \sum_{\emptyset \neq \mathcal{A} \subsetneq [\![1,r]\!]} (-1)^{\beta+|\mathbf{M}|_{|\mathcal{B}}-1} \prod_{i=1}^{\beta} M_{b_i}!  \\
    & \times \sum_{\substack{0 \leq u_1 \leq \tilde{c}_1-1 \\ \cdots \\ 0 \leq u_r \leq \tilde{c}_r-1}} \sum_{\substack{k'+k_1+\cdots+k_{\alpha} \\ =|\mathbf{M}|_{|\mathcal{B}}+\beta-1}} \frac{(-1)^{k'}\zeta\left(s-k',\frac{\langle\mathbf{c},\mathbf{x}+\mathbf{u}\rangle}{\tilde{c}} \right)}{k'!} \prod_{i=1}^{\alpha} \frac{\zeta \left( -M_{a_i}-k_i,\frac{x_{a_i}+u_{a_i}}{\tilde{c}_{a_i}} \right)}{k_i!},
\end{align*}
where we set $\langle \mathbf{c},\mathbf{x}+\mathbf{u} \rangle:=c_1(x_1+u_1)+\cdots+c_r(x_r+u_r)$. By evaluating the derivative with respect to $s$ at nonpositive integers, we get:

\begin{corollary}
\label{cor:barnes_derivative_formula}
Let $N \in \mathbb{Z}_{\geq 0}$,
\begin{align}
    \zeta_B'&(-N,\mathbf{M},\mathbf{x}|\mathbf{c}) = \tilde{c}^{N} \prod_{i=1}^r \tilde{c}_i^{M_i} \sum_{\substack{\emptyset \neq \mathcal{A} \subsetneq [\![1,r]\!] \\ 0 \leq u_1 \leq \tilde{c}_1-1 \\ \cdots \\ 0 \leq u_r \leq \tilde{c}_r-1}} (-1)^{\beta+|\mathbf{M}|_{|\mathcal{B}}-1} \prod_{i=1}^{\beta} M_{b_i}! \label{eq:zeta_barnes_prime} \\
    & \times \sum_{k'+k_1+\cdots+k_{\alpha}=|\mathbf{M}|_{|\mathcal{B}}+\beta-1} \frac{(-1)^{k'}}{k'!} \prod_{i=1}^{\alpha} \frac{\zeta \left( -M_{a_i}-k_i,\frac{x_{a_i}+u_{a_i}}{\tilde{c}_{a_i}} \right)}{k_i!} \nonumber \\
    & \qquad \times \left[ \zeta'\left(-N-k',\frac{\langle\mathbf{c},\mathbf{x}+\mathbf{u}\rangle}{\tilde{c}}\right)-\zeta\left(-N-k',\frac{\langle\mathbf{c},\mathbf{x}+\mathbf{u}\rangle}{\tilde{c}}\right) \ln \tilde{c} \right]. \nonumber
\end{align}
\end{corollary}

\begin{example}
\label{ex:derivative_values_barnes}
Let us denote by $\glaisher$ the Glaisher-Kinkelin constant. We have
\begin{align*}
    \zeta_B'(0,\mathbf{1}|(1,1)) =& \frac{\ln(2\pi)}{2}-\ln \glaisher +\frac{1}{12}, \qquad \zeta_B'(0,\mathbf{1}|(1,2)) = \frac{\ln 2}{4} + \frac{\ln \pi}{2} - \frac{\ln \glaisher}{2} + \frac{1}{24}, \\
    \zeta_B'(0,\mathbf{1}|(1,3)) =& -\frac{5\ln 3}{9} + \ln(2\pi) - \frac{\ln \glaisher}{3} - \frac{2\ln\Gamma\left(\frac{2}{3}\right)}{3} - \frac{\ln \Gamma\left(\frac{1}{3}\right)}{3} + \frac{1}{36}, \\
    \zeta_B'(0,\mathbf{1}|(2,3)) =& -\frac{11\ln 3}{18} + \frac{13\ln 2}{12} + \frac{4\ln \pi}{3} - \frac{\ln \glaisher}{6} - \frac{2\ln \Gamma\left( \frac{2}{3} \right)}{3} - \ln \Gamma\left( \frac{1}{3} \right) + \frac{1}{72}.
\end{align*}
\end{example}

\subsection{Application: A formula for \texorpdfstring{$\zeta_{\mathfrak{so}(5)}'(0)$}{zetaso5} and \texorpdfstring{$\zeta_{\mathfrak{g}_2}'(0)$}{zetag2}}
\label{subsection:values_zeta_witten_0}

We now want to prove (\ref{eq:zeta_so5_derivative_0}) and (\ref{eq:zeta_g_2_derivative_0}). We begin with the following lemma, which simplifies (\ref{eq:zeta_prime_N}) in the case where $\mathbf{N}=\mathbf{0}$.
\begin{lemma}
\label{lem:value_zeta_s_0}
We have
\begin{align*}
    & \zeta'(\mathbf{c},\mathbf{x},\underset{\bm{\theta}}{\mathbf{0}}) = \sum_{j=1}^{m} \theta_{r+j} \zeta_B'(0,\mathbf{0},\mathbf{x}|\mathbf{c}_{j\bullet}) \nonumber \\
    & + \sum_{\substack{\emptyset \neq \mathcal{A} \subseteq [\![1,r]\!] \\ 1 \leq j \leq m}} \frac{(-1)^{\beta}\theta_{r+j}}{|\bm{\theta}|_{|\mathcal{B}\cup[\![r+1,r+m]\!]}}\sum_{k_1+\cdots+k_{\alpha}=\beta} \frac{\prod_{p=1}^\alpha c_{j,a_p}^{k_p}}{\prod_{p=1}^\beta c_{j,b_p}} \sum_{i=1}^{\alpha} \theta_{a_i} \frac{\zeta'(-k_i,x_{a_i})}{k_i!} \prod_{\substack{p=1 \\ p \neq i}}^{\alpha} \frac{\zeta(-k_p,x_{a_p})}{k_p!} \\
    & + \sum_{\substack{\emptyset \neq \mathcal{A} \subseteq [\![1,r]\!] \\ 1 \leq j \leq m}} \frac{(-1)^{\beta}\theta_{r+j}}{|\bm{\theta}|_{|\mathcal{B}\cup[\![r+1,r+m]\!]}}\sum_{k_1+\cdots+k_{\alpha}=\beta} \prod_{p=1}^{\alpha} \frac{\zeta(-k_p,x_{a_p})}{k_p!} \Bigg[ \frac{\prod_{p=1}^\alpha c_{j,a_p}^{k_p}}{\prod_{p=1}^\beta c_{j,b_p}} \sum_{p=1}^{\beta} \theta_{b_p} \ln c_{j,{b_p}} \nonumber \\
    & \qquad \qquad \qquad+ \sum_{\substack{\ell=1 \\ \ell \neq j}}^{m} \theta_{r+\ell} \sum_{\substack{v_{j,1}+v_{\ell,1}=k_1 \\ \cdots \\ v_{j,\alpha}+v_{\ell,\alpha} = k_{\alpha}}} W_{\mathcal{B},j,\ell,\mathbf{0},-|\mathbf{v}_{\ell \bullet}|}(\mathbf{c}) \prod_{p=1}^{\alpha} \left( \binom{k_p}{v_{j,p}} c_{j,a_p}^{v_{j,p}}c_{\ell,a_p}^{v_{\ell,p}} \right) \Bigg].
\end{align*}
\end{lemma}

\begin{proof}
It directly follows from (\ref{eq:zeta_prime_N}), and from the two following equalities
\begin{align*}
    C^0_{j,\mathcal{A},\mathbf{k}}(\mathbf{c},\mathbf{0})=& \frac{\prod_{p=1}^\alpha c_{j,a_p}^{k_p}}{\prod_{p=1}^\beta c_{j,b_p}}, \\
    C_{j,\mathcal{A},\mathbf{k}}^1(\mathbf{c},\underset{\bm{\theta}}{\mathbf{0}}) =& \frac{\prod_{p=1}^\alpha c_{j,a_p}^{k_p}}{\prod_{p=1}^\beta c_{j,b_p}} |\bm{\theta}|_{|\mathcal{B}\cup[\![r+1,r+m]\!]} \gamma + \frac{\prod_{p=1}^\alpha c_{j,a_p}^{k_p}}{\prod_{p=1}^\beta c_{j,b_p}} \sum_{p=1}^{\beta} \theta_{b_p} \ln c_{j,{b_p}} \\
    & + \sum_{\substack{\ell=1 \\ \ell \neq j}}^{m} \theta_{r+\ell} \sum_{\substack{v_{j,1}+v_{\ell,1}=k_1 \\ \cdots \\ v_{j,\alpha}+v_{\ell,\alpha} = k_{\alpha}}} W_{\mathcal{B},j,\ell,\mathbf{0},-|\mathbf{v}_{\ell \bullet}|}(\mathbf{c}) \prod_{p=1}^{\alpha} \left( \binom{k_p}{v_{j,p}} c_{j,a_p}^{v_{j,p}}c_{\ell,a_p}^{v_{\ell,p}} \right).
\end{align*}
\end{proof}

Let 
$\mathbf{c}(\mathfrak{so}(5)):=\left(
\begin{smallmatrix}
    1 & 1\\
    1 & 2
\end{smallmatrix} 
\right)$ and 
$\mathbf{c}(\mathfrak{g}_2) := \left(
\begin{smallmatrix}
    1 & 1\\
    1 & 2\\
    1 & 3\\
    2 & 3
\end{smallmatrix}\right)$. 
By Example \ref{ex:multiple_zeta_rank_2}, we have $\zeta_{\mathfrak{so}(5)}(s)=6^s\zeta(\mathbf{c}(\mathfrak{so}(5)),\mathbf{1},(s,s,s,s))$ and $\zeta_{\mathfrak{g}_2}(s)=120^s\zeta(\mathbf{c}(\mathfrak{g}_2),\mathbf{1},(s,s,s,s,s,s))$. By Lemma \ref{lem:value_zeta_s_0}, Examples \ref{ex:derivative_values_barnes} and \ref{ex:W_constants}, we find (\ref{eq:zeta_so5_derivative_0}). Using the same type of arguments, and by using (\ref{eq:relation_prime_zeta_hurwitz_x_1}), the multiplication formula (\ref{eq:multiplication_formula}), and Euler's reflection formula, we get (\ref{eq:zeta_g_2_derivative_0}).

\section{Application: An asymptotic formula}
\label{section:asymptotic_formula}

Bridges, Brindle, Bringmann, and Franke established a more general variant of Meinardus's original theorem (see \cite{andrews1976partitions}) in \cite{bridges2024asymptotic}. We apply their results to obtain an asymptotic formula for the number of $n$-dimensional representations of the Lie algebra $\mathfrak{g}_2$.

\subsection{Statement of the Meinardus-type Theorem}
\label{subsection:meinardus_type_theorem}

Let $f: \mathbb{N} \to \mathbb{Z}_{\geq 0}$. We set for all $q=\phie^{-z} \ (z \in H_0)$, the functions
$$ G_f(z) = \sum_{n \geq 0} p_f(n) q^n := \prod_{n \geq 1} \frac{1}{(1-q^n)^{f(n)}}, \qquad L_f(s):=\sum_{n \geq 1} \frac{f(n)}{n^s}. $$
Let $\Lambda := \mathbb{Z}_{\geq 0} \setminus f^{-1}(\{ 0 \})$. We assume that:
\begin{enumerate}
    \item[(P1)] Let $\alpha>0$ be the largest pole of $L_f$. There exists an integer $L \in \mathbb{Z}_{\geq 0}$ such that, for every prime number $p$, $\left|\Lambda \setminus (p\mathbb{Z}_{\geq 0} \cap \Lambda) \right| \geq L > \frac{\alpha}{2}$.
    \item[(P2)] There exists a real number $R \in \mathbb{R}_{\geq 0}$ such that $L_f$ is meromorphic on $\overline{H}_{-R} = \{ z \in \mathbb{C}; \Re(z) \geq -R \}$, and is holomorphic on the line $(\Re(z) = -R)$. We also assume that the meromorphic function $L_f^*(s) := \Gamma(s) \zeta(s+1) L_f(s)$ has only real poles $\alpha := \gamma_1 > \ldots$, and that these poles are simple except at $s=0$, where the pole may be a double pole.
    \item[(P3)] There exists a real number $a < \frac{\pi}{2}$ such that, on each vertical strip $\sigma_1 \leq \sigma \leq \sigma_2$ contained in the domain of holomorphy of $L_f$, we have
    $$ L_f(s)\underset{|\tau| \to +\infty}{=} O_{\sigma_1,\sigma_2} \left( \phie^{a|\tau|} \right) \qquad (s=\sigma+\phii \tau).$$
\end{enumerate}

\begin{theorem}
\cite[Theorem 4.4]{bridges2024asymptotic}
\label{theorem:meinardus_type_theorem}
We assume conditions (P1), (P2), and (P3). Let $L$ be the real number from (P1) and $R$ be the real number from (P2). Moreover, we assume that $L_f$ has only two poles $\alpha > \beta > 0$ in $H_0$, and that there exists $\ell \in \mathbb{N}$ satisfying the inequality $\frac{\ell+1}{\ell}\beta < \alpha < \frac{\ell}{\ell-1}\beta$. Then we have
\begin{align*}
    p_f(n) \underset{n \to +\infty}{=}& \frac{C}{n^b} \exp \left( A_1 n^{\frac{\alpha}{\alpha+1}}+ A_2 n^{\frac{\beta}{\alpha+1}}+ \sum_{j=3}^{\ell+1} A_j n^{\frac{(j-1)\beta}{\alpha+1}+\frac{j-2}{\alpha+1}+2-j} \right) \\
    & \qquad \qquad \qquad \times \left( 1+ \sum_{j=2}^N \frac{\widetilde{B}_j}{n^{\nu_j}} + O_{L,R} \left( n^{-\min \{ \frac{2L-\alpha}{2(\alpha+1)},\frac{R}{\alpha+1} \}} \right)\right),
\end{align*}
with $(A_j)_{j \geq 3}$ explicit constants, $0<\nu_2<\cdots$ running over the positive elements of $\mathcal{N}+\mathcal{M}$ described in formulas \cite[(1.9),(1.10)]{bridges2024asymptotic}.
\end{theorem}

\subsection{Proof of Theorem \ref{th:representation_g_2}}
\label{subsection:proof_number_representation_g2}

We set $P(i,j):=\frac{ij(i+j)(i+2j)(i+3j)(2i+3j)}{5!}$. By \cite[§24.3]{humphreys1972lie}, we have that the number of $n$-dimensional representations of the exceptional Lie algebra $\mathfrak{g}_2$ is
$$ r_{\mathfrak{g}_2}(n) = \left| \left \{ (k_{i,j})_{i,j \geq 1} \in {\mathbb{Z}_{\geq 0}}^{{\mathbb{N}}^2}; \sum_{i,j \geq 1} k_{i,j} P(i,j) =n \right \} \right|. $$

The sequence $f(n) = \left| \left\{ (i,j) \in {\mathbb{N}}^2; P(i,j)=n \right\} \right|$ determines the generating function of $r_{\mathfrak{g}_2}(n)$:
$$ \prod_{n \geq 1} \frac{1}{(1-q^n)^{f(n)}} = \prod_{i,j \geq 1} \sum_{k_{i,j} \geq 1} q^{k_{i,j} P(i,j)} = \sum_{n=1}^{+\infty} r_{\mathfrak{g}_2}(n) q^n. $$
Let us define $L_f(s) = \zeta_{\mathfrak{g}_2}(s)=120^s\zeta(\mathbf{c}(\mathfrak{g}_2),\mathbf{1},(s,\ldots,s))$, and $L_f^*(s) = \zeta_{\mathfrak{g}_2}(s) \Gamma(s) \zeta(s+1)$. From \cite[Theorem 3.1]{komori2011witten}, we have that $\zeta_{\mathfrak{g}_2}(s)$ has only two poles in $H_0$, at $\alpha:=\frac{1}{3}$ and $\beta:=\frac{1}{5}$, and the integer $\ell=2$ satisfies the hypothesis of Theorem \ref{theorem:meinardus_type_theorem}. Moreover, $L_f^*(s)$ has a pole of order $2$ at $s=0$. We set $\omega_{\alpha} := \res_{s=\frac{1}{3}}\zeta_{\mathfrak{g}_2}(s)$ and $\omega_{\beta} := \res_{s=\frac{1}{5}} \zeta_{\mathfrak{g}_2}(s)$. Let $\Lambda := \mathbb{Z}_{\geq 0} \setminus f^{-1}(\{ 0 \})$. We now check the conditions (P1), (P2), and (P3) of Theorem \ref{theorem:meinardus_type_theorem}:
\begin{itemize}
\item[(P1)] For any prime number $p$, we have $\left|\Lambda \setminus (p \mathbb{Z}_{\geq 0} \cap \Lambda)\right|=+\infty$. Indeed, each sequence defined in the following
\begin{align*}
    u(2) :=& (P(8k+1,1))_{k \in \mathbb{Z}_{\geq 0}}, && u(3) := (P(9k+1,1))_{k \in \mathbb{Z}_{\geq 0}}, \\
    u(5) :=& (P(25k+1,1))_{k \in \mathbb{Z}_{\geq 0}}, && u(p) := (P(kp+1,1))_{k \in \mathbb{Z}_{\geq 0}} \quad \text{ if } p \geq 7.
\end{align*}
is strictly increasing, and for every prime $p$, we have $u(p) \subseteq \Lambda \setminus (p \mathbb{Z}_{\geq 0} \cap \Lambda)$. Therefore, any real number $L \geq \frac{1}{6}$ satisfies condition (P1).
\item[(P2)] Thanks to \cite[Theorem 3.1]{komori2011witten}, we know that the poles of $\zeta_{\mathfrak{g}_2}$ belong to the set 
$$ \mathcal{S}_{\mathfrak{g}_2} := \left\{ \frac{1}{5} \right\} \cup \left\{ \frac{k}{3}; k \in \mathbb{Z}_{\leq 1}, k \neq 0 \mod 3 \right\}. $$ 
Also, Theorem \ref{th:analytic_continuation_zeta} implies that all these poles are simple poles.
\item[(P3)] The polynomial $P$ satisfies the $\mathrm{H_0S}$ condition of the article \cite{essouabri97singularitedirichlet}. Therefore, from \cite[Theorem 3]{essouabri97singularitedirichlet}, $\zeta_{\mathfrak{g}_2}$ admits a polynomial bound on each vertical strip.
\end{itemize}
In the conditions (P1) and (P2), $L,R$ are arbitrarily large. From the definition of the sets (1.8), (1.9), and (1.10) in \cite{bridges2024asymptotic}, we find that the sequence $\nu_j$ in \cite[Theorem 4.4]{bridges2024asymptotic} corresponds to $\nu_2=\frac{1}{20}, \nu_3=\frac{2}{20}, \nu_4=\frac{3}{20}, \ldots $ Applying Theorem \ref{theorem:meinardus_type_theorem} to the integer sequence $f(n)$, we get Theorem \ref{th:representation_g_2} with the coefficients:
\begin{align*}
    &C = 2^{\frac{7}{4}} 15^{\frac{5}{12}} \pi^2 \left( \omega_{\alpha}\Gamma\left( \frac{4}{3} \right)\zeta\left( \frac{4}{3} \right)\right)^{\frac{-3}{16}}, \quad b=\frac{9}{16}, \quad K_2 = \frac{3\omega_{\beta} \Gamma\left( \frac{6}{5} \right) \zeta\left( \frac{6}{5} \right)}{4\left( \omega_{\alpha} \Gamma\left( \frac{4}{3} \right) \zeta\left( \frac{4}{3} \right)\right)^{\frac{3}{20}}}, \\
    &A_1 = 4\left( \omega_{\alpha}\Gamma\left( \frac{4}{3} \right)\zeta\left( \frac{4}{3} \right)\right)^{\frac{3}{4}}, \quad A_2 = \frac{\omega_{\beta} \Gamma\left( \frac{1}{5} \right) \zeta\left( \frac{6}{5} \right)}{\left( \omega_\alpha \Gamma\left( \frac{1}{3} \right)\zeta\left( \frac{4}{3}\right) \right)^{\frac{3}{20}}}, \\
    &A_3 = \frac{2K_2^{2}}{3\left( \omega_{\alpha} \Gamma\left( \frac{4}{3} \right) \zeta\left( \frac{4}{3} \right)\right)^{\frac{3}{4}}} - \frac{\omega_{\beta} \Gamma\left( \frac{6}{5} \right) \zeta\left( \frac{6}{5} \right)}{\left( \omega_{\alpha} \Gamma\left( \frac{4}{3} \right) \zeta\left( \frac{4}{3} \right)\right)^{\frac{9}{10}}} K_2.
\end{align*}
These coefficients were recently explicitly computed in Au's work \cite[Example 10.4]{au2025valueswitten}.

\section{Annex}
\label{section:annex}

We introduce a few elementary but technical tools that are used to derive our main results.

\begin{proposition}
\label{prop:pfd}
Let $\mathbb{K}=\mathbb{Q}(\mathbf{c})$ be the number field generated by the coefficients $c_{q,p}$. Let $\mathcal{B}=\{b_1,\ldots,b_{\beta}\} \subsetneq [\![1,r]\!]$, $j,\ell \in [\![1,m]\!]$, $n' \in \mathbb{Z}$, $\mathbf{n}=(n_1,\ldots,n_{\beta}) \in \mathbb{Z}_{\geq 0}^{\beta}$. We have the following partial fraction decomposition, in the $y$ variable,
\begin{align}
    y^{n'-1}& \times \prod_{p=1}^{\beta} \left( c_{j,b_p}+ c_{\ell,b_p} y \right)^{-n_p-1} \nonumber \\
    & = \widetilde{Q}_{\mathcal{B},j,\ell,\mathbf{n},n'}(y) + \sum_{k=1}^{n'} \frac{U_{\mathcal{B},j,\ell,\mathbf{n},n',k}}{y^k} + \sum_{p=1}^{\beta} \sum_{k=1}^{|\mathbf{n}|+\beta} \frac{V_{\mathcal{B},j,\ell,\mathbf{n},n',p,k}}{(c_{j,b_p}+ c_{\ell,b_p} y)^{k}} \label{eq:partial_fraction_decomposition}
\end{align}
where $U_{\mathcal{B},j,\ell,\mathbf{n},n',k}, V_{\mathcal{B},j,\ell,\mathbf{n},n',p,k} \in \mathbb{K}$, and where $\widetilde{Q}_{\mathcal{B},j,\ell,\mathbf{n},n'}(y) \in \mathbb{K}[y]$ is a polynomial.
\end{proposition}

\begin{notation}
Let $Q_{\mathcal{B},j,\ell,\mathbf{n},n'}(y)$ denote the antiderivative of the polynomial $\widetilde{Q}_{\mathcal{B},j,\ell,\mathbf{n},n'}(y)$, normalized so that it vanishes at $y=1$. Let
\begin{align}
    &W_{\mathcal{B},j,\ell,\mathbf{n},n'}(\mathbf{c}):= -Q_{\mathcal{B},j,\ell,\mathbf{n},n'}(0) - \sum_{k=2}^{n'} \frac{U_{\mathcal{B},j,\ell,\mathbf{n},n',k}}{k-1} \label{eq:def_coefficient_W} \\
    & \quad - \frac{1}{c_{\ell,b_p}} \sum_{p=1}^{\beta} \sum_{k=2}^{|\mathbf{n}|+\beta} \frac{V_{\mathcal{B},j,\ell,\mathbf{n},n',p,k}}{k-1} \left( (c_{j,b_p}+ c_{\ell,b_p})^{1-k} - c_{j,b_p}^{1-k} \right) \nonumber \\
    & \quad + \frac{1}{c_{\ell,b_p}} \sum_{p=1}^{\beta} V_{\mathcal{B},j,\ell,\mathbf{n},n',p,1} \ln \left( 1+\frac{c_{\ell,b_p}}{c_{j,b_p}} \right). \nonumber
\end{align}
\end{notation}

\begin{example}
\label{ex:W_constants}
Let $j \neq \ell$, then
\begin{align*}
    & W_{\{p\},j,\ell,0,0} = -\frac{1}{c_{j,p}} \ln\left( 1+\frac{c_{\ell,b_p}}{c_{j,b_p}} \right), \qquad W_{\{p\},j,\ell,0,1}(\mathbf{c})=\frac{1}{c_{\ell,b_p}} \ln\left( 1+\frac{c_{\ell,b_p}}{c_{j,b_p}} \right), \\
    & W_{\emptyset,j,\ell,\emptyset,0}(\mathbf{c})=0.
\end{align*}
\end{example}

\begin{lemma}
\label{lem:integral_computations_epsilon_1}
There exists a function with no constant term $G^*_{\mathcal{B},j,\ell,\mathbf{n},n'}(\varepsilon)=\mu \ln \varepsilon + \sum_{n \geq i} \lambda_n \varepsilon^n$ defined for all $0 < \varepsilon \ll 1$ with $i \in \mathbb{Z}$, and such that
\begin{align*}
    & \int_{\varepsilon}^1 y^{-1-n'} \times \prod_{p=1}^{\beta} \left( c_{j,b_p}+c_{\ell,b_p} y \right)^{-1-n_p} \dd y = G^*_{\mathcal{B},j,\ell,\mathbf{n},n'}(\varepsilon) + W_{\mathcal{B},j,\ell,\mathbf{n},n'}(\mathbf{c}),
\end{align*}
where $W_{\mathcal{B},j,\ell,\mathbf{n},n'}(\mathbf{c})$ is the constant defined in (\ref{eq:def_coefficient_W}).
\end{lemma}

\begin{proof}
Plugging (\ref{eq:partial_fraction_decomposition}) into the integrand, we get
\begin{align*}
    & \int_{\varepsilon}^1 y^{-1-n'} \times \prod_{p=1}^{\beta} \left( c_{j,b_p}+c_{\ell,b_p} y \right)^{-1-n_p} \dd y = - U_{\mathcal{B},j,\ell,\mathbf{n},n',1} \ln \varepsilon \\
    & - Q_{\mathcal{B},j,\ell,\mathbf{n},n'}(\varepsilon)-\sum_{k=2}^{n'} \frac{U_{\mathcal{B},j,\ell,\mathbf{n},n',k}(1-\varepsilon^{1-k})}{k-1} \\
    & + \frac{1}{c_{\ell,b_p}}\sum_{p=1}^{\beta} V_{\mathcal{B},j,\ell,\mathbf{n},n',p,1} \left( \ln \left( 1+\frac{c_{\ell,b_p}}{c_{j,b_p}}\right) - \ln \left( 1+\frac{c_{\ell,b_p}}{c_{j,b_p}} \varepsilon \right) \right) \\
    & - \frac{1}{c_{\ell,b_p}} \sum_{p=1}^{\beta} \sum_{k=2}^{|\mathbf{n}|+\beta} \frac{V_{\mathcal{B},j,\ell,\mathbf{n},n',p,k}}{k-1} \left( (c_{j,b_p}+ c_{\ell,b_p})^{1-k} - (c_{j,b_p}+ c_{\ell,b_p} \varepsilon)^{1-k} \right).
\end{align*}
We conclude via a Taylor expansion of $(c_{j,b_p}+ c_{\ell,b_p} \varepsilon)^{1-k}$ and $\ln \left( 1+\frac{c_{\ell,b_p}}{c_{j,b_p}} \varepsilon \right)$ near $\varepsilon=0$.
\end{proof}

\bibliography{./Bibliography.bib}
\bibliographystyle{alpha}
\end{document}